\documentclass[11pt, reqno]{amsart}
\usepackage[utf8]{inputenc}
\usepackage{amsfonts}
\usepackage{hyperref}

\hypersetup{}
\usepackage{amsmath}
\usepackage{setspace}
\usepackage{amsthm}
\usepackage{pdflscape}
\usepackage{pgfplots}
\usepackage{mathrsfs}
\usepackage{mathtools}
\usepackage[top=2.5cm, bottom=2.5cm, left=2.2cm, right=2.2cm]{geometry}
\usepackage{amssymb,bbm}
\usepackage{shuffle}
\usepackage{enumerate}
\usepackage{enumitem}
\usepackage[initials]{amsrefs}
\usepackage{stmaryrd}
\usepackage{appendix}
\usepackage{tikz}
\usepackage{tikz-cd}

\pgfplotsset{compat = 1.18}

\usepackage[nodayofweek]{datetime}

\usepackage{verbatim}

\usepackage{graphicx}
\usepackage{matlab-prettifier}

\numberwithin{equation}{section}
\numberwithin{figure}{section}

\theoremstyle{plain}
\newtheorem{theorem}{Theorem}[section]
\newtheorem{lemma}[theorem]{Lemma}
\newtheorem{proposition}[theorem]{Proposition}
\newtheorem{corollary}[theorem]{Corollary}

\theoremstyle{definition}

\newtheorem{remark}[theorem]{Remark}

\newtheorem{assumption}{Assumption}

\allowdisplaybreaks

\renewcommand{\baselinestretch}{1.01}
\newcommand{\bbD}{\mathbb{D}}
\newcommand{\bbE}{\mathbb{E}}

\newcommand{\bbL}{\mathbb{L}}

\newcommand{\bbN}{\mathbb{N}}

\newcommand{\bbP}{\mathbb{P}}
\newcommand{\bbQ}{\mathbb{Q}}
\newcommand{\R}{\mathbb{R}} 

\newcommand{\bbW}{\mathbb{W}}

\newcommand{\mA}{\mathcal{A}}

\newcommand{\mC}{\mathcal{C}}

\newcommand{\mF}{\mathcal{F}}

\newcommand{\mH}{\mathcal{H}}

\newcommand{\mN}{\mathcal{N}}

\newcommand{\mP}{\mathcal{P}}

\newcommand{\mS}{\mathcal{S}}

\newcommand{\mX}{\mathcal{X}}

\newcommand{\bfG}{\mathbf{G}}

\newcommand{\bfJ}{\mathbf{J}}

\newcommand{\sfE}{\mathsf{E}}

\newcommand{\sfN}{\mathsf{N}}

\newcommand{\brq}{\bar{q}}

\newcommand{\frF}{\mathfrak{F}}

\newcommand{\frT}{\mathfrak{T}}

\newcommand{\fra}{\mathfrak{a}}
\newcommand{\frb}{\mathfrak{b}}

\newcommand{\frf}{\mathfrak{f}}

\newcommand{\fro}{\mathfrak{o}}

\newcommand{\frr}{\mathfrak{r}}
\newcommand{\frs}{\mathfrak{s}}

\newcommand{\fru}{\mathfrak{u}}

\newcommand{\tlQ}{{\tilde{Q}}}

\newcommand{\tlX}{{\tilde{X}}}

\newcommand{\tlZ}{{\tilde{Z}}}

\newcommand{\scrP}{\mathscr{P}}

\newcommand{\eps}{\varepsilon}

\newcommand{\op}{{\text{\normalfont op}}}

\title[Quantitative propagation of chaos and universality for Langevin spin glass]{Quantitative propagation of chaos and universality for 
asymmetric Langevin spin glass dynamics}

\date{\today}
\thanks{ M.A. and K.H. would like to thank Dan Lacker for helpful suggestions on the presentation. K.H. was supported by the Simons Foundation as a Junior
 Fellow of the Simons Society of Fellows.}

\author{Manuel Arnese}
\address{Manuel Arnese: Department of Industrial Engineering and Operations Research, Columbia University}
\email{\href{mailto:ma4339@columbia.edu}{ma4339@columbia.edu}}

\author{Kevin Hu}
\address{Kevin Hu: Department of Industrial Engineering and Operations Research, Columbia University}
\email{\href{mailto:kjh2189@columbia.edu}{kjh2189@columbia.edu}}

\begin{document}
\begin{abstract}
We obtain quantitative estimates on quenched propagation of chaos for Langevin spin glass dynamics with i.i.d. disorder. Prior work in the case of Gaussian disorder established the qualitative convergence of the law of a single spin to a deterministic McKean--Vlasov limit. We prove convergence rates in expected Wasserstein distance and quantitative concentration rates for Lipschitz observables under the assumption that the disorder satisfies the T$_2$ inequality. The proof uses a coupling argument, together with techniques from concentration of measure, filtering theory, and Malliavin calculus.
\end{abstract}

\maketitle


 {
\setcounter{tocdepth}{1}
\renewcommand{\baselinestretch}{0.5}\selectfont
\tableofcontents 
}

\section{Introduction}

 We study the asymmetric Langevin spin glass model, which is given by the following stochastic differential equation (SDE) with random interactions:
\begin{align*}
\begin{split}
    dX_t^{i} &= - U'(X_t^i) dt + \frac{\beta}{\sqrt{N}} \sum_{j = 1}^N J_{ij} X_t^j dt + dW_t^i, \quad i = 1, \ldots, N.
\end{split}
\end{align*}
Here, $U: (-A,A) \rightarrow \R$ is an even potential satisfying $U(x) \rightarrow +\infty$ as $|x| \rightarrow A$ for some $A>0$, $\beta > 0$ is the inverse temperature,  $W^1, \ldots, W^N$ are independent 1-dimensional Brownian motions. An important parameter is the \emph{disorder} $\bfJ=(J_{ij})_{i,j=1}^N$, which is a random matrix whose entries are independent and identically distributed (i.i.d.) with mean $0$ and unit variance. We discuss the history and motivation of the model in Section \ref{ss: intro spin glas dynamics}.

The first set of rigorous mathematical results for this model were obtained in \cite{Arous1995-zj} in the setting where $\bfJ = \bfG$ is a matrix of independent standard Gaussian random variables and the initial condition $(X_0^1, \ldots, X^N_0)$ is i.i.d. and symmetric. By means of a large deviations principle, they establish convergence of the empirical measure; that is, there exists a deterministic probability measure $\mu$ on the path space $\mC_T = C([0, T];(-A, A))$ such that
\begin{equation}\hat{\mu}^N_{[T]} := \frac{1}{N}\sum_{i = 1}^N \delta_{X^i_{[T]}} \rightarrow \mu_{[T]},
\label{eq: intro-poc-1}    
\end{equation}
where $X^i_{[t]} = (X^i_s: s \in [0, t])$. This convergence is known as a \textit{mean field limit}. Since $\text{Law}(X^1, \ldots, X^N)$ is exchangeable, it is well known (see e.g., \cite[Proposition 2.2]{snitzmann}), that \eqref{eq: intro-poc-1} is equivalent to the asymptotic independence of fixed marginal distributions, which is known as \textit{propagation of chaos}: for any $k \in \bbN$, we have
\begin{equation}
\lim_{N \rightarrow \infty} \text{Law}(X^1_{[T]}, \ldots, X^k_{[T]}) = \mu^{\otimes k}_{[T]},
\label{eq: intro-poc-2}
\end{equation}
where the convergence occurs in the weak topology. The properties \eqref{eq: intro-poc-1}-\eqref{eq: intro-poc-2} were first studied in the context of mean-field interacting diffusions by McKean in \cite{McKean1966}. We refer the reader to Section \ref{ss: chaos review} for more context on mean-field limits and propagation of chaos.

{The limit law $\mu$ may be realized as the law of an It\^{o} process $Y$, which satisfies
\begin{align*}
    dY_t&=-U'(Y_t)\,dt+\beta G_t\,dt+dW_t,\quad
    Y_0\sim \mu_0,
\end{align*}
where $G=(G_t)_{t\geq 0}$ is a Gaussian process independent of the Brownian motion $W$ such that $\bbE[G_tG_s]=\bbE[Y_tY_s]$.} In turn it can be shown that $Y$ solves a non-Markovian SDE. More precisely, there exists a family of kernels $(H^t_\mu)_{t \in [0, T]}$ such that $\mu = \text{Law}(Y)$ is the weak solution to the following stochastic differential equation (SDE):
\begin{equation*}
\begin{aligned}
    dY_t &= - U'(Y_t) dt + dB_t,\\
    dB_t &= \beta^2 \bigg[\int_0^t H^t_\mu(t, s) dB_s \bigg] dt + dW_t,
    \\ Y_0 &\sim \mu_0,
\end{aligned}
\end{equation*}
where $W$ is a standard Brownian motion.
See Section \ref{ss: dynamics} for a precise definition of the limit process. Subsequent work has extended this convergence to a class of non-Gaussian disorder \cite{Dembo-Lubetzky-Zeitouni-2021-Universality}. 

In \cite[Corollary 2.11]{Guionnet-1997-PoC}, a refinement of \eqref{eq: intro-poc-2} known as \emph{quenched propagation of chaos} was established: for all $k \in \bbN$ and bounded continuous functions $(f_1, \ldots, f_k)$ on $\mC_T$, we have
\begin{equation}
    \label{eq: intro-quenched-poc}
    \bbE\bigg[\prod_{i = 1}^kf_i\big(X^i_{[T]}\big) \,\Big|\,\bfG\bigg] \rightarrow \prod_{i = 1}^k \int_{\mC_T} f_i(x)\,d\mu(x) \quad \text{ in probability as $N \rightarrow \infty$.}
\end{equation}
It is important to highlight that when conditioned on $\bfG$, the random variables $(X^1, \ldots, X^n)$ are no longer exchangeable, and thus \eqref{eq: intro-quenched-poc} is no longer an immediate consequence of the mean-field limit \eqref{eq: intro-poc-1}.

The purpose of this paper is to establish a quantitative version of \eqref{eq: intro-quenched-poc} in Theorem \ref{th: main theorem} for more general disorder $\bfJ$. We prove that for all $k, N \in \bbN$ and functions $f:\mC_T^k \rightarrow \R$ that are $1$-Lipschitz with respect to the $L^2$ norm, we have
    \begin{equation}
\bbP\bigg(\,\bigg| \bbE\big[f\big(X^1_{[t]}, \ldots, X^k_{[t]}\big)\mid\bfJ\big] - \int_{\mC_t^k} f(x^1, \ldots, x^k)\,d\mu^{\otimes k}_{[t]}(x) \bigg|\geq r \Big) \leq c_0e^{-c_1 r^2\frac{N}{k}},
\label{eq: intro-main-result-1}
\end{equation}
where $c_0, c_1 \in (0, \infty)$ are constants independent of $k, N$. Our results hold under a relatively large class of non-Gaussian disorder, as we only require the entries to satisfy a T$_2$ inequality. Similar assumptions have appeared elsewhere in random matrix theory, see for instance \cite{guionnet2000concentration}. To the best of our knowledge, this is the first time quenched propagation of chaos  has been established for any non-Gaussian disorder, as the lack of exchangeability prevents one from deducing \eqref{eq: intro-quenched-poc} from the results of \cite{Dembo-Lubetzky-Zeitouni-2021-Universality}. Furthermore, we relax the assumptions on $X_0$ from i.i.d. to \emph{chaotic} initial conditions, which is canonical in the study of propagation of chaos. We state our precise framework in Assumption \ref{assumption: main}. 

As a consequence of \eqref{eq: intro-main-result-1}, we obtain a quantitative rate of convergence in Wasserstein distance of the time-marginal law:
\begin{equation}\label{eq: intro k=1 quenched}
        \bbE\big[\bbW_1\big( \text{Law}(X^1_t\mid\bfJ), \mu_t\big)\big] = O\big(N^{-1/2}\big).
\end{equation}
We provide evidence in Section \ref{sec: optimality} that this is the optimal rate of convergence by studying a simplified model with $U=0$, where we prove lower bounds in total variation. This stands in contrast to the classic mean-field setting, where the sharp rate of convergence is $N^{-1}$, see \cite[Section 3]{HierarchiesPaper}. The presence of disorder appears to have a meaningful influence on propagation of chaos, and we refer the reader to Section \ref{ss: non-exch} for further discussion.

More generally, we obtain a rate of convergence for any set of $k$-spins for $k \geq 2$:
\begin{equation}
\label{eq: intro-k-quenched-poc}
    \bbE\Big[\bbW_1\big( \text{Law}(X^1_t, \ldots, X^k_t\mid\bfJ), \mu^{\otimes k}_t\big)\Big] = O_{\log}\big(N^{-1/k}\big).
\end{equation}
where the $O_{\log}$ notation hides $\log N$ factors. We expect this result to be sharp for $k = 2$ (up to log factors) and $k = 1$ in \eqref{eq: intro k=1 quenched}. However, the bound deteriorates dramatically when $k$ becomes large and we believe it to be suboptimal. On the other hand, we establish in Theorem \ref{th: averaged propagation of chaos} a significantly improved rate of convergence for averaged propagation of chaos \eqref{eq: intro-poc-2}:
\begin{equation}
\label{eq: intro-k-avg-poc}
    \bbW_1\big( \text{Law}(X^1_t, \ldots, X^k_t), \mu^{\otimes k}_t\big) = O\big(\sqrt{k/N}\big).
\end{equation}

The rest of the introduction is structured as follows. In Section \ref{ss: model}, we state precisely our model and assumptions (Assumption \ref{assumption: main}). Section \ref{sec: main results} contains our main result (Theorem \ref{th: main theorem}), as well as a sketch of the proof. In Section \ref{ss: chaos review} and Section \ref{ss: intro spin glas dynamics}, we review related works on propagation of chaos and spin glass dynamics, respectively. In Section \ref{ss: open problems} we discuss some open problems. Finally, Section \ref{ss: outline} provides an outline of the remainder of the paper.

\subsection{Model}
\label{ss: model}
\subsubsection{Notation} For a Polish space $\mX$, let $\mP(\mX)$ denote the space of Borel probability measures on $\mX$. Sometimes we will write $X \sim \nu$ for $\nu = \text{Law}(X)$. We use $\| \cdot\|_{\text{TV}}$ to denote the usual total variation distance.

We fix $A>0$ and an integer $m\in \bbN$. We define the space of continuous paths from $[0,T]$ taking values in $[-A,A]^m$ by
\[\mC_T^{m}:=\left\{f:[0,T]\to [-A,A]^m \mid f \text{ is continuous}\right\}.\]
For a probability measure $\nu \in \mP(\mC_T^m)$, we let $\nu_{[t]} \in \mP(\mC_t^m)$ denote the restriction of $\nu$ to $\mC_t^m$. Similarly, we write $\nu_t$ for the time-marginal law of $\nu$ at $t \in [0, T]$. We use $L^2([0, T];\R^d)$ to denote the space of square integrable functions from $[0, T] \rightarrow \R^d$. If $d = 1$ we simply write $L^2([0, T])$. 

For $p \in [1, \infty)$ and a Polish space $(\mX,\text d)$, the \emph{Wasserstein}-$p$ distance is given by
\begin{equation}
    \bbW_p^p(\alpha, \beta) = \inf_\pi\int_{\mX \times \mX}\text d(x,y)^p\,d\pi(x, y), \quad \alpha, \beta \in \mP(\mX),
    \label{eq: Wasserstein-def}
\end{equation}
where the infimum is taken over the set of couplings of $\alpha, \beta$. We also recall the useful dual characterization of $\bbW_1$,
\[ \bbW_1(\alpha, \beta) = \sup_{\substack{f:\mX \rightarrow \R \\ \text{Lip}(f) \leq 1}}\bigg[ \int f(x) \,d\alpha(x) - \int f(x)\,d\beta(x) \bigg],\quad \alpha, \beta \in \mP(\mX).\]
When we apply the Wasserstein distance on measures on path space $\alpha,\beta\in \mP(\mC_T^k)$, we will simply denote by $\bbW_p$ the above notation where we use, for $x,y \in \mC_T^k$, $d(x,y)=\sup_{t\leq T}|x_t-y_t|$. We will also need to use the Wasserstein distance on $L^2([0,T];[-A,A]^k)$:
\[\bbW_{p,L}^p(\alpha,\beta)= \inf_{\pi} \int \bigg|\int_0^T|x_t-y_t|^2\,dt\bigg|^{p/2}\,d\pi(x,y),\]
where once again the infimum is taken over the set of all couplings. 

We use $\| \cdot\|_{\text{Fr}}$ and $\| \cdot \|_{\op}$ to denote the Frobenius norm and (Euclidean) operator norm, respectively.
\subsubsection{Assumptions} Throughout the rest of the paper, we will assume the following holds.
\begin{assumption} Fix $N \in \bbN$, $\beta \in \R$, and $A, T \in (0, \infty)$. Suppose there exist constants $C_{\textup{P}}, C_{\textup{Lip}}, C_0, \sigma \in (0, \infty)$ such that $(U, P_0^N, \mu_0, \bfJ)$ satisfy the following.
\label{assumption: main}
    \begin{enumerate}
        \item The function $U:(-A,A)\to \R$ is twice differentiable and even. Moreover we have 
        \[U = U_c + U_\ell,\] where $U_c:(-A,A)\to \R$ is a convex function and $U_\ell:\R\to \R$ is a differentiable function such that $\nabla U_\ell$ is $C_{\textup{Lip}}$-Lipschitz. Furthermore $U$ satisfies
        \begin{equation*}
            \lim_{z \uparrow A }\int_0^z e^{2U(x)} \int_0^x e^{-2U(y)}\,dy\,dx=\infty.
        \end{equation*}
        \item The initial distribution $P^N_0$ is exchangeable, symmetric around the origin, supported on $(-A,A)^N$ and satisfies a Poincaré inequality with constant $C_{\textup{P}}<\infty$. That is, for every differentiable $f:\R^N\to \R$, we have
        \begin{align*}
            \int_{\R^N} |f(x)|^2\,dP_0^N(x)-\bigg|\int_{\R^N} f(x)\,dP_0^N(x)\bigg|^2\leq C_\textup{P} \int_{\R^n} |\nabla f(x)|^2\,dP_0^N(x).
        \end{align*}
        Moreover, $\mu_0$ is supported on $(-A, A)$ and symmetric around the origin. We also assume that $P_0^N$ is \emph{$\mu_0$-chaotic} in the sense that
        \[\bbW_2(P_0^N,\mu_0^{\otimes N})\leq C_0.\]
        \item The entries of $\bfJ \in \R^{N \times N}$  are i.i.d. with mean $0$ and variance $1$. Moreover $\eta:=\textup{Law}(J_{11})$ satisfies a T$_2$ inequality with constant $\sigma^2$, which implies for every $m \in \bbN$ and every $1$-Lipschitz function $f:\R^m\to \R$, we have
\[\log \int_{\R^m} e^{\lambda f(x)}\,d\eta^{\otimes m}(x)\leq \lambda \int_{\R^m}  f(x)\,d\eta^{\otimes m}(x)+\frac{\lambda^2\sigma^2 }{2}.\]

    \end{enumerate}
\end{assumption}

\begin{remark}[About the assumptions] We comment briefly on our assumptions.
\begin{enumerate}
\item
Assumption \ref{assumption: main} is in one sense slightly stronger than those appearing in \cite{Arous1995-zj} or Section 2 of \cite{Guionnet-1997-PoC}, as we impose the additional condition that $P_0^N$ satisfies a Poincaré inequality. However, our setting allows for chaotic rather than independent initial conditions as well as non-Gaussian disorder.
    \item 
The usual example of a confining potential $U$ that satisfies Assumption \ref{assumption: main} is \[U(x)=-\log(A^2-x^2),\] which is strongly convex. Assumption \ref{assumption: main} implies that the SDE
\[dX_t^i=-U'(X_t^i)\,dt+dW_t^i, \quad i = 1, \ldots, N,\]
has a unique strong solution which satisfies $ \sup_{t \geq 0} |X_t^i|<A$ almost surely, see \cite[Lemma 2.2]{Arous1995-zj}.
\item {The standard formulation of the T$_2$ inequality states that for every $\rho\in\mP(\R)$,
\[\bbW_2^2(\rho,\eta)\leq 2\sigma^2 \mH(\rho\,\|\,\eta)\]
where $\mH$ is relative entropy. Nevertheless these two formulations are equivalent by Gozlan's theorem, see e.g., \cite[Theorem 4.31]{VanHandelHDP}.}
\item  The condition \ref{assumption: main}.3 holds for a large class of distributions. In particular this holds for centered laws $\eta$ with unit variance and with density $\eta(dx)\propto e^{-\phi(x)}\,dx$ where $\phi:\R\to \R$ is strongly convex, see \cite[Section 2.1]{chewi2023log}. Another class of examples are measures that satisfy a log-Sobolev inequality, such as the uniform or centered beta distributions. {See \cite{GozlanCharacterizationT2} for a complete characterization of the T$_2$ inequality for non-atomic measures over $\R$.}

\end{enumerate}
    
\end{remark}

\subsubsection{The dynamics}\label{ss: dynamics} Suppose $(U, P_0^N, \mu_0, \bfJ)$ satisfies Assumption \ref{assumption: main}. We study the following dynamics:
\begin{align}
\begin{split}
    dX_t^{i} &= - U'(X_t^i) dt + \frac{\beta}{\sqrt{N}} \sum_{j = 1}^N J_{ij} X_t^j dt + dW_t^i, \quad i = 1, \ldots, N, \quad \text{Law}(X_0) = P_0^N.
\end{split}
\label{eq: intro quenched dynamics}
\end{align}
By \cite[Proposition 2.1]{Arous1995-zj}, this system admits a weak solution and thus there exists a probability space $(\Omega, \mF, \bbP)$ supporting random variables $(X, W, \bfJ)$ which satisfy the above equation.

For $v \subset \{1, \ldots, N\}$, we use the following to denote the law of $X$ conditioned on $\bfJ$,
\begin{equation*}
    \mP(\mC_T^{|v|})\ni P^{N,v}(\bfJ) := \text{Law}\big(X^v_{[T]}\mid\bfJ\big),
\end{equation*}
In words, $P^{N,v}(\bfJ)$ is the distribution of the particles in the subset $v$, quenched on the disorder $\bfJ$. As such, we refer to $P^{N, v}(\bfJ)$ as the \emph{quenched law}. The tower property implies that
\[ \text{Law}\big(X^v_{[T]}\big) = \bbE\big[ P^{N,v}_{[T]}(\bfJ) \big].\]

In the following, we use $\mu \in \mP(\mC_T)$ to denote the probability measure identified in \cite[Theorem 2.4]{Arous1995-zj}, which was referred to as $Q$ therein. We recall these limit dynamics as presented there, and we supply a more detailed discussion on its form in Section \ref{ss: form of limit dynamics}. We provide another formulation which will be more useful for our purposes in Section \ref{sect: representation}.

Fix $\nu \in \mP(\mC_T)$ with covariance $K_\nu(t, s) := \int_{\mC_T}x_t x_s d\nu(x)$. For each $t \in [0, T]$, we define a corresponding integral operator $K_\nu^t: L^2([0, t]) \rightarrow L^2([0, t])$ by
\[K_\nu^t f(s) = \int_0^t K_\nu(s, u) f(u) du, \quad s \in [0, t], \]
for $f \in L^2([0, t])$. The operator $K^t_\mu$ is positive and compact. Since $K$ is positive definite, the following operator $H^\mu_t: L^2([0, t]) \rightarrow L^2([0, t])$ is well-defined:
\begin{equation}
  H_\nu^t = K_\nu^t\big[\beta^2 K_\nu^t + I\big]^{-1}  .
  \label{eq: H_mu^t}
\end{equation}
We can see that $H_\nu^t$ is a Hilbert-Schmidt operator, and therefore there exists a kernel $H_\nu^t \in L^2([0, t]^2)$ such that
\[ H_\nu^t f(s) = \int_0^t H_\nu^t(s, u) f(u) du, \quad s \in [0, t],\]
for all $f \in L^2([0, t])$. 
\begin{lemma}[{\cite[Theorem 2.5]{Arous1995-zj}}] There exists a unique weak solution $(Y,B)$ with $\mu=\textup{Law}(Y)$ to the following McKean-Vlasov equation,
\begin{equation}
    \label{eq: limit dynamics 0}
    \begin{aligned}
        dY_t &= - U'(Y_s) ds + dB_t, \\
        dB_t &=  \beta^2  \bigg[ \int_0^t H^t_\mu(t, s) dB_s\bigg] dt + dW_t,   \\
        Y_0 &\sim \mu_0.
    \end{aligned}
\end{equation}
where $W$ is a standard Brownian motion.
\end{lemma}
\subsection{Main results}

Throughout the rest of the paper, assume that Assumption \ref{assumption: main} holds. Moreover, unless explicitly stated, all constants may depend on \[(A, \beta, T, C_{\textup{P}}, C_{\textup{Lip}}, C_0,\sigma^2),\] but are crucially independent of $N$. In particular, $C$ will be used to denote a finite positive constant which depends on $(A, \beta, T, C_{\textup{P}}, C_{\textup{Lip}}, C_0,\sigma^2)$ and may change from line to line. Let us note that the constants will depend on the time horizon $T$, thus precluding uniform in time results.

Our main result is the following. 
\begin{theorem}[Quantitative propagation of chaos] \label{sec: main results}

\label{th: main theorem} Fix $k, N \in \bbN$. Let $v \subset \{1, \ldots, N\}$ satisfy $|v| = k$. There exist constants $c_0, c_1 \in (0, \infty)$ such that for every $t\leq T, r\geq 0$ and  every $f:\mC_t^k\to \R$ that satisfies one of the following:
\begin{enumerate}
    \item $f$ is 1-Lipschitz in the $L^2([0, t]; \R^k)$ norm, that is
    \[|f(x)-f(y)|^2\leq \int_0^t |x_s-y_s|^2\,ds, \quad x, y \in \mC^k_t.\] 
    \item There exists $s \in [0, t]$ such that $f(x) = f_s(x_s)$ for a 1-Lipschitz function $f_s:\R^k \rightarrow \R$.
\end{enumerate}
Then we have
\begin{equation}
\label{eq: concentration equation thm}
\bbP\bigg(\,\bigg| \int f(x) \,d\Big( P^{N, v}_{[t]}(\bfJ) - \mu^{\otimes k}_{[t]}\bigg)(x) \bigg|\geq r \bigg) \leq c_0e^{-c_1 r^2\frac{N}{k}}.
\end{equation}
Moreover there exists a constant $C \in (0, \infty)$ such that for all $t \in [0, T]$, we have
\[\bbE\Big[\bbW_1\big(P_{t}^{N,v}(\bfJ), \mu^{\otimes k}_t\big)\Big] \leq \left \{ \begin{aligned}
    & \frac{C}{\sqrt N}, &\quad & k =1\\
    & \frac{C\log N}{\sqrt N}, &\quad & k =2 \\
    &  \frac{Ck}{N^{1/k}}
    , &\quad &k \geq 3.
\end{aligned} \right. \]
\end{theorem}
\begin{remark}[Generalizations]
    The above theorem can be strengthened by slightly changing the assumptions. 
    \begin{enumerate}
        \item If $\bfJ$ has Gaussian entries, then the concentration estimate \eqref{eq: concentration equation thm} holds also for functions $f$ that are 1-Lipschitz in the sup norm $|f(x)-f(y)|\leq \sup_{t}|x_t-y_t|$. The obstruction for non-Gaussian disorder is purely technical and is caused by the fact that the path space $\bbW_2^2$ distance is not sub-additive, as opposed to its time marginal and $L^2$ versions.
        \item Note that assuming the disorder $\bfJ$ is sub-Gaussian does not imply a T$_2$ inequality. Nonetheless, in this setting our arguments can still establish averaged universality and propagation of chaos, namely
        \begin{align*}
    \bbW_2\Big(\bbE\big[P_{t}^{N,v}(\bfJ)\big],\mu_t^{\otimes k}\Big)+\bbW_{2,L}\Big(\bbE\big[P_{[t]}^{N,v}(\bfJ)\big],\mu_{[t]}^{\otimes k}\Big)=O\big(\sqrt{k/N}\big).
        \end{align*}
        This is the case for many discrete distributions: for example if the entries $J_{ij}$ are Rademacher with $\bbP(J_{ij}=1)=\bbP(J_{ij}=-1)=1/2$.
    \end{enumerate}
\end{remark}

\begin{remark}[Mean-field limits]As a consequence of Theorem \ref{th: main theorem}, we also obtain a quantitative results on the quenched mean-field limit \eqref{eq: intro-poc-1}. In particular, for all $t \in [0, T]$ and for any $g: \mC_T \rightarrow \R$ which satisfies the conditions of Theorem \ref{th: main theorem}, we may apply \eqref{eq: concentration equation thm} with $k = N$, $r = \delta \sqrt{N}$, and \[f(x)= N^{-1/2}\big(g(x^1) + \cdots g(x^N)\big), \quad x \in \mC_T^N,\] to obtain
\[ \bbP\Big(\big| \bbE[\langle g, \hat{\mu}^N_{[t]} \rangle\mid\bfJ] - \bbE[g(Y_{[t]})] \big| \geq \delta \Big) \leq c_0 e^{-c\delta^2 N}. \]
 where $\langle g, \hat{\mu}^N_{[t]} \rangle$ denotes integration of $g$ against the empirical measure. 
\end{remark} 

The proof of Theorem \ref{th: main theorem} is divided into three steps: concentration of measure, universality, and propagation of chaos for the averaged measure $\bbE[P^N(\bfJ)]$. In the remainder of this subsection, we present these results and also sketches of the proof.

\subsubsection{Concentration of measure}
First, we establish that the quenched law concentrates on its mean, which we refer to as the \emph{averaged law}.
\begin{theorem}[Concentration]
\label{th: concentration}
Let $v\subseteq \{1,\dots,N\}$ with $|v|=k$. The following hold:
\begin{enumerate}
    \item There exist constants $C,\theta \in (0, \infty)$ such that for every function $f:\mC_T^k\to \R$ that is $1$-Lipschitz with respect to the supremum norm,
\[\bbE\bigg[\exp\bigg(\frac{\theta N}{k}\bigg|\int_{\mC^k_T}f(x)\,d\big(P^{N,v}_{[T]}(\bfJ)-\bbE\big[P^{N,v}_{[T]}(\bfJ)\big]\big)(x)
\bigg|^2\bigg)\bigg]\leq C.\]
\item There exists a constant $C \in (0, \infty)$ such that 
\[\bbE\Big[\bbW_1\Big(P_{t}^{N,v}(\bfJ), \bbE\big[P_{t}^{N,v}(\bfJ)\big]\Big)\Big] \leq \left \{ \begin{aligned}
    & \frac{C}{\sqrt N}, &\quad & k =1,\\
    & \frac{C\log N}{\sqrt N}, &\quad & k =2, \\
    &  \frac{Ck}{N^{1/k}}
    , &\quad &k \geq 3.
\end{aligned} \right. \]
\end{enumerate}
\label{th: quenched propagation of chaos}

\end{theorem}

Theorem \ref{th: quenched propagation of chaos} is proved in Section \ref{sect: concentration}. The main idea is to establish a suitable regularity property of the map $\bfJ \mapsto P^{N, v}_{[t]}(\bfJ)$ via an entropy-Girsanov argument. We establish that for every bounded function $f:\mC_t^{k}\to \R$,
\begin{align*}
    \int_{\mC_t^k} f(x)d\big(P^{N,v}_{[t]}(\bfJ)-P^{N,v}_{[t]}(\bfJ')\big)(x)\leq \frac{C\|f\|_\infty}{\sqrt N}\exp\big(CN^{-1/2}\|\bfJ\|_{\text{op}}\big)\|\bfJ-\bfJ'\|_{\text{Fr}}.
\end{align*}
In this sense, the map $\bfJ\mapsto P^N(\bfJ)$ is $N^{-1/2}$-Lipschitz on the event $\|\bfJ\|_{\text{op}}\leq C\sqrt N$, which occurs with high probability. We combine a theorem of \cite{Bobkov-Nayar-Tetali-2017-nonLip} about concentration of Lipschitz maps for restricted measures with classical concentration of measure results to obtain the first claim. The second point follows by a chaining argument.

\subsubsection{Averaged propagation of chaos}
Next, we specialize to Gaussian disorder and show quantitatively that the averaged law converges to $\mu^{\otimes k}$. More precisely, consider a random matrix $\bfG=(G_{ij})_{i,j=1}^N$ with i.i.d. standard Gaussian entries and the associated SDE
\begin{align}
\label{eq: quenched equation gaussian}
    dX_t^i=- U'(X_t^i)\,dt+\frac\beta{\sqrt N} X_t\cdot \bfG_{i}\,dt+dW_t^i,
\end{align}
where $\bfG_i$ is the $i$-th row of the random matrix $\bfG$.
\begin{theorem}[Averaged propagation of chaos] Let $v\subseteq \{1,\dots,N\}$ with $|v|=k$. There exists a constant $C \in (0, \infty)$ such that
    \label{th: averaged propagation of chaos}
    \[\bbW^2_2\Big(\bbE\big[P^{N, v}_{[t]}(\bfG)\big], \mu^{\otimes k}_{[t]}\Big) \leq \frac{Ck}{N}. \]
\end{theorem}
The proof of Theorem \ref{th: averaged propagation of chaos} can be found in Section \ref{sect: averaged propagation of chaos}. We begin by obtaining a slightly different representation for the limit law $\mu$ from \eqref{eq: limit dynamics 0}:
\begin{align*}
    dY_t=-U'(Y_t)\,dt+\beta^2\bigg[\int_0^t R_\mu(t,s)\,dW_s\bigg]\,dt+dW_t,
\end{align*}
where $W$ is a standard Brownian motion and $R_\mu(t,s)$ is a kernel that depends on $\mu$, see Lemma \ref{lem: limit dymamics}. Let us note that a similar representation appeared recently in \cite{Faugeras2025-bp}. 

A key difficulty is to describe the averaged dynamics of the $N$-particle system. To do this, we average over the disorder using the \emph{mimicking theorem}, a powerful tool from filtering theory which allows one to characterize any It\^{o} process as the solution of a SDE (see e.g., Theorem \ref{t:mimicking} or \cite{Brunick-Shreve-2013-Mimicking}). In the present context, this yields a useful SDE representation for $\bbE[P^N(\bfG)]$:
\begin{align}
\label{eq: intro averaged equation CE}
    dX_t^i=-U'(X_t^i)\,dt+\frac{\beta}{\sqrt N}X_t\cdot \bbE[\bfG_{i}\mid X_{[t]}]\,dt+dW_t^i,
\end{align}
where $(W^1, \ldots, W^N)$ is a standard $N$-dimensional Brownian motion. Since the disorder is Gaussian, it is possible to compute the conditional expectation in the SDE, yielding a final SDE representation for the averaged law:
\[dX_t^i=-U'(X_t^i)\,dt+\frac{\beta^2}{N}\bigg[ X_t\cdot \int_0^t Q_sX_s \,dW_s^i \bigg]dt+dW_t^i,\]
where $Q_t$ is an $N\times N$ matrix that depends on the path of $X_{[t]}$. 

This representation turns out to be quite insightful and highlights some of the key difficulties of the model. On the one hand it possesses the more classical $1/N$ McKean--Vlasov scaling instead of the $1/\sqrt N$ present in the original \eqref{eq: intro quenched dynamics}. On the other, it shows that the averaged dynamics evolve according to non-Markovian dynamics. In particular the interaction between particles happens through a stochastic integral, which is unusual and challenging to manipulate. Indeed, one would like to pass $X_t$ under the stochastic integral, since in Lemma \ref{lem: N kernels} below we show that
\[\frac{1}{N} X_t^\top Q_s X_s = R_{\hat{\mu}^N}(t, s), \quad \hat{\mu}^N = \frac{1}{N}\sum_{i = 1}^N \delta_{X^i}. \]
This would put the averaged dynamics in the same form as the limit dynamics and Theorem \ref{th: averaged propagation of chaos} would follow from a standard coupling argument.
Unfortunately swapping of $X_t$ and the stochastic integral is impossible since the integrand of an It\^{o} integral needs to be adapted. 

To address this problem, we use the theory of Malliavin calculus and Skorokhod integrals. A standard fact of Malliavin calculus allows us to pass $X_t$ under the stochastic integral, at the cost of an error term. In particular,
\begin{align*}
    \frac{\beta^2}{N}X_t\cdot \int_0^t Q_sX_s \,dW_s^i= \frac{\beta^2}{N}\int_0^t X_t\cdot Q_sX_s \,d W_s^i+\frac{\beta^2}{N}\int_0^t D_s^i X_t \cdot Q_s X_s\,ds,
\end{align*}
where the stochastic integral on the right-hand side is interpreted as a Skorokhod integral, and $D_s^i X_t$ denotes the $i$-th coordinate of the Malliavin gradient of $X_t$. A key part of the proof is showing that the extra Malliavin term goes to zero, which we establish in Proposition \ref{pr: Malliavin calculus estimates} and Corollary \ref{cor: Lambda-estimate}. The proof of Proposition \ref{pr: Malliavin calculus estimates} is complicated by the fact that the averaged SDE has non-Lipschitz (and non-Markovian) coefficients. Overcoming this difficulty is a key technical contribution of this work, the details of which may be found in Section \ref{sect: malliavin}.
\subsubsection{Universality}
The final ingredient is the following quantitative universality result, which controls the Wasserstein distance between general and Gaussian disorder, denoted $\bfJ$ and $\bfG$ respectively.
\begin{theorem}[Quantitative universality]
    \label{th: universality} 
Let $v\subseteq \{1,\dots,N\}$ with $|v|=k$. Assume that $\bfG=(G_{ij})_{ij}$ is a matrix with i.i.d. standard Gaussian entries, while $\bfJ=(J_{ij})_{ij}$ is a matrix with i.i.d. sub-Gaussian entries with mean $0$ and variance $1$. There exists a constant $C \in (0, \infty)$ such that 
    \[\bbW^2_{2,L}\Big(\bbE\big[P^{N, v}_{[t]}(\bfG)\big],\bbE\big[P^{N, v}_{[t]}(\bfJ)\big]\Big)+\bbW^2_{2}\Big(\bbE\big[P^{N, v}_{t}(\bfG)\big],\bbE\big[P^{N, v}_{t}(\bfJ)\big]\Big) \leq \frac{Ck}{N}. \]
\end{theorem}
To prove universality we exploit the SDE representation of the averaged dynamics. Indeed if we denote by $X_t$ the SDE with disorder $\bfJ$, we obtain the equivalent of  \eqref{eq: intro averaged equation CE}:
\begin{equation}
  dX_t^i=-U'(X_t^i)\,dt+\frac{\beta}{\sqrt N} X_t \cdot \bbE[\bfJ_i\mid X_{[t]}]\,dt+dW_t^i.
  \label{eq: intro universal mimicked SDE}
\end{equation}
\[\]
The main difficulty is that no closed form for the conditional expectations is available in the case of non-Gaussian disorder $\bfJ$. However using Bayes' theorem, we have
\[\bbE[\bfJ_i \mid X_{[t]}]= \frac{\bbE_\bbQ[\bfJ_i f(\bfJ_i)\mid X_{[t]}]}{\bbE_\bbQ[f(\bfJ_i)\mid X_{[t]}]}.\]
where $\bbQ$ is the law of the solution of \eqref{eq: intro quenched dynamics} at $\beta = 0$ and $f:\R^{N} \rightarrow \R_+$ is a function given by Girsanov's theorem related to $d\bbP/d\bbQ$.

The approach we take is inspired by Stein's method, since in the case of Gaussian disorder we have $\bbE_\bbQ[\bfG_i f(\bfG_i) \mid X_{[t]}]=\bbE_\bbQ[\nabla f(\bfG_i)\mid X_{[t]}]$. 
By slight abuse of notation, let $\bbE[\bfG_i \mid X_{[t]}]$ denote the drift function in the second term of \eqref{eq: intro averaged equation CE}, evaluated along the path $X$ which solves \eqref{eq: intro universal mimicked SDE}. Since $\bbE[\bfG_i \mid X_{[t]}]$ has a closed form solution, we may compute
\[\bbE[\bfJ_i \mid X_{[t]}]-\bbE[\bfG_i\mid X_{[t]}]= \bigg(I+\frac{\beta^2}{N}\int_0^t X_sX_s^\top\,ds\bigg)^{-1}\, \frac{\bbE_\bbQ[\bfJ_i f(\bfJ_i)-\nabla f(\bfJ_i)]}{\bbE_\bbQ[f(\bfJ_i)]}.\]
We then check that the right hand side is small using a calculation reminiscent of those appearing in Stein's method arguments. Since the error is small, a coupling argument shows that solutions to \eqref{eq: intro averaged equation CE} and \eqref{eq: intro universal mimicked SDE} must be close to each other. This then yields Theorem \ref{th: universality}.
\subsubsection{Conclusion}
Given Theorems  \ref{th: concentration}, \ref{th: averaged propagation of chaos}, \ref{th: universality}, we easily conclude Theorem \ref{th: main theorem}.
\begin{proof}[Proof of Theorem \ref{th: main theorem}] 
For the first item, notice that by Markov's inequality and the triangle inequality of Wasserstein distance for arbitrary $\alpha \in (0, \infty)$ we have
    \begin{align*}
        &\log \bbP\bigg(\,\bigg|\int_{\mC^k_T}f(x)\,d\big(P^{N, v}_{[T]}(\bfJ)-\mu_{[T]}^{\otimes k}\big)(x)\bigg|\geq r\bigg) \\ \leq &{ -\frac{\alpha r^2N}{k}} + \log\bbE\bigg[\exp\bigg(\frac{N\alpha}{k}\bigg|\int_{\mC^k_T}f(x)\,d\big(P^{N, v}_{[T]}(\bfJ)-\mu_{[T]}^{\otimes k}\big)(x)\bigg|^2\bigg)\bigg]
         \\ \leq &{-\frac{\alpha r^2 N}{k}} +\frac{3\alpha N}{k}\bbW_{1,L}^2\Big(\bbE\big[P_{[T]}^{N, v}(\bfG)\big],\mu^{\otimes k}_{[T]}\Big)+\frac{3\alpha N}{k}\bbW_{1,L}^2\Big(\bbE\big[P_{[T]}^{N, v}(\bfJ)\big],\bbE\big[P_{[T]}^{N, v}(\bfG)\big]\Big) \\
         &\hspace{3.5 em}+ \log\bbE\bigg[\exp\bigg(\frac{3N\alpha}{k}\bigg|\int_{\mC^k_T}f(x)\,d\big(P^{N, v}_{[T]}(\bfJ)-\bbE\big[P^{N, v}_{[T]}(\bfJ)\big]\big)(x) \bigg|^2\bigg)\bigg].
    \end{align*}
    By Theorems \ref{th: averaged propagation of chaos} and \ref{th: universality} and the fact that $\bbW_{1,L}\leq \bbW_{2,L}$, the two $\bbW_{1,L}$ terms are bounded by a constant. If $\alpha < \theta/(2T)$, Theorem \ref{th: concentration} yields that the last term is bounded independently of $k$ and $N$.
     
     The second item follows immediately by the triangle inequality of Wasserstein distance and Theorems \ref{th: concentration}(1), \ref{th: averaged propagation of chaos} and \ref{th: universality}. 
\end{proof}

\subsection{Related work on propagation of chaos} \label{ss: chaos review} The key motivation for the present work is to study the propagation of chaos phenomenon in disordered systems. Historically, propagation of chaos has been studied in the context of mean-field interacting diffusions of the form
\[dX_t^i= \frac1N \sum_{j=1}^N b(X_t^i,X_t^j) \,dt+dW_t^i.\]
When $N$ is large
it was established in the seminal work of McKean \cite{McKean1966} that each $X^i$ behaves approximately like a solution of the \emph{McKean-Vlasov equation}
\begin{equation}
    dY_t = \int_{\R} b(Y_t, z) d\rho_t(z) + dW_t, \quad \rho_t = \text{Law}(Y_t).
    \label{eq: intro-MV}
\end{equation}
In particular, if $b$ is suitably regular and the initial conditions which are \emph{chaotic} in the sense that
\[\lim_{N \rightarrow \infty} \text{Law}\big(X^1_0, \ldots, X^k_0\big) = \rho_0^{\otimes k}, \quad k \in \bbN,  \]then the following equivalent formulations of propagation of chaos hold: for all $t \in [0, T]$, we have
\[\lim_{N \rightarrow \infty}\frac{1}{N}\sum_{i = 1}^N \delta_{X^i_{t} }=\rho_t; \quad \quad\lim_{N \rightarrow \infty} \text{Law}(X^1_t, \ldots, X^k_t)  = \rho_t^{\otimes k}, \quad k \in \bbN. \]
This equivalence depends crucially on the exchangeability of the particles $(X^1, \ldots, X^N)$, see e.g. \cite[Chapter 2]{snitzmann}. The situation is more delicate for non-exchangeable systems, such as the asymmetric Langevin spin glass \eqref{eq: intro quenched dynamics} when quenched on the random matrix $\bfJ$, and the convergence of the empirical measure no longer implies asymptotic independence. We discuss further the problem of propagation of chaos for non-exchangeable systems  in Section \ref{ss: non-exch}.

\subsubsection{Quantitative propagation of chaos in mean-field systems} For a reference on propagation of chaos in mean-field systems, we refer the reader to the classic monograph \cite{snitzmann} and the reviews \cite{ChaintronChaos1, ChaintronChaos2} for modern developments. We will focus on a recent strand of the literature dedicated to proving quantitative rates of propagation of chaos, which is most relevant to the context of our results. 

A diverse set of tools have been developed to study quantitative propagation of chaos. An influential approach based on synchronous coupling was introduced by Sznitzman \cite{snitzmann}, which yields quantitative rates in Wasserstein distance in the setting of Lipschitz interaction functions $b$.:
\begin{align*}
    \bbW_2(\text{Law}(X_t^1,\dots,X_t^k), \rho_t^{\otimes k})=O\big(\sqrt{k/N}\big).
\end{align*}
This was refined further in the works \cite{eberleReflection, eberle2019couplings, reflection} into the \textit{reflection coupling method}, which can be used  to prove propagation of chaos estimates that are uniform-in-time. Another method, which exploits the subadditivity of relative entropy, was introduced in \cite{BenArous-Zeitouni-1999-increasing}. The works \cite{JabinWangBounded, JWEntropicChaos} improved upon the relative entropy method, and in particular their results hold even in settings where $b$ has very low regularity. Such models are important for physical applications and are an area of current research, see \cite{JabinWangReview}.

Both the coupling and the relative entropy methods above work globally by studying the joint law of the full ensemble $(X^1,\dots,X^N)$. The recent works \cite{HierarchiesPaper, HierarchiesUiT} introduce a local perspective which analyzes directly the law of a sub-collection of particles $(X^1,\dots,X^k)$ with $k<N$. This local approach exploits the dependence of the law of the first $k$ particles on the $k+1$-th particle, leading to a sharp estimate in the case of suitably regular drifts:
\begin{align}
    \bbW_2(\text{Law}(X_t^1,\dots,X_t^k),\rho_t^{\otimes k})=O(k/N). 
    \label{eq: sharp MF rate}
\end{align} Subsequent work has used this local methodology to obtain convergence rates in stronger distances \cite{grass2025propagationchaosfisherinformation},\cite{Hess-Childs2025-ux} and study a variety of other models, for instance those with interaction in the diffusion coefficient \cite{sharpChaosDiffusion}, singular drifts \cite{wang2024sharplocalpropagationchaos}, and non-linear dependence in the measure \cite{arnese2026sharppropagationchaosmean}. 

\subsubsection{Non-exchangeable interacting diffusions}\label{ss: non-exch} A class of models closely related to \eqref{eq: intro quenched dynamics} are systems of interacting diffusions with heterogeneous interactions, which take the form
\begin{equation*}
  dX_t^i = \sum_{j = 1}^N \xi_{ij} b(X_t^i, X_t^j) dt + dW_t^i.   
\end{equation*}
Here, $\boldsymbol{\xi} = (\xi_{ij})_{i, j = 1}^N$ is known as the \emph{interaction matrix}, and plays a crucial role in determining the large-$N$ behavior of these systems. Early work focused on examples of $\boldsymbol{\xi}$ which still converge to the McKean-Vlasov equation \eqref{eq: intro-MV}, for instance when $\boldsymbol{\xi}$ is the normalized adjacency matrix of a dense Erd\"{o}s-Reny\'{i} graph with $np_n \rightarrow \infty$, see e.g., \cite{bhamidi2019weakly, coppini2020law, delattre2016note, oliveira2019interacting}. However, for many choices of $\boldsymbol{\xi}$ the dynamics are not well-approximated by the McKean-Vlasov equation and a different limit must be identified, for instance using the theory of graphons \cite{bayraktar2022stationarity, bayraktar2023graphon} and graphops \cite{gkogkas2022graphop, kuehn2022vlasov}. A recent work \cite{jabin2025mean} introduces the notion of extended graphons and obtains a limit for an extremely large class of interaction matrices. {See also the review \cite{ayi2024large}.}

In an entirely different direction, the works \cite{lacker2023localweakconvergence, lacker2023marginal} study the case where $\boldsymbol{\xi}$ is the adjacency matrix of a sparse random graph, for instance the $d$-regular graph or Erd\"{o}s-Reny\'{i} graph with $np_n = O(1)$. In this extremely sparse regime, a different kind of self-averaging mechanism takes over which gives rise to an alternative version of propagation of chaos. The limiting equation in this case is known as the \emph{local-field equation}, which in fact shares many features with the limiting equation \eqref{eq: limit dynamics 0}. In particular, both models are non-Markovian equations whose non-linearity may be interpreted as a suitable conditional expectation. Few results on the local-field equation are available, though see \cite{hu2025case} for an analysis of the long-time behavior in a tractable example and \cite{hu2024h}, which studies a related Markovian version.

To the best of our knowledge, little is known about the rate of propagation of chaos when $\boldsymbol \xi$ is a random matrix. The questions considered in our work are most similar to those investigated in \cite{lacker2024quantitativepropagationchaosnonexchangeable}, which studies quantitative propagation of chaos for non-exchangeable systems. The authors build upon the entropic techniques of \cite{HierarchiesPaper} and identify a remarkable connection with first-passage percolation. While the results of \cite{lacker2024quantitativepropagationchaosnonexchangeable} cover a substantial family of dense matrices $\boldsymbol{\xi}$, 
there are two obstructions which prevent their application in the present setting.
The first is mild and is due to the following row sum condition imposed in \cite[(1.6)]{lacker2024quantitativepropagationchaosnonexchangeable}, see also  \cite[Remark 2.1]{lacker2024quantitativepropagationchaosnonexchangeable}:
\[\max_{1 \leq i \leq N} \sum_{j = 1}^N |\xi_{ij}| = O(1).\]
For the asymmetric Langevin spin glass \eqref{eq: intro quenched dynamics}, the random interaction matrix $\boldsymbol \xi = N^{-1/2} \bfJ$ causes the left hand side of the last display to be of size $O(N)$ with high probability. 
The key obstruction can be seen from \cite[Theorem 2.17]{lacker2024quantitativepropagationchaosnonexchangeable}, which deals with the case of linear $b$ and obtains the sharpest rate in the paper. For $k \ll N$, this result states that
\[\frac{1}{{n\choose k}} \sum_{|S| = k} \bbW_2^2\big(\text{Law}((X^i)_{i \in S}), \rho^{\otimes k} \big) = O\Bigg(\frac{k^2}{N^2}\sum_{i, j =1}^N \xi_{ij}^2 + \frac{k}{N} \sum_{i = 1}^N \bigg[\bigg( \sum_{j = 1}^N \xi_{ij}^2 \bigg)^2 + \bigg( \sum_{j = 1}^N \xi_{ji}^2 \bigg)^2\bigg]\Bigg).\]
From this bound, one recovers the optimal mean-field rate \eqref{eq: sharp MF rate} in weakly-interacting settings where a macroscopic number of $\xi_{ij}$ are of order $N^{-1}$.
This however is not satisfied by our model \eqref{eq: intro quenched dynamics}, as $\xi_{ij} = N^{-1/2} J_{ij} \sim \mN(0, N^{-1})$. Indeed,  in this case the expectation of the second sum on the right-hand side is fatally of size $O(k)$. See \cite[Section 1.3.2]{basak2017universality} for discussion of a related phenomenon in the static setting.

From these two points, it is clear that the fluctuations of the disorder $\bfJ$ are not easily handled by the techniques of \cite{lacker2024quantitativepropagationchaosnonexchangeable}. In contrast, we exploit the cancellations due to those fluctuations to establish concentration of measure and subsequently quenched propagation of chaos. Furthermore in Theorem \ref{t: optimal} we obtain a lower bound in the case $U =0$, which shows that the mean-field rate \eqref{eq: sharp MF rate} cannot hold in our setting.

Finally, we note that our methods are quite different. Due to the complexity of the limiting equation \eqref{eq: limit dynamics 0}, we resort to classical coupling techniques rather than the more sophisticated relative entropy approach of \cite{lacker2024quantitativepropagationchaosnonexchangeable}. As such we expect some of our bounds to be suboptimal, see e.g., Section \ref{ss: open problems}. It would certainly be interesting to see if the entropic techniques of \cite{HierarchiesPaper, JWEntropicChaos, lacker2024quantitativepropagationchaosnonexchangeable} could be used to study diffusions with disordered interaction.

\subsection{Spin glass dynamics}
\label{ss: intro spin glas dynamics}
Spin glasses are disordered materials which display interesting magnetic properties. In their seminal paper, Sherrington and Kirkpatrick introduced a simplified mean-field model of a spin-glass, known eponymously as the \emph{SK model}, which is the Gibbs measure $e^{-\beta H_N}$ on the state space $\{\pm 1\}^{N}$ with Hamiltonian
\[H_N(x) := \frac{1}{\sqrt{N}}\sum_{i, j = 1}^N J_{ij} x^ix^j, \quad x \in \{\pm 1\}^N, \]
where $\bfJ = (J_{ij})$ is a symmetric Gaussian random matrix and $\beta > 0$ is the inverse temperature. Due to the presence of disorder, the Hamiltonian $H_N$ is highly non-convex and induces an intricate energy landscape. As such, the SK model and its asymmetric variants have been used in the modeling of complex systems in neuroscience, biology, statistics, and computer science. See the introduction of \cite{fan2025dynamicalmeanfieldanalysisadaptive} for pointers to applications. We do not attempt to survey the massive literature on the SK model here, instead we refer to the excellent books \cite{talagrand2010mean, panchenko2013sherrington, mezard2009information}. 

The asymmetric dynamic model \eqref{eq: intro quenched dynamics} was introduced in \cite{HGS1986, CS1987} in the study of neural networks, while the symmetric version was proposed as a tool to study the equilibrium SK model in \cite[Chapter V]{mezard1988spin}. Indeed for symmetric disorder, the dynamics \eqref{eq: intro quenched dynamics} is the Langevin equation for a soft-spin approximation of the SK model on $\R^N$,
\[ e^{ - \beta H_N(x) }\prod_{i = 1}^N e^{- U(x_i)} dx_i,\]
where $U$ is a double-well potential. 

As discussed above, the first mathematical results on \eqref{eq: intro quenched dynamics} in both the symmetric and asymmetric setting were obtained by \cite{Arous1995-zj, symmetriclangevinspinglass} in the case of Gaussian disorder. These works obtained large-deviations principles for the empirical measures, from which quenched propagation of chaos at high temperature may be deduced via a replica trick. A follow up work \cite{Guionnet-1997-PoC} extended these results to all temperatures.

Significantly more progress has been made in the simpler setting of the \emph{spherical SK model}, where the single-site product measure is replaced with the uniform distribution on the $N$-sphere. In particular, the work \cite{AgingSpinGlass} studies the large-time behavior of the associated Langevin dynamics and the phenomenon of \emph{aging}, whereby the system's decorrelation properties are time-dependent. Other work studies the Langevin dynamics for the $p$-spin version of the spherical SK model in the large $N$ limit \cite{Dembo2007SphericalHighTemp} and the behavior of this limit in the high temperature regime\cite{Ben_Arous2006-cq} . The more recent work \cite{Subagdynamicssphericalspinglasses} studies the same Langevin equation for the spherical SK model, but under the trickier assumption that the dynamics are initialized from the disordered Gibbs measure. In another direction, the work \cite{BoundingFlowsBenArousGheisari} develops a new method for understanding the dynamics of the spherical spin glass, which is then applied to various problems in high dimensional statistics  \cite{AlgorithmicThresholdsBenArousGheissari, multispikedTensorPCABenArous}.

\subsubsection{Universality}
{Two important and related papers \cite{Dembo-Lubetzky-Zeitouni-2021-Universality, dembo-gheissari-2021-diffusions}  study diffusions with random matrix interactions beyond the Gaussian setting, and in this section we discuss differences between these results and those of the present paper. Our motivations are quite different from those of \cite{Dembo-Lubetzky-Zeitouni-2021-Universality, dembo-gheissari-2021-diffusions}: the main objective of our work is to establish rates of convergence on quenched propagation of chaos, which is not addressed in \cite{Dembo-Lubetzky-Zeitouni-2021-Universality, dembo-gheissari-2021-diffusions}. However, we obtain a universality result in Theorem \ref{th: universality} for a more restrictive class of disorder $\bfJ$.}

In \cite{Dembo-Lubetzky-Zeitouni-2021-Universality}, the authors study universality for the same asymmetric Langevin spin glass model that we consider. They modify the original large deviations method of \cite{Arous1995-zj} together with a Lindeberg-type argument to obtain a qualitative quenched universality result for the empirical measure. However, quenched propagation of chaos does not follow from such a result due to the lack of exchangeability. Moreover, their results are non-quantitative but allow for the more general case of sub-exponential disorder.

The work \cite{dembo-gheissari-2021-diffusions} studies universality for the following stochastic differential equation:
\begin{align}
\label{eq: Dembo-Gheisari equation}
    dX_t=h_t\,dt+\frac{1}{\sqrt N}\bfJ X_t\,dt+\Sigma(X_t)\,dW_t,
\end{align}
where $\bfJ$ is a random matrix with sub-exponential entries (either i.i.d. or symmetric), $\Sigma$ is an affine function while $h_t$ is a bounded random field. The results of \cite{dembo-gheissari-2021-diffusions} are in spirit closer to ours than those of \cite{Dembo-Lubetzky-Zeitouni-2021-Universality}, but the model has several important differences which we now highlight. The main point is that \eqref{eq: Dembo-Gheisari equation} is an affine equation, while the asymmetric Langevin spin glass \eqref{eq: intro quenched dynamics} possesses a non-linear confining potential $U$. In particular, the conditions on $U$ imposed in Assumption \ref{assumption: main}(1) causes the solution to \eqref{eq: intro quenched dynamics} to be bounded, a property which we exploit extensively. On the other hand, the solution to \eqref{eq: Dembo-Gheisari equation} is not confined to a compact set, though the approach of \cite{dembo-gheissari-2021-diffusions} cannot handle the case of an explosive $U$, see e.g., \cite[Remark 1.5]{dembo-gheissari-2021-diffusions}. Another difference is that \eqref{eq: Dembo-Gheisari equation} allows for a degenerate or state dependent diffusion coefficient, which is not addressed by our technique. Finally, the results of \cite{dembo-gheissari-2021-diffusions} include sub-exponential disorder, whereas our results are restricted to sub-Gaussian.
There are also several differences between the types of universality results of the present paper and those of \cite{dembo-gheissari-2021-diffusions}, which are established via stochastic Taylor expansion. The first result \cite[Theorem 1]{dembo-gheissari-2021-diffusions} establishes that for normalized polynomials $F_\bfJ$ of the time-marginals of $(X^1, \ldots, X^N)$, which includes, for instance $\frac{1}{N}\sum_{i = 1}^N X^i_t$ and $\frac{1}{N^2}\sum_{i, j = 1}^NX^i_{t_1}X^j_{t_2}$, the following quantitative universality estimate:

\begin{align}
\label{eq: DG universality}
    \sup_{t\leq T}\Big|\bbE\big[F_\bfJ\big]-\bbE\big[F_\bfG\big]\Big|=O(N^{-1/2}),
\end{align} 
where $\bfG$ as usual is the matrix of i.i.d. standard Gaussian random variables. See \cite[(1.4)]{dembo-gheissari-2021-diffusions} for the precise class of polynomials. 
The second result \cite[Theorem 2]{dembo-gheissari-2021-diffusions} is a sub-exponential concentration property for the aforementioned class of polynomials up to degree 2, for which they prove:
\begin{align}
\label{eq: DG concentration}
    \bbP\Big(\sup_{t_1,t_2}\big|F_\bfJ-\bbE[F_\bfJ]\big|\geq \lambda\Big)\leq \left \{\begin{aligned}
        &N^C e^{-\lambda \sqrt N/C} &\quad &\text{ if }\lambda \leq C,\\
         &e^{-\sqrt N\log \lambda /C} &\quad &\text{ otherwise }.
    \end{aligned} \right.
\end{align}
 These results are comparable to Theorem \ref{th: universality}, but hold only for single observables instead of uniformly over Lipschitz functions. However, despite the similarity to Theorem \ref{th: universality}, these results are quite different from our quenched propagation of chaos result, Theorem \ref{th: main theorem}. In particular, neither \eqref{eq: DG universality} nor \eqref{eq: DG concentration} alone establishes propagation of chaos.  Indeed the SDE \eqref{eq: Dembo-Gheisari equation} is highly flexible and it is unlikely that propagation of chaos should hold for such a general model.

\subsubsection{Dynamical mean-field theory} Finally, we discuss recent work on some interesting related models arising in the study of high dimensional gradient flows. Dynamics of this type arise in Bayesian regression problems, where one is tasked with minimizing the cost function:
\begin{align}
\label{eq: intro stochastic opt}
    \R^d\ni x\mapsto \frac12 |y-\bfJ x|^2,
\end{align}
 for a random matrix $\bfJ \in \R^{N\times d}$ and deterministic vector $y\in \R^N$. The authors of \cite{celentano2025highdimensionalasymptoticsordermethods} study the high dimensional asymptotics in the challenging $d\approx N$ regime of a noiseless gradient flow
\[dX_t=(y\bfJ-\bfJ^\top \bfJ X_t)\,dt.\]
Their techniques are based on time discretization, which allow the authors to reduce the problem to the \emph{approximate message passing algorithm}. With a similar motivation, \cite{GerbelotSGD} establishes high dimensional limits for discrete time first order methods applied to minimizing \eqref{eq: intro stochastic opt} such as gradient descent and Langevin algorithms; while this analysis is qualitative, \cite{GeneralFirstOrderMethodsQuant} provides a quantitative universality rate. The recent \cite{dandi2026rigorousasymptoticsfirstorderalgorithms} gives a rigorous proof of DMFT using Gaussian conditioning for general first order methods.

A related problem is studied in \cite{fan2025dynamicalmeanfieldanalysisadaptive,fan2025dynamical}, who study a Langevin dynamics process for empirical Bayes. The motivation is to sample the posterior for the Bayesian regression 
\[y= \bfJ X+\varepsilon,\quad \varepsilon \sim \mN(0,I),\]
with an i.i.d. prior $g(\cdot, \alpha^*)$ on the regression coefficients $X$, where $\alpha^*$ is some parameter. In this equation, $\bfJ$ represents the random data.
The process studied in \cite{fan2025dynamicalmeanfieldanalysisadaptive,fan2025dynamical} incorporates a joint sampling and learning procedure and displays both random interactions and more traditional McKean--Vlasov components:
\begin{align*}
    dX_t&=-\frac12\nabla_x  \|y-\bfJ X_t\|^2\,dt+\sum_{i=1}^N \nabla_x \log g(X_t^j,\alpha_t)\,dt+\sqrt 2dW_t,\\
    d\alpha_t&=\frac 1N \sum_{j=1}^N \nabla_\alpha \log g(X_t^j,\alpha_t)\,dt.
\end{align*}
The associated McKean-Vlasov equation is also non-Markovian and is referred to as a \emph{dynamical mean-field equation}. The first contribution of \cite{fan2025dynamical} is to prove well-posedness of the candidate limit; moreover, building on the techniques of \cite{celentano2025highdimensionalasymptoticsordermethods}, \cite{fan2025dynamical} proves quenched propagation of chaos to the high dimensional limit, as well as almost sure convergence of the associated empirical measures. The long time behavior of the limit is studied in \cite{fan2025dynamicalmeanfieldanalysisadaptive}. This model is more intricate than the asymmetric Langevin spin glass we consider here, and it certainly would be interesting but challenging to study it quantitatively with our techniques.

The very recent work \cite{chaintron2026resnetsshapessizesconvergence} is one of the few examples of quantitative concentration rates around the DMFT limit. They study gradient descent dynamics used to train residual neural networks, obtaining quantitative rates for convergence with high probability of the training dynamics to the \textit{mean field ODE} as a function of the embedding dimension, width and depth of the network. Their analysis is based on the cavity method, a Lindeberg-type central limit theorem and propagation of chaos techniques. With the exception of the shared propagation of chaos perspective, our techniques and settings are thus quite different.

\subsection{Open problems} \label{ss: open problems} We now discuss a few open directions. A natural question is whether our techniques transfer to the case of symmetric disorder, which were studied in \cite{symmetriclangevinspinglass}. The symmetric case is of great physical interest as it corresponds to the Langevin dynamics of the original SK model, see Section \ref{ss: intro spin glas dynamics}. However, the limit dynamics identified in \cite{symmetriclangevinspinglass} are significantly more complicated than those of the asymmetric case, see Remark \ref{rk: symmetric}. Similarly, it would be interesting to study the model with an additional magnetic field, as considered in \cite{symmetriclangevinspinglass, Arous1995-zj, Guionnet-1997-PoC}, or $p$-spin models as in \cite{AgingSpinGlass, Ben_Arous2006-cq, Subagdynamicssphericalspinglasses}.

In another direction, it would be interesting to understand better the optimal rate. As we discuss in Section \ref{sec: optimality}, we believe that Theorem \ref{th: main theorem} is optimal for $k = 1$. However, due to the poor estimates from the concentration of measure argument, we expect our bounds to be suboptimal for large $k$. We are far from making an informed conjecture, but we speculate that the optimal rate should be \[\bbE\big[\bbW_2\big(P^{k}_{[t]}(\bfJ ), \mu^{\otimes k}_{[t]}\big)\big] = O\big( k/\sqrt{N}\big).\] 
This is partially informed by forthcoming joint work of the second author \cite{hu2026unimodular} in the setting where $N^{-1/2} \bfJ$ is replaced by the adjacency matrix of a sparse random graph. In this work, a version of the previous estimate is established, though with suboptimal rate $\sqrt{k^3/N}.$ 

It is also possible that the averaged rate of Theorem \ref{th: averaged propagation of chaos} could be improved. In the averaged regime the McKean-Vlasov scaling becomes more apparent, and our proofs use a synchronous coupling argument that is known to be suboptimal in simpler cases. We would expect (say in the case of Gaussian disorder),
\[\bbW_2(\bbE[P_{[t]}^{N,k}(\bfG)],\mu^{\otimes k}_{[t]})=O(k/N).\]
Such a rate is well known to be optimal in  classic mean-field models \cite{HierarchiesPaper}, even in non-Markovian and certain singular settings \cite{wang2024sharplocalpropagationchaos}. Nevertheless, the averaged equation \eqref{eq: intro averaged equation CE} of the present setting is quite unusual. In particular, the presence of the noise in the drift makes it difficult to use the modern entropy-based methods typically employed to obtain sharp rates.

Finally we mention that any results on uniform-in-time propagation of chaos or long-time behavior of the limit equation would be of great interest.

\subsection{Outline} \label{ss: outline}  The rest of the paper is structured as follows. In Section \ref{sect: concentration} we prove concentration of measure for the quenched law (Theorem \ref{th: concentration}). Next in Section \ref{sect: averaged propagation of chaos} we prove Theorem \ref{th: averaged propagation of chaos}, that is quantitative averaged propagation of chaos \eqref{eq: intro-k-avg-poc} in the case of Gaussian disorder. The proof of Theorem \ref{th: averaged propagation of chaos} makes use of several representations of the averaged and limit dynamics, which we collect in Section \ref{sect: representation}. One such representation requires a rather technical argument involving Malliavin calculus, which we defer to Section \ref{sect: malliavin}. In Section \ref{sect: universality}, we prove Theorem \ref{th: universality}, which provides the quantitative universality estimate. Finally, in Section \ref{sec: optimality} we argue that Theorem \ref{th: main theorem} is optimal in the case of $k = 1$ by studying \eqref{eq: intro quenched dynamics} in the case $U \equiv 0$. Due to the absence of confinement, this simple setting is not covered by our results or those of \cite{Arous1995-zj, Guionnet-1997-PoC} and it is not immediate that there is a limiting process. We address this minor point in Appendix \ref{sec: wick} and show that the law of a single spin converges to a deterministic limit. In Appendix \ref{sec: Volterra}, we recall a few results about Volterra processes which will be used in Section \ref{sect: representation}.

\section{Concentration of measure}
\label{sect: concentration}
\subsection{Preliminaries}
In this section we prove Theorem \ref{th: quenched propagation of chaos}. Recall that for a fixed matrix $\bfJ$ we denote by $P^N_{[T]}(\bfJ)=\text{Law}(X^1_{[T]},\dots,X^N_{[T]}\mid \bfJ)$ the path-space law of the unique strong solution of \eqref{eq: intro quenched dynamics}.
Indeed, the strong well-posedness of \eqref{eq: intro quenched dynamics} follows from \cite[Proposition 2.1]{Arous1995-zj}, together with a standard strong-weak uniqueness argument: since the coefficients are locally Lipschitz, the solutions are pathwise unique by \cite[Theorem 5.2.5]{Karatzas_and_Shreve_BMSC}, which implies that they are strong by Yamada-Watanabe theory \cite[Corollary 5.3.25]{Karatzas_and_Shreve_BMSC}. 

For convenience, we will omit the dependence on $N$ in the superscript and write $P^v = P^{v,N}$ for $v \subset\{1, \ldots, N\}$. In the following, we use $\sfE_{\bfJ}$ to denote the quenched expectation, that is for all $t \in [0, T]$ and bounded continuous $f: \mC_t^N \rightarrow \R$,
\[\sfE_{\bfJ}[f(X_{[t]})] := \int_{\mC_t^N} f(x) dP_{[t]}(\bfJ)(x).\]

\subsection{A local Lipschitz property of the solution map}

The purpose of this section is to show that the map $\bfJ \mapsto P_{[t]}(\bfJ)$ is $N^{-1/2}$-Lipschitz on a large subset of $\R^{N \times N}$. This property, which is formulated precisely in Lemma \ref{lm: Lipschitzianity of integral map} below, allows us to access classical concentration of measure results from which Theorem \ref{th: quenched propagation of chaos} will follow. The set where the mapping is regular is exactly the set where the operator norm of $\bfJ$ is not too large, and a classical result in random matrix theory implies that this happens with high probability, see for instance \cite[Theorem 4.4.3]{Vershynin2018HDP}.
\begin{proposition}
\label{pr: rnd matrix operator bound}
There exists a constant $C_S \in (0, \infty)$ depending on $\sigma^2$ only such that if
\begin{align*}
    S:= \Big\{\bfJ \in \R^{N\times N}: \|\bfJ\|_{\textup{op}}\leq C_{S}\sqrt N\Big\},
\end{align*}
then
\begin{equation}
\label{eq: concentration of operator norm}
    \bbP(S^c)\leq 2e^{-N}.
\end{equation}
In particular $C_S$ does not depend on $N$.
\end{proposition}

\begin{lemma}
\label{lm: Lipschitzianity of integral map} Let $f:C([0,T],[-A,A]^{N})\to \R$ be a bounded function. Define
\begin{equation}
    F_f(\bfJ):=\sfE_\bfJ[f(X_t)].
    \label{eq: F_f}
\end{equation}
    Fix $\bfJ,\bfJ' \in \R^{N \times N}$ with $\bfJ \in S$. Then there exists a constant $C \in (0, \infty)$ independent of $N$ such that
    \begin{align*}
        |F_f(\bfJ)-F_f(\bfJ')|\leq \frac{C\|f\|_\infty}{\sqrt N}\|\bfJ-\bfJ'\|_{\textup{Fr}}.
    \end{align*}
\end{lemma}

The proof of Lemma \ref{lm: Lipschitzianity of integral map} proceeds in several steps. First, we obtain uniform-in-$N$ bounds on the operator norm of the covariance $\sfE_\bfJ[X_tX_t^\top]$ via a Poincaré inequality.
\begin{lemma}
\label{lm: Poincare inequality}
    For all $\bfJ \in \R^{N\times N}$, the measure $P_t(\bfJ)$ satisfies the covariance bound 
    \begin{align*}
        \|\sfE_\bfJ[X_tX^\top_t]\|_{\textup{op}}\leq C_{\textup{Var}}\exp\Big(2t\beta N^{-1/2}\|\bfJ\|_{\textup{op}}\Big),
    \end{align*}
    where
    \begin{align*}
        C_{\textup{Var}}:=(C_\textup{P}(0)+T) e^{2tC_{\textup{Lip}}}.
    \end{align*}
\end{lemma}
    \begin{proof} 
    By slight abuse of notation, we  write $\nabla U(x)=(U'(x^1),\dots,U'(x^N))$. Let $b_\bfJ(x):=-\nabla U(x)+\beta N^{-1/2}\bfJ x$ denote the drift of $X$ under $P_{\bfJ}$ and observe that $b_{\bfJ}$ satisfies a one sided Lipschitz bound with constant $C_{\textup{Lip}}+\beta N^{-1/2}\|\bfJ\|_{\text{op}}$:
        \begin{align*}
            \big(b_\bfJ(x)-b_\bfJ(y)\big)\cdot (x-y)\leq (C_{\textup{Lip}}+\beta N^{-1/2}\|\bfJ\|_{\text{op}})|x-y|^2.
        \end{align*}
        By \cite[Proposition 9]{Cattiaux2014LogCONCAVEMARKOV}, the measure 
        $P_\bfJ(t)$ satisfies a Poincaré inequality with constant 
    \[C_\text{P}(0)e^{2t(C_{\textup{Lip}}+\beta N^{-1/2}\|\bfJ\|_{\text{op}})}+\frac{e^{2t(C_{\textup{Lip}}+\beta N^{-1/2}\|\bfJ\|_{\text{op}})}-1}{(C_{\textup{Lip}}+\beta N^{-1/2}\|\bfJ\|_{\text{op}})}\leq C_{\textup{Var}}\exp\Big(2t\beta N^{-1/2}\|\bfJ\|_{\textup{op}}\Big).\]
    In particular, for each differentiable and 1-Lipschitz function $f:\R^N\to \R$,
    \begin{align*}
        \sfE_{\bfJ}\Big[|f(X_t)-\sfE_\bfJ[f(X_t)]|^2\Big]&\leq C_{\textup{Var}}\exp\Big(2t\beta N^{-1/2}\|\bfJ\|_{\textup{op}}\Big).
    \end{align*}
    Since $U$ is even and $P_0^N$ is symmetric by Assumption \ref{assumption: main}, $\sfE_\bfJ[X_t]=0$ for every $\bfJ$ and $t\geq0$. This is because $\overline X:=-X$ solves the same SDE as $X$, whose solution is unique in law. We then obtain
    \begin{align*}
\|\sfE_\bfJ[X_tX^\top_t]\|_{\textup{op}}&=\sup_{|u|=1} \sfE_\bfJ[|u^\top X_t-\sfE_\bfJ[u^\top X_t]|^2]\leq C_{\textup{Var}}\exp\Big(2t\beta N^{-1/2}\|\bfJ\|_{\textup{op}}\Big),
    \end{align*}
    which is the desired result.
    \end{proof}
\begin{proof}[Proof of Lemma \ref{lm: Lipschitzianity of integral map}] We begin by controlling the difference $|F_f(\bfJ) - F_f(\bfJ')|$ by relative entropy. For two probability measures $\alpha,\beta$ recall the definition of the relative entropy of $\alpha$ with respect to $\beta$: 
\begin{align*}
    \mH(\alpha\|\beta):=\left \{ \begin{aligned} &\int \bigg(\frac{d\alpha}{d\beta}\log  \frac{d\alpha}{d\beta}\bigg) d\beta &\quad &\text{ if } \alpha \ll \beta,\\
    &+\infty &\quad & \text{ otherwise.}
    \end{aligned} \right.
\end{align*}
By \eqref{eq: F_f} and Pinsker's inequality \cite[Lemma 2.5(1)]{Tsybakov2009}, we have
\begin{equation}
\label{e: F_f pinsker}
    |F_f(\bfJ)-F_f(\bfJ')|^2\leq \|f\|^2_{\infty}\|P_{[T]}(\bfJ)-P_{[T]}^N(\bfJ')\|_{\text{TV}}^2 \leq \frac {\|f\|^2_{\infty}}{2}{\mH\big(P_{[T]}(\bfJ)\,\big\|\,P_{[T]}(\bfJ')\big)}.
\end{equation}

Next, we estimate the relative entropy by a standard Girsanov argument. By \cite[Lemma 2.1]{Arous1995-zj} the SDE \eqref{eq: intro quenched dynamics} has a unique weak solution for any choice of disorder. Then by \cite[Lemma 4.4]{HierarchiesPaper} we have
\begin{align*}
    \mathcal \mH\big(P_{[T]}(\bfJ)\big\|P_{[T]}(\bfJ')\big)&=\frac12 \sum_{i=1}^N\int_0^T \sfE_\bfJ\bigg[\Big|\frac{\beta}{\sqrt{N}}\sum_{j=1}^NX_t^j (J_{ij}-J_{ij}')\Big|^2\bigg]dt\\
    &=\frac{\beta^2}{2N} \int_0^T\text{Tr}\big((\bfJ-\bfJ')^\top \sfE_\bfJ\big[X_tX_t^\top\big](\bfJ-\bfJ')\big)\,dt \\ &\leq \frac{\beta^2}{2N} \big\|\bfJ-\bfJ'\big\|_{\text{\textup{Fr}}}^2\int_0^T\big\|\sfE_\bfJ\big[X_tX_t^\top\big]\big\|_{\textup{op}}\,dt.
\end{align*}
Since $\bfJ \in S$, the last display together with Lemma \ref{lm: Poincare inequality} implies
\begin{align*}
     \mathcal \mH\big(P_{[T]}(\bfJ)\big\|P_{\bfJ'}[T]\big)\leq \frac{\beta^2}{2N} \big\|\bfJ-\bfJ'\big\|_{\text{Fr}}^2\int_0^TC_{\text{var}}e^{\frac{2t\beta}{\sqrt N}\|\bfJ\|_\textup{op}}\,dt\leq \frac{\beta^2}{2N} \big\|\bfJ-\bfJ'\big\|_{\text{Fr}}^2C_{\text{var}}e^{
     2T\beta C_S}T.
\end{align*}
The claim follows on combining the last display with \eqref{e: F_f pinsker}.

\end{proof}

Finally, we recall the following Theorem from \cite{Bobkov-Nayar-Tetali-2017-nonLip}, which provides estimates on the sub-Gaussian norm of conditional measures.

\begin{theorem}{\cite[Theorem 1.1]{Bobkov-Nayar-Tetali-2017-nonLip}}
\label{th: bobkov restricted measure}
Suppose $\nu \in \mP(\R^d)$ satisfies a T$_2$ inequality with constant $\sigma^2$. For a Borel set $A\subseteq \R^d$ with $\nu(A)>0$, define  $\nu_A(S):=\nu(A\cap S)/\nu(A)$. Then $\nu_A$ satisfies a T$_2$ inequality with constant $\sigma^2_A$, and
\begin{align*}
    \sigma_A^2\leq c\log\Big(\frac{e}{\nu(A)}\Big)\sigma^2,
\end{align*}
where $c$ is a universal constant. In particular for every $L \in (0, \infty)$ and each $L$-Lipschitz function $f:\R^d\to \R$ with $\int f\,d\nu_A=0,$ the following hold:
\begin{enumerate}
    \item For each $\lambda\in \R$ we have
\begin{align}
    \int_{\R^d} e^{\lambda f(x)}d\nu_A(x)\leq e^{L^2\lambda^2\sigma^2_A/2}.
    \label{eq: bobkov eq}
    \end{align}
    \item For each $\lambda < \frac{1}{2\sigma_AL}$, we have
    \begin{equation}
    \int_{\R^d} e^{\lambda^2 f(x)^2/2} d\nu_A(x)\leq 2.
        \label{eq: bobkov eq 2}
    \end{equation}
\end{enumerate}
\end{theorem}

The above, together with the estimate \ref{pr: rnd matrix operator bound} $\bbP(S)\geq 1-2/e$ shows that the law of $\bfJ$ conditioned on $S$ satisfies sub-Gaussian concentration of Lipschitz functions with constant $\sigma^2_S\leq 10\sigma^2$.
\subsection{Proof of Theorem \ref{th: quenched propagation of chaos}}

Fix $k\leq N$ and a set $v\subseteq \{1, \ldots, N\}$ of size $k$. Let $f:\mC_T^k\to \R$ be a $1$-Lipschitz function, which we may assume without loss of generality satisfies $f(0)=0$. By \eqref{eq: F_f} we have \[\|F_f(\bfJ)\|_\infty\leq \|f\|_\infty\leq A\sqrt k.\] Recall the definition of the set $S$ from Proposition \ref{pr: rnd matrix operator bound}. By the concentration bound $\eqref{eq: concentration of operator norm}$, we have
\begin{align}
\label{eq: simple bound on bad event}
    \bbE[F_f(\bfJ)]-\bbE[F_f(\bfJ)\mid S]= \bbE[F_f(\bfJ)(1-\mathbf{1}_{S}\bbP(S)^{-1})]\leq 2\|f\|_\infty \bbP(S^c)\leq 2A\sqrt k e^{-N}.
\end{align}
    Fix $0 < \theta < \min\{1/(8A^{2}),1/(32\sigma_S^2A^2 C^2))\}$ where $C$ is the constant from Lemma \ref{lm: Lipschitzianity of integral map}. Then  \eqref{eq: simple bound on bad event} implies,
    \begin{align*}
        \bbE\bigg[\exp\bigg(\frac{ N\theta}{k}\big|F_f(\bfJ)-\bbE[F_f(\bfJ)]|^2\bigg)\bigg]
        \leq &\bbE\bigg[\exp\bigg(\frac{2N\theta}{k}\big|F_f(\bfJ)-\bbE[F_f(\bfJ) \mid S]|^2\bigg)\bigg] \exp\big(4N\theta A^2 e^{-2N}\big)\\
        \leq &C\bigg(\bbE\bigg[\exp\bigg(\frac{4N\theta}{k}\big|F_f(\bfJ)-\bbE[F_f(\bfJ) \mid S]\big|^2\bigg)\bigg| S\bigg]+2 e^{-N(1-8\theta A^2)}\Big).
    \end{align*}
    Since $\theta<1/(8A^2)$, the second term on the right-hand side decays exponentially in $N$. By Lemma \ref{lm: Lipschitzianity of integral map}, the map $\bfJ\mapsto F_f(\bfJ)$ is $C A\sqrt{k/N}$-Lipschitz on the set $S$ for some constant $C \in (0, \infty)$ which does not depend on $N$. By \eqref{eq: bobkov eq 2}, the following bound holds, which together with the last display implies the first item:
    \[\bbE\bigg[\exp\bigg(\frac{4N\theta}{k}\big|F_f(\bfJ)-\bbE[F_f(\bfJ) \mid S]\big|^2\bigg)\bigg| S\bigg] \leq 2. \]
      
We now prove the second item by combining the first item with a chaining argument. Similar arguments have appeared elsewhere in the literature, for example in the proof of the convergence rate of an empirical measure to its limit in Wasserstein distance, see  e.g., \cite[Proposition 2.5]{Chewi-Niles-Weed-Rigollet-2025-SOT}. For each $\bfJ \in \R^{N\times N}$, the measure $P_t^v(\bfJ)$ is supported on $[-A,A]^k$. In particular, we have the following trivial bound:
    \[\bbW_1(P_\bfJ^v(t),\bbE[P_\bfJ^v(t)])\leq 2\sqrt kA.\] Thus by \eqref{eq: simple bound on bad event}, we have
\begin{equation}
\label{eq: first step chaining}
\bbE\Big[\bbW_1(P^v_{t}(\bfJ),\bbE[P^v_{t}(\bfJ)])\Big]\leq \bbE\Big[\bbW_1(P^v_{t}(\bfJ),\bbE[P^v_{t}(\bfJ)])\mid S\Big]+4\sqrt k A e^{-N}.
\end{equation}
Define $\frF:=\{f:[-A,A]^k \to \R\mid f \text{ is }1\text{-Lipschitz}, f(0)=0\}$. By Kantorovich duality, the triangle inequality, and \eqref{eq: simple bound on bad event}, we have
\[\begin{aligned}\bbW_1(P^v_{t}(\bfJ),\bbE[P^v_{t}(\bfJ)])&=\sup_{f\in \frF}\,(F_f(\bfJ)-\bbE[F_f(\bfJ)]) \\ & \leq \sup_{f\in \frF}\,(F_f(\bfJ)-\bbE[F_f(\bfJ)|S]) + 2A\sqrt k e^{-N}.  \end{aligned}\]
This, together with the last display, implies
\begin{align*}
\bbE\Big[\bbW_1(P^v_{t}(\bfJ),\bbE[P^v_{t}(\bfJ)])\Big]\leq \bbE\Big[\bbW_1(P^v_{t}(\bfJ),\bbE[P^v_{t}(\bfJ)\mid S])\mid S\Big]+8\sqrt k A e^{-N}.    
\end{align*}
For $f,g\in \frF$, by \eqref{eq: F_f} we have
         \[
         \begin{aligned}
         |F_f(\bfJ)-\bbE[F_f(\bfJ)\mid S]-(F_g(\bfJ)-\bbE[F_g(\bfJ)\mid S])| &\leq  |\sfE_\bfJ [(f - g)(X_{[T]})] + \big|\bbE\big[ (f - g) (X_{[T]})\mid  S] \big|
         \\   
         & \leq 2\|f-g\|_\infty.             
         \end{aligned}
\] Moreover by Theorem \eqref{th: bobkov restricted measure} and Lemma \ref{lm: Lipschitzianity of integral map} for $r\in \R$ 
        \begin{align*}
            &\bbE\Big[\exp\Big(r(F_f(\bfJ)-\bbE[F_f(\bfJ)\mid S]-(F_g(\bfJ)-\bbE[F_g(\bfJ)\mid S]))\Big)\mid S\Big] \\
            =&\bbE\Big[\exp\Big(r(F_{f-g}(\bfJ)-\bbE[F_{f-g}(\bfJ)\mid S])\Big)\mid S\Big]\\
            \leq &\exp\bigg( \frac{Cr^2 \|f-g\|_\infty^2}{N}\bigg).
        \end{align*}
    we can now apply Dudley's bound (\cite[Theorem 5.31]{VanHandelHDP}) to obtain
    \begin{align*}
\bbE\Big[\bbW_1(P^v_{t}(\bfJ),\bbE[P^v_{t}(\bfJ)\mid S])\mid S\Big]&\leq \bbE\Big[\sup_{f\in \frF}(F_f(\bfJ)-\bbE[F_f(\bfJ)\mid S])\mid S\Big]\\
        &\leq \inf_{\delta>0} \bigg[4\delta+\frac{12C}{\sqrt N}\int_\delta^{\infty} \sqrt{\log \sfN (\frF,r)}\,dr\bigg].
    \end{align*}
    where $\sfN(\frF,r)$ is the covering number of $\frF$ in sup-norm with balls of radius $r$. Observe that for $r>\sqrt k A$, $\sfN(\frF,r)=1$. On the other hand, by \cite[Lemma 2.7]{Chewi-Niles-Weed-Rigollet-2025-SOT} we have $\log \sfN (\frF,r)\leq (A4\sqrt k /r)^k$. Therefore,
     \begin{align*}
        \bbE\Big[\bbW_1(P^v_{t}(\bfJ),\bbE[P^v_{t}(\bfJ)\mid S])\mid S\Big]&\leq \inf_{\delta>0} \bigg[4\delta+\frac{C}{\sqrt N}\int_\delta^{A\sqrt k} \frac{(A\sqrt k)^{k/2}}{r^{k/2}}\,dr\bigg].
    \end{align*}
    Now if $k=1$, we may choose $\delta=0$ to find $\bbE[\bbW_1(P^v_{t}(\bfJ),\bbE[P^v_{t}(\bfJ)\mid S])\mid S]\leq C/\sqrt N$. Otherwise we conclude by picking $\delta=A\sqrt k N^{-1/k}$ and computing the integral. \hfill $\square$

\section{Representation of limit and averaged dynamics}
\label{sect: representation}
\subsection{Representation of the limit dynamics} 
\label{ss: limit-rep}
First we improve the representation \eqref{eq: limit dynamics 0} by removing the dependence on the auxiliary process $B$. This in particular will be useful in the proof of Theorem \ref{th: averaged propagation of chaos}. A similar representation has previously appeared in \cite{Faugeras2025-bp}. 

We introduce in the following lemma a function $R_\mu:[0, T]^2 \rightarrow \R$  which acts as the \emph{Volterra resolvent} of $H_t^\mu$.

\begin{lemma} Let $\nu \in \mP(\mC_T)$. Then there exists a function $R_\nu\in L^2([0, T]^2)$ such that
\begin{equation}
  R_\nu(t, s) - H_\nu^t(t, s) = \beta^2 \int_s^t H^t_\nu(t, u) R_\nu(u, s) du = \beta^2 \int_s^t R_\nu(t, u) H^u_\nu(u, s) du,  
  \label{eq: H resolvent}
\end{equation}
and
\begin{equation}
    R_\nu(t, s) \leq Ae^{A(t - s)},
    \label{eq: R bound}
\end{equation}
for all $0 \leq s \leq t \leq T$. Moreover there exists a constant $C \in (0, \infty)$ depending only on $(\beta, A, T)$ such that for all $\xi, \nu \in \mP(\mC_T)$, we have
\begin{equation}
\label{eq: R contiuity}
    \int_0^T\int_0^t |R_\xi(t, s) - R_\nu(t, s)|^2 ds\, dt \leq C \int_0^T\int_0^T |K_\xi(t, s) - K_\nu(t, s)|^2 ds\, dt.
\end{equation}
\label{lem: resolvent}
\end{lemma}
We defer the proof of Lemma \ref{lem: resolvent} to Appendix \ref{sec: Volterra}. With Lemma \ref{lem: resolvent} in hand, we obtain a simpler representation of the limit dynamics \eqref{eq: limit dynamics 0}. 

\begin{lemma} Let $\mu \in \mP(\mC_T)$ be the solution to \eqref{eq: limit dynamics 0} and $Y \sim \mu$. Then we have
\begin{equation}
\begin{aligned}
    dY_t &= - U'(Y_t) dt + \beta^2 \bigg[\int_0^t R_\mu(t, s) dW_s \bigg] dt + dW_t. 
\end{aligned} 
\label{eq: limit dynamics 1}
\end{equation}
with $Y_0  \sim \mu_0$. Moreover this SDE has a unique strong solution.
\label{lem: limit dymamics}
\end{lemma}
\begin{proof}
    Recall the process $B$ from \eqref{eq: limit dynamics 0}. Define:
    \[Z_t := \int_0^t H_\mu^t(t, s) dB_s, \quad t \in [0, T]. \]
    In particular \eqref{eq: limit dynamics 0} may be written as $dB_t = \beta^2 Z_t dt + dW_t$ and thus
    \[Z_t = \beta^2 \int_0^t H_\mu^t(t, s) Z_s ds + \int_0^t H_\mu^t(t, s) dW_s. \]
    Now, consider the following process, which is well-defined by \eqref{eq: R bound}.
    \begin{equation}
    \label{eq: tlZ}
        \tlZ_t = \int_0^t R_\mu(t, s) dW_s. 
    \end{equation}
    We may apply \eqref{eq: H resolvent} and the stochastic Fubini theorem (see e.g., Theorem 65 of \cite{Protter2004}) to obtain
    \[\begin{aligned}\tlZ_t &= \int_0^t H_\mu^t(t, s) dW_s + \beta^2  \int_0^t \bigg[\int_s^t H_\mu^t(t, u) R_\mu(u, s) du\bigg] dW_s \\
    &= \int_0^t H_\mu^t(t, s) dW_s + \beta^2  \int_0^t H_\mu^t(t, u)  \bigg[\int_0^uR_\mu(u, s)dW_s \bigg] du \\
    &= \int_0^t H_\mu^t(t, s) dW_s + \beta^2  \int_0^t H_\mu^t(t, u)  \tlZ_u du.
    \end{aligned}\]
    In particular by the last 3 displays and \eqref{eq: H bound} we have
    \[|\tlZ_t - Z_t|^2 \leq C\int_0^t |H^t_\mu(t, s)|^2|\tlZ_s - Z_s|^2 ds \leq C \int_0^t |\tlZ_s - Z_s|^2 ds,\]
    and by Gr\"onwall's inequality we have $\tlZ = Z$. Substituting \eqref{eq: tlZ} into \eqref{eq: limit dynamics 0} yields the claim. To prove the existence of a strong solution, consider a filtered probability space $(\Omega,\mA,\mF=(\mF)_{t\geq 0}, \bbP)$ supporting a $\mF$-Wiener process $W$ and define the random variable $Z_t:=\int_0^t R_\mu(t,s)\,dW_s$. Consider the SDE with random coefficients
    \begin{align}
    \label{eq: random coefficients SDE}
        dY_t=-U'(Y_t)\,dt+Z_t\,dt+dW_t,
    \end{align} 
    with $Y_0\sim \mu_0$. This SDE has locally Lipschitz coefficients, and thus it has a unique strong solution adapted to the filtration of $(Y_0,Z,W)$ (which is just the filtration of $(Y_0,W)$) up to some explosion time. By \cite[Theorem 2.5]{Arous1995-zj} and the above analysis there exists a unique weak solution on arbitrary time horizons.
\end{proof}
This representation is better suited for our analysis than the one in \eqref{eq: limit dynamics 0} because, as observed in \cite{Faugeras2025-bp}, it makes the role of the Wiener process clear. It moreover allows us to exploit the stability properties of $R_\mu$ given in \eqref{eq: R contiuity}.

\subsubsection{On the form of the limit dynamics} \label{ss: form of limit dynamics}
It may not be obvious why the limit dynamics should take the form \eqref{eq: limit dynamics 0} and \eqref{eq: limit dynamics 1} and how the kernels $K_\mu, H_\mu$, and $R_\mu$ arise. In this section, we provide a formal interpretation of $H_\mu$ as a conditional expectation operator. This connection is clarified further in the next section, where we derive a form of the averaged dynamics $\bbE[P_{[T]}^N(\bfJ)]$. 

Fix $\beta  \in (0, \infty)$ and $t \in [0, T]$. Let $(\Omega, \mF)$ be a probability space supporting a triple of continuous process $(X, B, G)$. Let $\bbP_0$ denote a probability measure under which $B$ is a Brownian motion, $X$ satisfies
\[dX_t = -U'(X_t)\, dt + dB_t, \]
and $G$ be a centered Gaussian process independent of $(X, B)$ such that the covariance of $G$ is $K_{\mu}$. Since $K^t_{\mu}$ is a positive semi-definite Hilbert-Schmidt operator, we may diagonalize it to obtain
\begin{equation}
K_\mu(u, s) = \sum_{i \in \bbN} \lambda_i g_i(u) g_i(s), \quad s, u \in [0, t],
    \label{eq: diaogonalize K}
\end{equation}
where $(\lambda_i) \subset \R_+$ and $(g_i)$ forms an orthonormal family of $L^2([0, t])$. In particular by \eqref{eq: H_mu^t}, we have
\begin{equation}
H_\mu^t f(s)  = \sum_{i \in \bbN} \frac{\lambda_i}{1 + \beta^2\lambda_i} \bigg[\int_0^t g_i(u) f(u) du\bigg] g_i(s) , \quad s, u \in [0, t],
    \label{eq: diaogonalize H}
\end{equation}
The process $G$ may then be represented as
\begin{equation}
  G_s = \sum_{i \in \bbN} Z_i \sqrt{\lambda_i} g_i(s), \quad s \in [0, t], \label{eq: G} 
\end{equation}
for a family of i.i.d. standard normal random variables $(Z_i)$ which are independent of $(X, B).$

Define the measure $\bbP_\beta$ given by
\[\frac{d\bbP_\beta}{d\bbP_0} = \exp\bigg( \beta\int_0^t G_s dB_s - \frac{\beta^2}{2}\int_0^t |G_s|^2 ds\bigg).  \]
By Girsanov's theorem this implies that under $\bbP_\beta$, the processes $(X, B)$ satisfies
\[dX_s = - U'(X_s) ds +dB_s, \quad dB_s = \beta G_s dt + dW_s, \quad s \in [0, t],\]
where $W_t$ is a standard Brownian motion. Formally, we may write $\frac{d\bbP_\beta}{d\bbP_0}$ as a tilt of the family of standard normals $(Z_i)$:
\[\frac{d\bbP_\beta}{d\bbP_0} = \exp\bigg( \beta \sum_{i \in \bbN} Z_i \bigg( \sqrt{\lambda_i}\int_0^t g_i(s) dB_s\bigg) - \frac{\beta^2}{2}\sum_{i \in \bbN}\lambda_i Z_i^2  \bigg).  \]
Recall that the law of a finite-dimensional Gaussian random variable $Z \sim \mN(0, \Sigma)$ under the tilt $e^{m \cdot Z - \frac{1}{2}\langle Z, AZ \rangle }$ has mean $(\Sigma + A)^{-1} m$. Applying this to the above discussion, we see that the conditional law of $Z$ given $X$ (and hence $B$) takes the form
\begin{equation*}
    \bbE\big[Z_i|X_{[t]} \big] = \frac{\beta \sqrt\lambda_i}{1 + \beta^2 \lambda_i} \int_0^t g_i(s) dB_s.
\end{equation*}
In particular by \eqref{eq: G} and \eqref{eq: diaogonalize H}, this implies
\[ \bbE\big[G_t|X_{[t]}\big]  = \sum_{i \in \bbN} \frac{\beta\lambda_i}{1 + \beta^2 \lambda_i}\bigg[\int_0^t g_i(s) dB_s\bigg]  g_i(t) = \beta\int_0^t H^t_\mu(t, s) dB_s. \]
In light of the mimicking theorem, see Theorem \ref{t:mimicking} below, we may interpret \eqref{eq: limit dynamics 0} as the marginal dynamics of $X$ under $\bbP_\beta$, and the drift of $B$ is simply the mean of the process $G$ conditioned on $X$.

\subsection{Representation of the averaged dynamics} Let $\bfG$ be a $N\times N$ random matrix with i.i.d. standard Gaussian entries. We now present a characterization of the averaged dynamics $\bbE[P_{{T}}^N(\bfG)]$, which is the marginal distribution of $X$ in \eqref{eq: intro quenched dynamics} if the disorder is Gaussian. Our key tool for this is the \emph{mimicking theorem}, which allows one to represent the marginal distribution of an It\^{o} process in terms of a solution to a SDE. This result has appeared in various guises in more general settings, for example {\cite[Corollary 3.11]{Brunick-Shreve-2013-Mimicking}}. We provide a short proof here for convenience.

\begin{theorem}[Mimicking theorem]
    \label{t:mimicking} Let $d \in \bbN$ and $(\Omega, \mA, \mF, \bbP)$ be a filtered probability space supporting a $d$-dimensional $\mF$-Brownian motion $B$ and a $d$-dimensional $\mF$-adapted continuous process $X$ satisfying
\begin{equation}
\label{eq: mimicking equation 1}
    X_t = X_0 + \int_0^t a_sds + B_t,
\end{equation}
where $a$ is a $d$-dimensional progressively measurable $\mF$-adapted process satisfying
\begin{equation*}
    \bbE \bigg[\int_0^t|a_s| ds  \bigg] < \infty, \quad t \in [0, T].
\end{equation*}
Let $(\mF_t^X)_{t\geq 0}$ be the filtration generated by $X$ and define  $\tilde{a}:[0, T] \times \mC^d \rightarrow \R^d$ to be a progressively measurable function satisfying
\begin{equation*}
    \tilde{a}_t(X) = \bbE\big[a_t \mid \mF_t^X\big], \quad \text{a.e. } t \in [0, T], \quad \bbP\text{-a.s.}
\end{equation*}
 There exists a $d$-dimensional $\mF^X$-adapted Brownian motion ${W}$ such that
\begin{equation*}
    X_t = X_0 + \int_0^t \tilde{a}_s(X) ds + {W}_t, \quad t \in [0, T], \quad \bbP\text{-a.s.}
\end{equation*}
\end{theorem}
\begin{proof}
    Define
    \[W_t=X_t-X_0-\int_0^t \bbE\big[a_s \mid \mF_s^X\big]\,ds= B_t+\int_0^t a_s-\bbE\big[a_s \mid \mF_s^X\big]\,ds.\]
The first equality shows that $W$ is adapted to the filtration generated by $X$. We check that $W$ is also a martingale in this filtration. Indeed for $s\leq t$,
\begin{align*}
    \bbE[W_t-W_s \mid \mF_s^X]&=\bbE\Big[X_t-X_s-\int_s^t \bbE[a_u \mid \mF_u^X]\,du\,\big\vert\, \mF_s^X\Big]\\
    &=\bbE\big[B_t-B_s\mid \mF_s^X\big]+\bbE\Big[\int_s^t a_u-\bbE[a_u \mid \mF_u^X]\,du\,\big\vert\, \mF_s^X\Big].
\end{align*}
The second term of the sum is zero by Fubini (which we can apply by the integrability assumption on $a$) and the tower property. For the first term, the tower property, the assumption that $X$ is $\mF$-adapted and the definition of a $\mF$-Brownian motion yield $\bbE\big[B_t-B_s\mid \mF_s^X\big]=\bbE\big[\bbE[B_t-B_s\mid \mF_s]\mid \mF_s^X\big]=0$. Hence $W$ is a martingale with respect to the sigma algebra generated by $X$. On the other hand its quadratic variation is the same as the quadratic variation of $B$, which allows us to conclude that $W$ is a $\mF^X$-Brownian motion. To conclude, simply substitute the definition of $W$ into  \eqref{eq: mimicking equation 1}.
\end{proof}
On applying the Theorem \ref{t:mimicking} to \eqref{eq: intro quenched dynamics}, we obtain a representation of $\bbE[P_{[T]}(\bfJ)]$ as the solution of a non-Markovian SDE.
\begin{lemma} 
\label{lm: gaussian disorder averaged equation}
 Suppose $X$ is a weak solution to \eqref{eq: intro quenched dynamics} with Gaussian disorder $\bfJ=\bfG$. Then $X$ satisfies the following SDE, which admits a unique strong solution:
\begin{equation}
\begin{aligned}
dX_t^i &= - U'(X_t^i) dt + \frac{\beta^2}{N} \bigg[X_t^{\top} \int_0^t Q_s X_s dW^i_s \bigg] dt + dW^i_t, \quad i =1,\ldots, N, \quad t \in (0, T), \\
Q_t&=\Big(I+\frac{\beta^2}N \int_0^t X_sX_s^\top \,ds\Big)^{-1},\\
    X_0 &\sim P_0^N,
\end{aligned}
    \label{eq: averaged dynamics 1}
\end{equation}
where $W = (W^i)_{i = 1}^N$ is a standard $N$-dimensional Brownian motion with respect to the filtration generated by $X$.
\label{lem: averaged dynamics 0}
\end{lemma}

\begin{remark}\label{rk: markovian}
 By keeping track of an additional variable, one may write \eqref{eq: averaged dynamics 1} in a Markovian form. In particular, define the $N$-dimensional stochastic processes $Z^i$ for $i=1,\dots,N$ by
\[Z_t^i:= \frac{\beta}{\sqrt N}\int_0^t Q_s X_s\,dW_s^i.\]
Then it is easy to see from \eqref{eq: averaged dynamics 1} and It\^{o}'s formula that the triple $(X,Z,Q)$ solves the following Markovian SDE: 
    \begin{align}
    \label{eq: averaged dynamics Markovian}
    \begin{split}
        dX_t^i&=-U'(X_t^i)\,dt+\frac{\beta}{\sqrt N}X_t \cdot Z_t^i\,dt+dW_t^i\\
        dZ_t^i&=\frac{\beta}{\sqrt N}Q_tX_tdW_t^i,\\
        dQ_t&=-\frac{\beta^2}{N}Q_t X_tX_t^\top Q_t\,dt,
    \end{split}
    \end{align}    
    with initial conditions $Q_0=I$, $Z_0=0,$ and $ X_0\sim P_0^N$. We point out that $Z_t^i=\bbE[\bfG_i\mid \mF_t^X]$. 
\end{remark}

\begin{proof}[Proof of Lemma \ref{lem: averaged dynamics 0}] First, we use Girsanov's theorem to obtain a tractable form of the conditional expectation $\bbE[G_{ij}\mid\mF_t^X]$. Let $(\Omega, \mA,\mF=(\mF_t)_{t \in [0, T]}, \bbP)$ be a probability space supporting a weak solution $X$ to \eqref{eq: intro quenched dynamics} and $\bbQ$ be the probability measure under which $X$ solves the SDE:
\begin{align*}
    dX_t^i=-U'(X_t^i)\,dt+dV_t^i.
\end{align*}
Here $V = (V^i)_{i = 1}^N$ is a $\mF$-Brownian motion independent of $\bfG$ under $\bbQ$. In particular, letting $\bfG_i$ denote the $i$-th row of $\bfG$, we have 
\[  V_t^i= B_t^i  +\int_0^t\frac\beta{\sqrt{N}} X_s \cdot \bfG_i\,ds. \]
It is easy to verify that in this setting, Novikov's condition is satisfied and thus by Girsanov's theorem and \eqref{eq: intro quenched dynamics}, we have
\begin{align*}
    \frac{d\bbP}{d\bbQ}[t]:= \bbE\bigg[\frac{d\bbP}{d\bbQ}\,\bigg|\,\mF_t\bigg] =\exp\bigg(\frac{\beta}{\sqrt N}\sum_{i=1}^N \int_0^t  \bfG_{i} \cdot X_s\,dV_s^i-\frac{\beta^2}{2N} \sum_{i=1}^N\int_0^t |\bfG_i \cdot X_s|^2\,ds \bigg).
\end{align*}
In particular, under $\bbP$ the process $B$ is a Brownian motion independent of $\bfG$, and we have 
\begin{align}
\label{eq: quenched under P with V}
\begin{split}
    dX_t^i&=-U'(X_t^i)\,dt+\frac\beta{\sqrt N} \bfG_i \cdot X_t\,dt+dB_t^i.
\end{split}
\end{align}

Since under $\bbQ$, the disorder $\bfG$ and the process $V$ are independent, we may factor out $\bfG$ from the stochastic integral to obtain that $\bbQ$-almost surely
\begin{align*}
    \frac{d\bbP}{d\bbQ}[t]&=\prod_{i=1}^N \exp\bigg(\frac{\beta}{\sqrt N}\bfG_i \cdot \int_0^t   X_s\,dV_s^i-\frac{\beta^2}{2N}\bfG_i^\top\bigg(\int_0^t   X_s X_s^\top \,ds\bigg)\bfG_i  \bigg).
\end{align*}
Now by Bayes' formula,
\begin{align*}
    \bbE\Big[\bfG_i \mid \mF_t^X\Big]= \frac{\bbE_\bbQ\Big[\bfG_i \frac{d\bbP}{d\bbQ}[t]\mid \mF_t^X\Big]}{\bbE_\bbQ\Big[\frac{d\bbP}{d\bbQ}[t]\mid \mF_t^X\Big]}= \frac{\bbE_\bbQ\Big[\bfG_i e^{\frac{\beta}{\sqrt N}\bfG_i \cdot \int_0^t   X_s\,dV_s^i-\frac{\beta^2}{2N} \bfG_i^\top\int_0^t   X_s X_s^\top \,ds\bfG_i}\mid \mF_t^X\Big]}{\bbE_\bbQ\Big[e^{\frac{\beta}{\sqrt N}\bfG_i \cdot \int_0^t   X_s\,dV_s^i-\frac{\beta^2}{2N} \bfG_i^\top\int_0^t   X_s X_s^\top \,ds\bfG_i}\mid \mF_t^X\Big]}. 
\end{align*}
In the last display, we used the fact that under $\bbQ$, the columns $\bfG_i$ and $(\bfG_j)_{j\neq i}$ are conditionally independent given $X$ to cancel out the terms not involving $\bfG_i$ (indeed, $\bfG_i, \bfG_j$, and $X$ are independent). Since $V_t$ is $\mF_t^X$ measurable, we may treat the stochastic integral as a fixed parameter to compute the conditional distribution of $\bfG$. 
 A standard Gaussian calculation then yields:
\begin{align}
\label{eq: conditional expectation under P}
    \bbE\Big[\bfG_i \mid \mF_t^X\Big]=\frac{\beta}{\sqrt N}\bigg(I+\frac{\beta^2}{N}\int_0^t X_sX_s^\top\bigg)^{-1} \int_0^t X_s\,dV_s^i=\frac{\beta}{\sqrt N}Q_t \int_0^t X_s\,dV_s^i.
\end{align}
Applying Theorem \ref{t:mimicking} to  \eqref{eq: quenched under P with V} with the $\mF_t^X$-Brownian motion \[W_t:= X_t-X_0-\frac\beta{\sqrt N}\int_0^t \bbE[\bfG X_s \mid \mF_s^X]\,ds,\] and applying \eqref{eq: conditional expectation under P} we find that under $\bbP$, the processes $X$ and $V$ satisfy
\begin{align}
\begin{split}
\label{eq: mimicking incomplete}
    dX_t^i&=-U'(X_t^i)\,dt+\Big[\frac{\beta^2}{N}  X_t^\top Q_t \int_0^t X_s  \,dV^i_s\,\Big]dt+dW_t^i,\\
    dV_t^i&=\Big[\frac{\beta^2}{N}  X_t^\top Q_t \int_0^t X_s  \,dV^i_s\,\Big]dt+dW_t^i.
\end{split}
\end{align}
The second line of the last display, together with an application of It\^{o}'s formula to $Q_t\int_0^t X_s\,dV_s^i$, implies
\begin{align*}
    dQ_t\int_0^t X_s\,dV_s^i&=-\frac{\beta^2}N Q_t X_tX_t^\top Q_t \int_0^t X_s\,dV_s^i\,dt+Q_tX_t\,dV_t^i =Q_t X_t\,dW_t^i.
\end{align*}
Substituting this into the first line of \eqref{eq: mimicking incomplete} yields the desired form of the dynamics:
\begin{align*}
    dX_t^i=-U'(X_t^i)\,dt+\bigg[\frac{\beta^2}N X_t^\top \int_0^t Q_s X_s\,dW_s^i\bigg]\,dt+dW_t^i.
\end{align*}

To conclude, we check that the SDE has a unique strong solution. We have just proved that it does have a weak solution. As discussed in Remark \ref{rk: markovian},  \eqref{eq: averaged dynamics 1} may be rewritten a Markovian SDE \eqref{eq: averaged dynamics Markovian} with locally Lipschitz coefficients. By a standard argument using local Lipschitz property, pathwise uniqueness holds for \eqref{eq: averaged dynamics Markovian}, see \cite[theorem 5.2.5]{Karatzas_and_Shreve_BMSC}. Thus by the Yamada-Watanabe theory (\cite[Corollary 5.3.25]{Karatzas_and_Shreve_BMSC}), \eqref{eq: averaged dynamics Markovian} and hence \eqref{eq: averaged dynamics 1} has a unique in law strong solution $X$, which is adapted to the filtration of $(W,X_0)$. 
\end{proof}

\subsection{A second representation of the averaged dynamics} In this section, we obtain yet another form of the averaged dynamics similar to that of the limit dynamics presented in \eqref{eq: limit dynamics 1}. Such a representation will facilitate the coupling argument at the heart of the proof of Theorem \ref{th: averaged propagation of chaos}. The following lemma, which is proved in Appendix \ref{sec: Volterra}, shows that the kernels of Section \ref{ss: limit-rep} appear naturally in the SDE for the averaged dynamics derived in Lemma \ref{lem: averaged dynamics 0}. 
\begin{lemma} Define $\hat{\mu}^N = N^{-1} \sum_i \delta_{X^i}$. With $Q$ as in \eqref{eq: averaged dynamics 1}, for $u, s, t \in [0, T]$ we have
\[H^t_{\hat{\mu}^N}(u, s) = \frac{1}{N}X_u^\top Q_t X_s, \quad R_{\hat{\mu}^N}(t, s)=  \frac{1}{N}X_t^\top Q_s X_s. \]
\label{lem: N kernels}
\end{lemma}
It is tempting now to apply Lemma \ref{lem: N kernels} to \eqref{eq: limit dynamics 1}, which suggests for instance that we have $dX^i_t = -U'(X_t^i) dt + \beta^2 \int_0^t R^t_{\hat{\mu}^N}(t, s) dW_s^i dt + dW_t^i$. However, this is incorrect since
\[X_t^\top  \int_0^t Q_s X_s dW^i_s \neq \int_0^t X_t^\top Q_s X_s dW^i_s. \]
In fact, the right hand side is not well-defined as an It\^{o} integral and must be understood in terms of a Skorokhod integral since the processes $s\mapsto X_t Q_s X_s$ and $s\mapsto W_s$ cannot be adapted to the same filtration. Fortunately, one may recover a precise statement through Malliavin calculus, and we will show in the next section that
\[
\begin{aligned}
dX_t^i &= - U'(X_t^i) dt + \beta^2 \bigg[ \int _0^t R_{\hat{\mu}^N}(t, s) dW^i_s \bigg] dt + \frac{\beta^2}{N}\Lambda^i(t)dt +  dW^i_t,
\end{aligned}
\]
where the stochastic integral is understood in the Skorokhod sense and $\Lambda = (\Lambda^i)_{i =1}^N$ is a progressively measurable function of $X$ and its Malliavin derivatives satisfying \[\bbE\big[|\Lambda_N^i|^2\big] = O(N).\]
We review the necessary tools of Malliavin calculus and provide a precise statement of this claim in Proposition \ref{pr: representation of averaged dynamics} below.
\begin{remark}[On the case of symmetric disorder] \label{rk: symmetric}
    In the symmetric case and without the use of Malliavin calculus, \cite[Theorem 3.3]{Guionnet-1997-PoC} derived a similar representation for the averaged particle law:
    \[dX_t^i=-U'(X_t^i)\,dt+F_t(\hat\mu^N)\,dt+\frac1N G_t(\hat\mu^N)\,dt+dW_t^i,\]
    where $F,G$ are progressively measurable functions and $G$ is bounded. It is not easy to see how $G_t$ and $\Lambda_t$ are related, if at all. Moreover, the function $\mu\mapsto F_t(\mu)$ appears to have poor stability properties. For instance, it does not appear to be uniformly Lipschitz  in $\mu$ (see \cite[Lemma 3.9]{Guionnet-1997-PoC}), which makes it hard to control the behavior of $F$ near the random measure $\hat \mu^N$. 
\end{remark}

\subsubsection{Malliavin calculus} 
The theory of Malliavin calculus is rich and we make no attempt to fully summarize it here. We will sketch only the relevant facts and point out the classic reference \cite{Nualart2005Malliavin}. Let $(\Omega,\mF,\bbP)$ be a probability space supporting an $N$-dimensional Wiener process $W$ under the assumption that $\mF$ is generated by $W$. We denote by $(\mF_t)_{t\geq 0}$ the augmented Wiener filtration. The goal is to analyze the variation of a random variable $X:\Omega\to \R$ seen as a function on Wiener space as the input $\omega\in \Omega$ is perturbed. This is first done for smooth random variables of the form 
\[X=F\Big(\int_0^T h^1_t\cdot dW_t,\dots,\int_0^T h_t^k\cdot dW_t\Big),\]
where $F:\R^{k} \to \R$ is a smooth function with polynomial growth, $h_1,\dots,h_k \in L^2([0,T];\R^N)$ and $k\in \bbN$ is an arbitrary natural number. For such random variables we can define the Malliavin derivative to be the stochastic process $(D_t X)_{t\geq 0} \in (\R^N)^{[0,T]}$
\begin{equation*}
    D_t X=\sum_{i=1}^k \partial_i F\Big(\int_0^T h^1_t\cdot dW_t,\dots,\int_0^T h_t^k\cdot dW_t\Big) h^i_t.
\end{equation*}
This defines a linear operator on $L^2(\Omega;\R)$ on the class of smooth random variables. We will denote by $D_t^i X$ the $i$-th coordinate of $D_t X$. This linear operator is closeable with respect to the norm 
\[\|X\|_{\bbD^{1,2}}^2:= \bbE\big[|X|^2\big]+\int_0^T\bbE\Big[|D_t X|^2\Big]\,dt.\]
We will denote by $\bbD^{1,2}$ the domain of its closure. Of course the same procedure can be carried out for a vector valued stochastic processes $(X_t)_{t\geq 0}:\Omega\times[0,T]\to \R^N$. We will denote by $\bbL^{1,2}$ the set of stochastic processes in $L^2(\Omega,[0,T];\R^N)$ such that for each coordinate $i, \, X_t^i \in \bbD^{1,2}$ for almost every $t\leq T$, endowed with the norm
\[\|X\|_{\bbL^{1,2}}^2= \int_0^T \bbE\big[|X_t|^2\big]\,dt+\int_0^T \int_0^T \bbE\big[\|D_s X_t\|^2_{\text{Fr}}\big]\,ds\,dt.\]
We recall that both $\bbD^{1,2}$ and $\bbL^{1,2}$ are Hilbert spaces. An important feature of the Malliavin derivative is that it respects the time dependence of stochastic processes in the following way: if $X\in L^2(\Omega)$ is a random variable that is measurable with respect to $\mF_t$, then for every $s>t$, 
\begin{equation}
\label{eq: malliavin derivative not anticipative}
    D_s X=0 \text{ a.s.}
\end{equation}
In particular, for an adapted stochastic process $(X_t)_{t\geq 0}$, $D_s X_t=0$ almost surely for $s>t$. In other words, if $X_t$ is an adapted process perturbing the driving Brownian motion at a future time $t+h$ does not influence the value of $X_t$.
\par Another important object is the adjoint of the Malliavin derivative. Given a process $X\in L^2(\Omega,[0,T];\R^N)$, if there exists a random variable $\delta(X)\in L^2(\Omega;\R)$ such that for every $Y\in \bbD^{1,2}$ we have 
\begin{align*}
    \bbE\Big[ \int_0^T D_t Y \cdot X_t \,dt\Big]= \bbE\big[Y \cdot \delta(X)\big],
\end{align*}
the random variable $\delta(X)$ the \emph{Skorokhod integral} of the process $X$. Skorkohod integration is an extension of It\^{o} integration (see \cite[Proposition 1.3.18]{Nualart2005Malliavin}): if $X \in L^2(\Omega,[0,T])$ is an adapted $\R^N$ valued stochastic process, then $X$ is in the domain of $\delta$ and 
\[\delta(X)=\int_0^T X_t \cdot \,dW_t.\]
In fact, we will abuse of notation in the sequel and write  $\delta(X)=\int_0^T X_t\cdot\,dW_t$ whenever $X$ is a stochastic process in the domain of $\delta$.

Two key facts about the Skorokhod integral will be utilized in the proof of Theorem \ref{th: averaged propagation of chaos}. The first is the following extension of the It\^{o} isometry to Skorokhod integrals, see \cite[Proposition 1.3.1]{Nualart2005Malliavin}: if $X$ is in $\bbL^{1,2}$, we have
\[\bbE\bigg[\bigg|\int_0^T X_t\,\cdot dW_t\bigg|^2\bigg]=\bbE\bigg[\int_0^T |X_t|^2dt\bigg]+\bbE\bigg[\int_0^T\int_0^T \text{Tr}(D_s X_t D_t X_s) dt\,ds\bigg].\]
The Skorokhod integral allows us to bring a (non-adapted) random variable inside the stochastic integral.
Consider a scalar random variable $Y\in \bbD^{1,2}$ and a process $X$ in the domain of $\delta$ such that $YX$ is also in the domain of $\delta$. Then by \cite[Proposition 1.3.3]{Nualart2005Malliavin}, we have
\begin{equation}
    Y  \int_0^T X_t \cdot \,dW_t=\int_0^T Y X_t\cdot \,dW_t+\int_0^T D_t Y \cdot X_t\,dt.
    \label{eq: Malliavin commute}
\end{equation}

\subsubsection{Malliavin derivative estimates}
Malliavin differentiability of the solution to an SDE with regular coefficients is by now classical, see for instance \cite[Chapter 2.2]{Nualart2005Malliavin}. The issue in the present context is that the coefficients are smooth but not Lipschitz, and thus the Malliavin derivatives may not be integrable. Notice that there are multiple terms of \eqref{eq: averaged dynamics Markovian} that may cause trouble. In particular, the potentially explosive confining term $U$ and the interaction term $X\cdot Z^i$ are not globally Lipschitz. The problem of Malliavin differentiability of locally Lipschitz SDEs has been attacked in \cite{MalliavinDifferentiabilitySuperlinear}, but the key one sided bound that they require does not hold in our context. Nevertheless, we obtain the following crucial estimate. 

\begin{proposition}[Malliavin derivative estimates] 
Let $X$ be the unique strong solution to \eqref{eq: averaged dynamics 1}. Then $X \in \bbL^{1, 2}_T$ and there exists $C \in (0, \infty)$ depending only on $(A, \beta, T, C_{\textup{Lip}})$ such that for arbitrary $i=1,\dots,N$, 
\begin{equation*}
     \bbE\big[|D_u^i X_t|^2\big]+\int_0^t \int_0^t \bbE\big[|D_r^i X_s|^2\big]\,dr\,ds\leq C, \quad 0 \leq u \leq t \leq T.
\end{equation*}
Moreover with $Q$ as in \eqref{eq: averaged dynamics 1}, we have
\begin{equation*}
    \bbE\big[\|D^i_s Q_t\|^2_{\op}\big] \leq CN^{-1}, \quad i = 1, \ldots, N, \quad 0 \leq s \leq t \leq T.
\end{equation*}
\label{pr: Malliavin calculus estimates}
\end{proposition}

The proof of Proposition \ref{pr: Malliavin calculus estimates} is important but quite technical, and thus we defer its proof to Section \ref{sect: malliavin}. As a consequence of Proposition \ref{pr: Malliavin calculus estimates}, we may apply \eqref{eq: Malliavin commute}  Lemma \ref{lem: N kernels} to \eqref{eq: averaged dynamics 1} to obtain the following result, which plays an essential role in the proof of Theorem \ref{th: averaged propagation of chaos}.

\begin{proposition}[Representation of averaged dynamics]
    \label{pr: representation of averaged dynamics}
    Suppose $\textup{Law}(X)=\bbE[P_{[T]}^N(\bfJ)]$. With $Q$ as in \eqref{eq: averaged dynamics 1}, define
    \[\Lambda^i(t) := \int_0^t D^i_s X_t^\top  Q_s X_s\, ds, \quad i = 1, \ldots, N, \quad t \in [0, T]. \]
    Then $X$ satisfies
\begin{equation}
\begin{aligned}
dX_t^i &= - U'(X_t^i) dt + \beta^2 \bigg[ \int _0^t R_{\hat{\mu}^N}(t, s) dW^i_s \bigg] dt + \frac{\beta^2}{N}\Lambda^i(t)dt +  dW^i_t, \quad i = 1, \ldots, N, \quad t \in (0, T).\\
\end{aligned}
    \label{eq: averaged dynamics}
\end{equation}
Moreover, for $0 \leq s \leq t \leq T$ we have $R_{\hat{\mu}^N}(t, s) = \frac{1}{N} X_t^{\top} Q_s X_s.$
\end{proposition}

Moreover, Proposition \ref{pr: Malliavin calculus estimates} together with the trivial bounds $\|Q_t\|_{\op} \leq 1$ and $|X_t| \leq \sqrt{NA}$ for $t \in [0, T]$, immediately yields the following corollary. 

\begin{corollary}\label{cor: Lambda-estimate} Let $\Lambda = (\Lambda^i)_{i = 1}^N$ be as in Proposition \ref{pr: representation of averaged dynamics}. Then there exists $C \in (0, \infty)$ depending only on $(A, \beta, T, C_{\textup{Lip}})$ such that 
\[    \bbE\big[|\Lambda^i(t)|^2\big] \leq CN, \quad i = \{1, \ldots, N\}, \quad t \in [0, T]. \]
\end{corollary}

\section{Quantitative averaged propagation of chaos}
\label{sect: averaged propagation of chaos}

\subsection{Proof of Theorem \ref{th: averaged propagation of chaos} for Gaussian disorder}
In this section, we prove Theorem \ref{th: averaged propagation of chaos} in the case $\bfG$ is a matrix of i.i.d standard Gaussians. Recall that $C \in (0, \infty)$ denotes a finite constant depending on the global parameters $(A, \beta, T, C_{\textup{P}}, C_{\textup{Lip}}, C_{0})$ which may change from line to line.

We begin by coupling the limit dynamics \eqref{eq: limit dynamics 1} to the averaged dynamics \eqref{eq: averaged dynamics}. By Lemma \ref{lem: limit dymamics}, there exists a filtered probability space $(\Omega,\mA, \bbP)$ supporting an $N$ dimensional Wiener process $(W^i)_{i = 1}^N$ adapted to its augmented natural filtration. After potentially enlarging the probability space, we may assume the space also supports random variables $(Y_0^i)_{i=1}^N \sim \mu_0^{\otimes N}$ and $(X_0^i)_{i=1}^n \sim P_0^N$ independent of $W$. Furthermore, we assume $(X^i_0)_{i =1}^N$ and $(Y_0^i)_{i = 1}^N$ are coupled according to the Wasserstein coupling of $P_0^N$ and $\mu_0^{\otimes N}$. In particular, this implies for all $i \in \{1, \ldots, n\}$, we have
\begin{equation}
    \bbE\big[|X_0^i - Y_0^i|^2 \big] = \bbW_2^2(P_0^{N, i}, \mu_0)  \leq \frac{C_0}{N}.
    \label{eq: X_0 - Y_0}
\end{equation}
By Lemma \ref{lem: limit dymamics} and Proposition \ref{pr: representation of averaged dynamics}, \eqref{eq: limit dynamics 1}, \eqref{eq: averaged dynamics} have unique strong solution with the given Brownian motion $W$.

By definition of Wasserstein distance on path space and exchangeability, we have
\begin{equation}
    \bbW_2^2(\bbE[P_{[T]}^k(\bfG)], \mu^{\otimes k}_{[T]}) \leq  \bbE\Big[ \sup_{t \in [0, T]}\sum_{i = 1}^k|X_t^i - Y_t^i|^2 \Big] \leq k\bbE\Big[ \sup_{t \in [0, T]}|X_t^1 - Y_t^1|^2 \Big].
    \label{eq: averaged 1st W bound}
\end{equation}

It remains to bound the right hand side. By \eqref{eq: averaged dynamics}, \eqref{eq: limit dynamics 1}, and the chain rule, we have
\begin{equation*}
    \begin{aligned}
    |X_t^1 - Y_t^1 |^2 = |X_0^1 - Y_0^1|^2 &- 2\int_0^t (X_s^1 - Y_s^1)\big( U'(X_s^1) - U'(Y_s^1) \big) ds \\
    &  + 2\beta^2 \int_0^t(X_s^1 - Y_s^1) \bigg[\int_0^s(R_{\hat{\mu}^N} - R_{\mu})(s, r)\, dW_r^1\bigg] ds \\
    & + \frac{2\beta^2}{N} \int_0^t(X_s^1 - Y_s^1) \Lambda^1(s) ds.
\end{aligned} 
\end{equation*} 
Applying Young's inequality yields
\[    \begin{aligned}
    |X_t^1 - Y_t^1 |^2 \leq |X_0^1 - Y_0^1|^2 &- 2\int_0^t (X_s^1 - Y_s^1)\big( U'(X_s^1) - U'(Y_s^1) \big) ds + 2\beta^2 \int_0^t |X_s^1 -Y_s^1|^2 ds \\
    &  + \beta^2 \int_0^t\bigg|\int_0^s(R_{\hat{\mu}^N} - R_{\mu})(s, r)\, dW_r^1\bigg|^2 ds + \frac{\beta^2}{N^2} \int_0^t|\Lambda^1(s)|^2 ds .
\end{aligned}  \]
By Corollary \ref{cor: Lambda-estimate}, we have that $\int_0^T \bbE[|\Lambda^1(t)|^2] dt \leq CN$. This, together with \eqref{eq: X_0 - Y_0} implies 
\begin{equation}
\bbE\Big[\sup_{t \in [0, T]}|X_t^1 - Y_t^1|^2\Big] \leq \frac{C}{N} + \Theta_1 + \Theta_2,
\label{eq: X - Y estimate 1}
\end{equation}
where
\begin{align*}
    \Theta_1 &:= \bbE\bigg[\sup_{t \in [0, T]}\bigg(- 2\int_0^t (X_s^1 - Y_s^1)\big( U'(X_s^1) - U'(Y_s^1) \big) ds \bigg) + 2 \beta^2 \int_0^T|X_t^1 - Y_t^1|^2 dt\bigg], \\
    \Theta_2 &:= \beta^2 \bbE\bigg[\int_0^T \bigg|\int_0^t(R_{\hat{\mu}^N} - R_{\mu})(t, s) dW_s^1\bigg|^2 dt\bigg].
\end{align*}
We proceed by bounding each $\Theta_i$ by terms which either have the correct decay in $N$ or are controlled by $\bbE\int_0^T\sup_{t \in [0, T]}|X_t^1 - Y_t^1|^2  dt$.

By Assumption \ref{assumption: main}, $U = U_c + U_l$ where $U_c$ is convex and $U'_\ell$ is $C_{\textup{Lip}}$-Lipschitz. Thus we have
\[\begin{aligned}
    &\sup_{t\leq T}\Big(-2\int_0^t (X_s^1-Y_s^1)(U'(X_s^1)-U'(Y_s^1))\,ds\Big)\\
    \leq &\sup_{t\leq T}\Big(-2\int_0^t (X_s^1-Y_s^1)(U'_\ell(X_s^1)-U'_\ell(Y_s^1))+(X_s^1-Y_s^1)(U'_c(X_s^1)-U'_c(Y_s^1))\,ds\Big)\\
 \leq &C_{\textup{Lip}}\int_0^T \sup_{s \in [0, t]}|X_s^1-Y_s^1|^2\,dt    
\end{aligned} \]
The term $\Theta_1$ is therefore easy to estimate:
\begin{equation}
    \Theta_1 \leq C \int_0^T\bbE\Big[\sup_{s\in[0, t]}|X_s^1 - Y_s^1|^2\Big] dt.
    \label{eq: Xi_1 final estimate}
\end{equation}

To estimate $\Theta_2$, we use the $\bbL^{1, 2}$-isometry \cite[Proposition 1.3.1]{Nualart2005Malliavin} to obtain
\begin{equation}
\begin{aligned}
        \Theta_2 =\beta^2\bbE \bigg[\int_0^T\int_0^t&|R_{\hat{\mu}^N}(t,s)-R_\mu(t,s)|^2\,ds\,dt\bigg] \\
        &+\beta^2 \int_0^T \bbE\bigg[\int_0^t\int_0^t D_u^1 R_{\hat{\mu}^N}(t,s)D_s^1 R_{\hat{\mu}^N}(t,u)\,ds\,du\bigg]\,dt
\end{aligned}
\label{eq: Xi_2 decomposition}
\end{equation}
By Lemma \ref{lem: resolvent} and the definition of $K$, we have
\[ \begin{aligned}
    \bbE\bigg[ \int_0^T\int_0^t&|R_{\hat{\mu}^N}(t,s)-R_\mu(t,s)|^2\,ds\,dt\bigg] \leq  C \int_0^T \int_0^T \bbE\big[  \big|K_{\hat{\mu}^N}(t, s)-K_\mu(t,s)\big|^2\big] ds\,dt\\
    &\leq C \int_0^T \int_0^T \bbE\bigg[  \bigg|\frac 1N\sum_{i=1}^N X_t^i X_s^i-Y_t^iY_s^i\bigg|^2\bigg] +\bbE\bigg[  \bigg|\frac 1N\sum_{i=1}^NY_t^i Y_s^i-K_\mu(t,s)\bigg|^2\bigg]ds\,dt
\end{aligned}\]
By exchangeability and the fact that $\|Y^i\|_\infty$ and $\|X^i\|_\infty$ are bounded by $A$, we have
\[ \int_0^T \int_0^T \bbE\bigg[  \bigg|\frac 1N\sum_{i=1}^N X_t^i X_s^i-Y_t^iY_s^i\bigg|^2\bigg] \leq 2A^2T\int_0^T\bbE[\sup_{s\leq t}|X_s^1-Y_s^1|^2]\,dt\]
Moreover, the processes $(Y^i)_{i = 1}^N$ are independent, and thus
\[\bbE\bigg[  \bigg|\frac 1N\sum_{i=1}^NY_t^i Y_s^i-K_\mu(t,s)\bigg|^2\bigg] = \frac{1}{N^2}\sum_{i = 1}^N \bbE\big[|Y_t^iY_s^i -K_\mu(t, s)|^2 \big] \leq \frac{C}{N}. \]
Combining the last 3 displays together yields
\begin{equation}
    \bbE\bigg[ \int_0^T\int_0^t|R_{\hat{\mu}^N}(t,s)-R_\mu(t,s)|^2\,ds\,dt\bigg]\leq C\bigg( \int_0^T \bbE\Big[\sup_{s \in [0, t]}|X_s^1 - Y_s^1|^2\Big] dt  + \frac{1}{N}\bigg).
    \label{eq: Xi_2 term 1}
\end{equation}

Next we estimate the second term of \eqref{eq: Xi_2 decomposition}. By Cauchy-Schwarz, we have
\[\begin{aligned}
    \bbE\bigg[\int_0^t\int_0^t D_u^1 R_{\hat{\mu}^N}(t,s)D_s^1 R_{\hat{\mu}^N}(t,u)\,ds\,du\bigg] 
    \leq  \bbE\bigg[\int_0^t \int_0^t |D_u^1 R_{\hat{\mu}^N}(t,s)|^2\, ds\, du\bigg ].
\end{aligned}  \]
Recall from Proposition \ref{pr: representation of averaged dynamics} that $R_{\hat{\mu}^N}(t, s) = \frac{1}{N}X_t^\top Q_s X_s$. By the chain rule for Malliavin derivatives, for $u, s  \in [0, t]$ we have
\[ \begin{aligned}
    |D^1_u R_{\hat{\mu}^N}(t, s) | &= \frac{1}{N} |D^1_u (X_t^\top Q_s X_s)| \\
    & \leq \frac{1}{N}\bigg( \|Q_s\|_{\op} \big(|D^1_u X_t| |X_s| + |X_t| |D^1_u X_s|\big) + |X_t| \|D^1_u Q_s\|_{\op}|X_s| \bigg).
\end{aligned}  \]
Applying Proposition \ref{pr: Malliavin calculus estimates} and the trivial bounds $\|Q_t\|_{\op} \leq 1$ and $|X_t| \leq \sqrt{NA}$ to the last display yields
\[\bbE[|D^1_u R_{\hat{\mu}^N}(t, s) |^2] \leq CN^{-1}, \quad 0 \leq u, s \leq t \leq T. \]
The last three displays together imply that the second term of \eqref{eq: Xi_2 decomposition} is bounded by $CN^{-1}$, and thus by \eqref{eq: Xi_2 decomposition} and \eqref{eq: Xi_2 term 1} we have that
\begin{equation}
    \Theta_2 \leq C\bigg( \int_0^T \bbE\Big[\sup_{s \in [0, t]}|X_s^1 - Y_s^1|^2\Big] dt  + \frac{1}{N}\bigg).
    \label{eq: Xi_2 final estimate}
\end{equation}
Thus by \eqref{eq: X - Y estimate 1}, \eqref{eq: Xi_1 final estimate}, and \eqref{eq: Xi_2 final estimate}, we conclude that
 \[ \bbE\Big[\sup_{t \in [0, T]}|X_t^1 - Y_t^1|^2\Big]  \leq C\bigg( \int_0^T \bbE\Big[\sup_{s \in [0, t]}|X_s^1 - Y_s^1|^2\Big] dt  + \frac{1}{N}\bigg).\]
 The proof concludes on applying Gr\"{o}nwall's inequality and \eqref{eq: averaged 1st W bound}. \hfill $\square$

\section{Proof of Proposition \ref{pr: Malliavin calculus estimates}}
\label{sect: malliavin}
\subsection{Preliminaries and the stochastic Gr\"{o}nwall inequality} We break up the proof of Proposition \ref{pr: Malliavin calculus estimates} in various steps. A priori, it is not obvious that $X\in \bbL^{1,2}$ since the coefficients of \eqref{eq: averaged dynamics Markovian} are not Lipschitz. The goal then is to regularize the SDE, obtain the desired bounds and pass them to the limit. The caveat is that the boundedness of $X_t$ is important to close the estimates, and thus we need to take some care to balance the regularization of the confining potential $U$ and the interaction $X_t\cdot Z_t^i$.

Before presenting the proof of Proposition \ref{pr: Malliavin calculus estimates}, we recall the stochastic Gr\"{o}nwall inequality of \cite{StochasticGronwallScheutzow2013} (for the particular choice of parameters $p=1/2,\mu=5,\nu=5/4$):
\begin{theorem}
\label{th: stochastic gronwall}{\cite[Theorem 4]{StochasticGronwallScheutzow2013}}
    Let $X$ be an adapted, non-negative process with continuous paths and let $h_t\geq0$ be adapted with continuous paths. Assume that $(a_t)_{t\geq 0}$ is a non-negative progressively measurable process and let $M$ be a continuous local martingale with $M_0=0$. If
    \begin{align*}
        X_t\leq \int_0^t a_s X_s\,ds+h_t+M_t
    \end{align*}
    holds for all $t\geq 0$, then
    \begin{align*}
        \bbE\Big[\sup_{s\leq t}\sqrt{X_s}\Big]\leq C\bbE\bigg[\exp\Big(\frac 52 \int_0^t a_s\,ds\Big)\bigg]^{1/5}\bbE\Big[\sup_{s\leq t}|h_s|^{5/8}\Big]^{4/5}
    \end{align*}
\end{theorem}
The strength of this theorem is that the estimate does not depend on the specific local martingale $M$, which is heavily constrained by the requirement that $X$ remains non-negative. This is especially well suited for our purposes since the dynamics \eqref{eq: averaged dynamics Markovian} feature an interaction term $X_t\cdot Z_t^i$ between our process of interest and a martingale $Z_t^i$ that is integrable but not bounded, on top of the usual Brownian motion. As a technical point, the process $X$ lives in a probability space $(\Omega_0\times \widetilde \Omega, \mA_0\times \widetilde \mA,\bbP_0\otimes \widetilde\bbP)$ where $X_0$ is supported on $\Omega_0$ and the driving Brownian motion $W$ is supported on $\widetilde \Omega$. Without loss of generality, $X$ is adapted to the filtration generated by $(X_0,W)$. This allows us to differentiate only in the Wiener direction while ignoring the initial condition.

 \subsection{Step 1} We obtain the estimate for an SDE with Lipschitz coefficients. Take $\zeta:\R \to \R_+$ to be a smooth cutoff function supported on the $2$-ball that is equal to $1$ only the $1$-ball. Define $f_h:\R\to \R$ as $f_h(x)=x \zeta(h^{-1}x)$. Observe that $f_h$ has the following properties:
    \begin{enumerate}
        \item $|f_h(x)|\leq \min\{h,|x|\},$
        \item $| \partial_xf_h(x)|\leq C$ uniformly in $h$.
    \end{enumerate}Pick $\frs>1/A$, let $A_\frs:=A-1/\frs$ and define \begin{align*}
    U^\frs_c(x):=\begin{cases}
        U_c(x) \text{ if } |x|\leq A_\frs,\\
        U_c(A_\frs)+U_c'(A_\frs)(x-A_\frs)+U_c''(A_\frs)\frac{(x-A_\frs)^2}{2} \text{ if } x>A_\frs,\\
        U_c(-A_\frs)+U_c'(-A_\frs)(x+A_\frs)+U_c''(-A_\frs)\frac{(x+A_\frs)^2}{2} \text{ otherwise.}
    \end{cases}
\end{align*}
It is easily checked that for each $r>1/A$, $U_c^\frs$ is convex and $\mathcal C^2$. Moreover, $U_c^\frs$ converges to $U_c$ uniformly on compact subsets of $(-A,A)$ and $ \nabla U^\frs_c$ is Lipschitz. Denote $U^\frs:=U_c^\frs+U_\ell$. We abuse notation by denoting by $\nabla U^\frs(X)$ and $f_h(X)$ the vector where the functions $\nabla U^\frs, f_h$ are applied to every entry:
\[f_h(X):=(f_h(X^i))_{i=1}^N,\quad \nabla U^\frs(X):=(\nabla U^\frs(X^i))_{i=1}^N.\]
 Similarly, if $M$ is an $N\times N$ matrix, we write $f_A(M)$ for the matrix with entries $f_A(M^{ij})$.

 Consider the SDE \begin{align}
\begin{split}
    \label{eq: fully regularized SDE}
         d\widehat X_t^{i}&=-\nabla U^\frs(\widehat X_t^i)\,dt+\frac{\beta}{\sqrt N}\sum_{\ell=1}^N f_A(\widehat X^\ell_t) f_h(\widehat Z_t^{i,\ell})\,dt+dW_t^i\\
        d\widehat Z_t^i&=\frac{\beta}{\sqrt N}\widehat Q_t f_A(\widehat X_t)dW_t^i\\
        d\widehat Q_t&=-F\Big(\frac{\beta^2}N \widehat Q_t f_A(\widehat X_s)f_A(\widehat X_s)^\top \widehat Q_t\Big)\,dt\\
        \widehat Q_0&=I, \widehat Z_0=0, \widehat X_0^i=X_0^i\,\,,\,\, i=1,\dots,N.
\end{split}
\end{align}
where again we take $\widehat X_t^i$ to be one dimensional stochastic processes, while the $\widehat Z^i$s are $N\times 1$ random vectors and $\widehat Q$ is a random $N\times N$ matrix. Of course $\widehat X$ will depend on the values of $r,h$; we do not keep track of this dependency to make notation lighter. This SDE has locally Lipschitz coefficients and thus has a unique strong solution, up to an explosion time. 
For an arbitrary continuous vector field $x:[0,T]\to \R^N$, the matrix ODE
\begin{align*}
    dq_t=-\frac{\beta^2}{N} q_t x_t x_t^\top q_tdt
\end{align*}
with unknown $q:[0,T]\to \R^{N\times N}$ and $q_0=I$ has a unique solution since the coefficients are locally Lipschitz; the solution is given by $q_t=(I+\beta^2/N \int_0^t x_sx_s^\top\,ds)^{-1}$. By unique solvability of locally Lipschitz ODEs we have that 
\[\widehat Q_t=\Big(I+\frac{\beta^2}N \int_0^t f_A(\widehat X_s)f_A(\widehat X_s)^\top\,ds\Big)^{-1}.\]
The Frobenius norm of $\widehat Q$ is uniformly bounded by $2\sqrt N$, and thus we can rewrite the SDE \eqref{eq: fully regularized SDE} to have Lipschitz coefficients by applying a smooth truncation to $\widehat Q$ that will not be different from identity with probability $1$.
\begin{lemma}
\label{lm: entries of mollified matrix}
The entries of $\widehat Q_tf_A(\widehat X_t)$ are bounded by $A^3T$. The same result applies to $Q_tX_t$.
\end{lemma}
\begin{proof}
 Let $D_s=\frac{\beta^2}{N}\int_0^s f_A(\widehat X_u)f_A(\widehat X_u)^\top du$. Then $D_s$ is a symmetric PSD matrix with entries bounded by $\beta^2 sA^2/N$; we can thus factorize it as $D_s= V^\top V$ for some symmetric, psd matrix $V$ with rows $v_1,\dots,v_n$. In particular for each $i$ and $s\leq T$, $|v_i|^2=D_s^{i,i}\leq A^2s/N$.
    By the Morrison inverse formula, 
    \[Q_s=(1+V^\top V)^{-1}=I- V^\top(1+VV^\top)^{-1}V.\] Multipliying from the left by the base vector $e_i$ and the right from $e_j$ (with $i\neq j)$, we get
    \[|Q_{s}^{i,j}|= |v_i^\top (I+VV^\top)^{-1}v_j|\leq |v_i||v_j|\leq \frac{sA^2}{N}.\]
    Now, multiplying by $X_s$ we get
    \begin{align*}
        |(\widehat Q_s f_A(\widehat X_s))^i|= |\sum_{j=1}^N Q_s^{ij}f_A(\widehat X_s^j)|\leq\sum_{j=1}^N |\widehat Q_s^{ij}f_A(\widehat X_s^j)|\leq A^3 s,
    \end{align*}
    as desired. The proof for $Q_t X_t$ is the same.
\end{proof}
\begin{proposition}
\label{pr: bound for fully regularized SDE}
Let $\widehat X$ be a strong solution of \eqref{eq: fully regularized SDE}. Then $\widehat X\in \mathbb L^{1,2}$, and in particular for every $i=1,\dots,N$
\begin{align*}
    \int_0^T \int_0^T \mathbb{E}[|D_s^i \widehat X_t|^2]\,ds\,dt\leq C,
\end{align*}
where $C$ does not depend on $N, \frs$ or $h$.
\end{proposition}
In what follows, $C$ will be a constant that will change from line to line that is allowed to depend on everything except $N,h$ and $r$. 
\begin{proof}
Since the coefficients of the SDE are differentiable and Lipschitz we know that its solution is in $\mathbb L^{1.2}$ by \cite[Theorem 2.2.1]{Nualart2005Malliavin}. By differentiating, we obtain the SDE
\begin{align*}
         D_r^j\widehat X_T^{i}&=\mathbf{1}_{i=j}-\int_r^T\nabla^2 U^\frs(\widehat X_t^i)D_r^j\widehat X_t^{i}\,dt+\frac{\beta}{\sqrt N}\sum_{\ell=1}^N \int_r^T D_r^j \widehat X_t^\ell \partial_x f_A(\widehat X^\ell_t) f_h(\widehat Z_t^{i,\ell})\,dt,\\
         &+\frac{\beta}{\sqrt N}\sum_{\ell=1}^N\int_r^T f_A(\widehat X_t^\ell) \partial_x f_h (\widehat Z_t^{i,\ell})D_r^j \widehat Z_t^{i,\ell}\,dt\\
        D_r^j\widehat Z_T^i&=\mathbf 1_{i=j}\frac{\beta}{\sqrt N}\widehat Q_r f_A(\widehat X_r)+\frac{\beta}{\sqrt N}\int_r^T D_r^j(\widehat Q_t f_A(\widehat X_t))dW_t^i.
\end{align*}
For two matrices $A,B$ we denote by $A\odot B=(a_{ij}b_{ij})_{ij}$ the entry wise product, and similarly for vectors $x\odot y:=(x_i y_i)_{i}$. 
Use It\^{o} to get 
\begin{align*}
    d|D^j_r \widehat X_t|^2&=-2D^j_r \widehat X_t^\top\nabla^2U^\frs(\widehat X_t)D^j_r \widehat X_t+\frac{2\beta}{\sqrt N}\sum_{i,\ell=1}^N D^j_r \widehat X_t^\ell D^j_r \widehat X_t^i \partial_x f_A(\widehat X_t^{\ell}) f_h(\widehat Z_t^{i,\ell})\,dt\\
    &+\frac{2\beta}{\sqrt N}\sum_{i,\ell=1}^N f_A(\widehat X_t^\ell)\partial_x f_h (\widehat Z_t^{i,\ell})D_r^j \widehat Z_t^{i,\ell}D_r^j \widehat X_t^i\,dt.
    \end{align*}
    Use Hessian bounds of $U^\frs_c$ (and specifically the fact that uniformly in $\frs,\nabla^2 U^\frs_c\geq 0$ and $\nabla^2 U_\ell\leq C_{\textup{Lip}}$), rewrite the terms in matrix form and use matrix norm bounds to obtain
    \begin{align*}
    d|D^j_r \widehat X_t|^2&\leq C |D_r^j \widehat X_t|^2\,dt+\frac{C}{\sqrt N} (D_r^j \widehat X_t \odot \partial_x f_A(\widehat X_t))^\top D_r^j  f_h(\widehat Z_t)^\top D_r^j\widehat X_t\\
    &\hspace{2 em}+\frac{C}{\sqrt N} f_A(\widehat X_t)^\top (\partial_x f_h(\widehat Z_t)\odot D_r^j\widehat Z_t)^\top D_r^j \widehat X_t\,dt\\
    &\leq C |D^j \widehat X_t|^2+\frac{C}{\sqrt N}\|f_h(\widehat Z_t)\|_{\text{op}}|(D_r^j \widehat X_t \odot \partial_x f_A(\widehat X_t)||D_r^j \widehat X_t|\\
    &\hspace{2 em}+\frac{C}{\sqrt N} |f_A(\widehat X_t)||D_r^j \widehat X_t|\|\partial_x f_h(\widehat Z_t)\odot D_r^j \widehat Z_t\|_{\text{Fr}}.
    \end{align*}
    Notice that $|D_r^j \widehat X_t \odot \partial_x f_A(\widehat X_t)|^2=\sum_{i}|D_r^j \widehat X_t^i \partial_x f_A(\widehat X_t^i)|^2\leq C\sum_{i}|D_r^j \widehat X_t^i|^2$; the same reasoning holds for $\partial_x f_h(\widehat Z_t)\odot D_r^j \widehat Z_t$ and the Frobenius norm so that $\|\partial_x f_h(\widehat Z_t)\odot D_r^j \widehat Z_t \|_{\text{Fr}}\leq C\|D_r^j \widehat  Z_t \|_{\text{Fr}}$, and we obtain
    \begin{equation}
    \label{eq: malliavin DX square dynamics}
    d|D^j_r \widehat X_t|^2\leq C|D^j \widehat X_t|^2+\frac{C}{\sqrt N}\|Z_t\|_{\text{Fr}}|D_r^j \widehat X_t |^2+C|D_r^j \widehat X_t|\|D_r^j \hat Z_t\|_{\text{Fr}} .
\end{equation}
On the other hand,
\begin{align*}
    d|D_r^j \widehat Z_t^i|^2&=\frac{2\beta}{\sqrt N} D_r^j \widehat Z_t^i \cdot D_r^j(\widehat Q_t f_h(\widehat X_t)) dW_t^i+\frac{\beta^2}{N}|D_r^j (\widehat Q_t f_h(\widehat X_t))|^2\,dt,
\end{align*}
and hence
\begin{align*}
    d\|D_r^j \widehat Z_t\|_{\text{Fr}}^2&=\frac{2\beta}{ \sqrt N}\sum_{i=1}^N  D_r^j \widehat Z_t^i \cdot  D_r^j(\widehat Q_t f_h(\widehat X_t)) dW_t^i+\beta^2|D_r^j (\widehat Q_t f_h(\widehat X_t))|^2\,dt.
\end{align*}
Thus by using It\^{o} once again and Young's inequality, we obtain
\begin{align}
\label{eq: malliavin DZ dynamics fourth power}
\begin{split}
    d\|D_r^j \widehat Z_t\|_{\text{Fr}}^4&=\frac{2\beta^2}{N}\sum_{i=1}^N| D_r^j \widehat Z_t^i \cdot D_r^j(\widehat Q_t f_h(\widehat X_t)) |^2 dt+2\beta^2\|D_r^j \widehat Z_t\|^2_{\text{Fr}}|D_r^j (\widehat Q_t f_h(\widehat X_t))|^2\,dt+\,dM_t\\
    &\leq C(\beta^2+1)\big[\|D_r^j \widehat Z_t\|^4_{\text{Fr}}+|D_r^j(\widehat Q_t f_h(\widehat X_t))|^4\big]+dM_t,
\end{split}
\end{align}
where $M_t$ is a martingale term.
By the chain rule for Malliavin derivatives, we have
\begin{align}
\label{eq: bound on malliavin derivative of Q}
\begin{split}
    \|D_r^j \widehat Q_t\|_{\text{op}}&= \big\|-\widehat Q_t\frac{\beta^2}{N}\big(\int_r^t D_r^j f_A(\widehat X_s)f_A(\widehat X_s)^\top+f_A(\widehat X_s)D_r^j f_A(\widehat X_s)^\top \,ds\big)\widehat Q_t\big\|_{\text{op}} \\
    &\leq \frac{\beta^2C}{N}\int_r^t \|D_r^j f_A(\widehat X_s)f_A(\widehat X_s)^\top\|_{\text{op}}\,dt\\
    &\leq \frac{\beta^2C}{N}\int_r^t |D_r^j f_A(\widehat X_s)|\sqrt N A\,ds\leq \frac{\beta^2 CA}{\sqrt N}\int_r^t |D_r^j f_A(\widehat X_s)|\,ds.
\end{split}
\end{align}
On the other hand, $|D_r^j f_A(\widehat X_s)|^2=\sum_{i} |D_r^j\widehat X_t^i|^2 |\partial_x f_A(\widehat X_t^i)|^2\leq C|D_r^j \widehat X_t|^2$, and hence for some constant that does not depend on $A,h$ or $N$, we have
\begin{align}
\label{eq: malliavin operator bound}
 \begin{split}
|D_r^j(\widehat Q_t f_h(\widehat X_t))|^4&=|D_r^j\widehat Q_t f_h(\widehat X_t)+\widehat Q_t D_r^jf_h(\widehat X_t)|^4\\
    &\leq C(\|D_r^j \widehat Q_t\|_{\text{op}}|f_A(\widehat X_t)|)^4+C(\| \widehat Q_t\|_{\text{op}}|D_r^jf_A(\widehat X_t)|)^4\\
    &\leq C\int_r^t |D_r^j\widehat X_s|^4\,ds+C|D_r^j\widehat X_t|^4.
 \end{split}
\end{align}
Combining \eqref{eq: malliavin operator bound},\eqref{eq: malliavin DZ dynamics fourth power}, \eqref{eq: malliavin DX square dynamics} and using Young's inequality we conclude that there exists a constant $C \in (0, \infty)$ which does not depend on $N$ or $h$ such that
\begin{align*}
    |D_r^j \widehat X_t|^4+\|D_r^j \widehat Z_t\|^4_{\text{Fr}}&\leq 1+N^{-2}|Q_r \widehat X_r|^4+C\int_0^t (1+\|N^{-1/2}\widehat Z_s\|_{\text{Fr}})(|D_r^j \widehat X_s|^4+\|D_r^j \widehat Z_s\|_{\text{Fr}}^4)\,ds\\
    &+C\int_0^t \int_0^s |D_r^j \widehat X_u|^4\,du\,ds+M_t\\
    &\leq C+C\int_0^t (1+\|N^{-1/2}\widehat Z_s\|_{\text{Fr}})(|D_r^j \widehat X_s|^4+\|D_r^j \widehat Z_s\|_{\text{Fr}}^4)\,ds+M_t.
\end{align*}
By the stochastic Gr\"{o}nwall inequality (Theorem \ref{th: stochastic gronwall}), we get
\begin{align*}
    \mathbb{E}\Big[\sup_{s\leq t}|D_r^j \widehat  X_t|^2\Big]\leq C\mathbb{E}\bigg[\exp\bigg(\frac{5C}2 \int_0^t \|N^{-1/2}\widehat  Z_s\|_{\text{Fr}}\,ds\bigg)\bigg]^{1/5}.
\end{align*}
Now we estimate the Frobenius norm of $N^{-1/2}Z_s$. We claim that the expectation can be bounded independently of $N$ and $h$. First, notice that the entries of $\widehat Q_t f_A(\widehat X_t)$ need to be bounded by Lemma \ref{lm: entries of mollified matrix}. We use the Burkholder-Davis-Gundy (BDG) inequality and Jensen's inequality to estimate the exponential moments of $\|Z_t  N^{-1/2}\|_{\text{Fr}}$. Use Jensen's inequality and BDG again to find for $p\geq 2$,
\begin{align*}
\mathbb{E}[\|N^{-1/2}\widehat  Z_t\|_{\text{Fr}}^{p}]&=\mathbb{E}\bigg[\bigg(\frac{\beta^2}{N^2}\sum_{i,\ell=1}^N\bigg|\int_0^t (\widehat Q_sf_A(\widehat X_s))^{\ell}\,dW_s^i\bigg|^2\bigg)^{p/2}\bigg]\\
    &\leq \mathbb{E}\bigg[\frac{C^{p}p^{\frac{p}{2}}}{N}\sum_{\ell=1}^N\bigg|\int_0^t |(\widehat Q_sf_A(\widehat X_s))^{\ell}|^{2}\,ds\bigg|^{p/2}\bigg]\\
    &\leq  \Big(\frac{C^{2p}p^{p/2}}{N}\sum_{\ell=1}^N |t^{p/2}A^{2p}|\Big)\leq  p^{p/2} C^p t^p,
\end{align*}
where in the last display, we used the fact that the optimal constant in BDG scales like $\sqrt p$ by \cite[Theorem 1]{CarlenKreeBDG}). Therefore, the following holds for all $\lambda>0$:
\begin{align}
\label{eq: sub-Gaussian estimate on Z}
    \sup_{N\in \bbN}\sup_{t\leq T}\mathbb{E}[\exp(\lambda \|N^{-1/2}\widehat  Z_t\|_{\text{Fr}})]&=\sup_{N\in \bbN}\sup_{t\leq T}\sum_{p=0}^\infty\frac{\mathbb{E}[(\|N^{-1/2}Z_t\|^{p}_{\text{Fr}}]\lambda^p}{p!}\leq \sum_{p=0}^\infty\frac{C^{2p}p^{p/2}\lambda^p}{p!}<\infty,
\end{align}
and taking $\lambda=5C/2$ concludes the proof.
\end{proof}

\subsection{Step 2} We remove the regularization from the confining potential. Consider the SDE with locally Lipschitz coefficients
\begin{align}
\begin{split}
    \label{eq: h-regularized SDE}
         d\widetilde X_t^{i,h}&=- U'(\widetilde X_t^{i,h})\,dt+\frac{\beta}{\sqrt N}\sum_{\ell=1}^N f_A(\widetilde X^{\ell,h}_t) f_h(\widetilde Z_t^{i,\ell,h})\,dt+dW_t^i\\
        d\widetilde Z_t^{i,h}&=\frac{\beta}{\sqrt N}\widetilde Q^h_tf_A(\widetilde X^{h}_t)dW_t^i\\
        \widetilde Q^h_t&=\Big(I+\frac{\beta^2}N \int_0^t f_A(\widetilde X^h_s)f_A(\widetilde X^h_s)^\top\,ds\Big)^{-1}.\\
        \widetilde Q_0^h&=I, \widetilde Z^h_0=0, \widetilde X_0^{i,h}=X_0^i\,\,,\,\, i=1,\dots,N.
\end{split}
\end{align}
We keep track of the dependence on $h$ as it will be the only remaining regularization parameter.
\begin{lemma}
    The SDE \eqref{eq: h-regularized SDE} has a unique strong solution, with $\bbP(\sup_{t\leq T}|\widetilde X_t^{i,h}|\geq A)=0$ for each $i,t$ and $h$.
\end{lemma}
\begin{proof}
Consider the non-interacting SDE,
\begin{equation}
    \label{eq: driftless SDE}
        dX_t^i=-U'(X_t^i)\,dt+dW_t^i,\quad i=1,\dots,N.
    \end{equation}
    Recall that \eqref{eq: driftless SDE} has a unique solution such that $\bbP(\sup_{t\leq T}|X_t^i|\geq A)=0$ for each $i$ and $t\geq 0$, see \cite[Lemma 2.2]{Arous1995-zj}. Since the coefficients (aside from $U'$) are bounded, a routine Girsanov argument yields weak existence of \eqref{eq: h-regularized SDE} globally in time; in particular, the law of $\widetilde X^h$ is mutually absolutely continuous with respect to the law of the solution of \eqref{eq: driftless SDE}, so that $\bbP(\sup_{s\leq t}|\widetilde X_s^{i,h}|\geq A)=0$ for each $i,t\leq T$ and $h$. The coefficients of \eqref{eq: h-regularized SDE} are also locally Lipschitz, so \cite[Theorem 5.2.5]{Karatzas_and_Shreve_BMSC} and \cite[Corollary 5.3.23]{Karatzas_and_Shreve_BMSC} yield strong well-posedness. 
\end{proof}
\begin{proposition}
\label{pr: bound for h regularized SDE}
Let $\widetilde X^h$ be a strong solution of \eqref{eq:  h-regularized SDE}. Then $\widetilde X^h\in \mathbb L^{1,2}$, and in particular for every $i=1,\dots,n$ and every $t\leq T$
\begin{align*}
    \mathbb{E}[|D_r^i \widetilde X^h_T|^2]+\int_0^T \int_0^T \mathbb{E}[|D_s^i \widetilde X^h_t|^2]\,ds\,dt\leq C,
\end{align*}
where $C$ does not depend on $N$, or $h$. The constant is allowed to depend on $A$.
\end{proposition}
\begin{proof}
Let $\widehat X,\widetilde X^h$ solve the corresponding SDEs driven by the same Brownian motion on the same probability space. Given Proposition \ref{pr: bound for fully regularized SDE}, we want to show that $\widehat X$ converges to $\widetilde X^h$ as $\frs\to \infty$. In what follows, to point out that $\widehat X$ depends on $\frs$ we denote it by $\widehat X(\frs)$.
    By pathwise uniqueness, $\widehat X_t(\frs)= \widehat X_t^{h}$ up to an exit time
    \[\tau_\frs:=\inf\{t\geq 0 \mid |\widehat X_t^i(\frs)|\geq A-\frs^{-1} \text{ for some }i\}\wedge \inf\{t\geq 0 \mid |\widehat X_t^{i,h}|\geq A-\frs^{-1} \text{ for some }i\}.\]
    In particular by uniqueness, $\tau_\frs=\inf\{t\geq 0 \mid |\widehat X_t^{i,h}|\geq A-\frs^{-1} \text{ for some }i\}$, meaning that we only need to look at $\widehat X_t^h$. Notice that the moments of $\widehat X_t(\frs)$ are bounded uniformly in $\frs$ since $U^\frs_c$ remains convex. Then by using Cauchy--Schwarz,
    \begin{align*}
        \mathbb{E}[|\widehat X_t(\frs)-\widetilde X_t^{h}|^2]&=
        \mathbb{E}[|\widehat X_t(\frs)-\widetilde X_t^{h}|^2\mathbf 1_{t>\tau_\frs}]\leq C \bbP\Big(\max_{i=1,\dots,N}\sup_{s\leq t}|\widetilde X_s^{i,h}|>A-\frs^{-1}\Big)^{1/2}\\
        &\leq NC\bbP\big(\sup_{s\leq t}|\widetilde X_s^{i,h}|>A-\frs^{-1}\big)^{1/2}.
    \end{align*}
    which converges to $0$ as $\frs \to \infty$. Hence $\widehat X_t(\frs)$ converges to $\widetilde X_t^h$ in $L^2(\Omega)$; the same holds in $L^2(\Omega\times [0,T])$. Since by Proposition \ref{pr: bound for fully regularized SDE}, $\|\widehat X(\frs)\|_{\mathbb L^{1,2}}$ is bounded uniformly in $\frs$, we can extract a subsequence of $\widehat X(\frs)$ that converges weakly in $\mathbb L^{1,2}$ to some limit $Y\in \mathbb L^{1,2}$, which must evidently be $\widetilde X^h$. Hence $\widehat X(\frs)\rightharpoonup \widetilde X^h$ in $\mathbb L^{1,2}$; in particular for each $i=1,\dots,N$, $D^i \widehat X(\frs)$ converges weakly in $L^{2}(\Omega\times[0,T])$ to $D^i\widetilde X^h$. By weak lower-semicontinuity of norms in Hilbert spaces, we conclude from Proposition \ref{pr: bound for fully regularized SDE}.
    \end{proof}

    \subsection{Step 3} We remove the regularization from the other coefficients and conclude the proof. 
    \begin{proof}[Proof of Proposition \ref{pr: Malliavin calculus estimates}]
        We claim that $\widetilde X_t^h$ converges to $X_t$ in $L^2(\Omega,[0,T])$. Indeed, since $\bbP(|\widetilde X^{i,h}_t|\geq A)=0$ for every $t,i,h$, $f_A(\widetilde X_t^i)= \widetilde X_t^i$ almost surely. Let 
        \begin{equation*}
            \tau_h=\inf_{i,j}\inf\{t\geq 0 \mid |Z_t^{i,j}|\geq h\},
        \end{equation*}
        that is, the first time an entry of $Z_t$ has size larger than $h$. By pathwise uniqueness, $X_{t\wedge \tau_h}= \widetilde X^h_{t\wedge \tau_h}$. By a union bound, Markov's inequality and BDG,
        \begin{align*}
            \mathbb{E}[|X_t-\widetilde X_t^h|^2]&=\mathbb{E}[|X_t-\widetilde X_t^h|^2\mathbf 1_{t>\tau_h}]\leq 2A^2 \bbP\Big(\sup_{i,j} \sup_{s\leq t}|Z_s^{i,j}|\geq h\Big)\\
            &\leq N^2\bbP\Big(\sup_{s\leq t}|Z_s^{1,1}|\geq h\Big)\leq \frac{N^2C}{h^2}\mathbb{E}\Big[\int_0^t |(Q_sX_s)^{1}|^2\,ds\Big]\leq \frac{N^2C}{h^2}.
        \end{align*}
        If we let $h\to \infty$, we get the desired $L^2(\Omega)$ and $L^2(\Omega,[0,T])$ convergence. As before, since the $\mathbb L^{1,2}$ norm of $\widehat X^h$ is bounded uniformly in $h$, we can assume without loss of generality that $\widehat X^h \rightharpoonup X$ in $\mathbb L^{1,2}$, and thus by lower semicontinuity of the norm we conclude the first bound. The bound on $D_s^i Q_t$ follows from the chain rule for Malliavin derivatives and the same reasoning as \eqref{eq: bound on malliavin derivative of Q}.
    \end{proof}
\section{Universality}
\label{sect: universality}
\subsection{Preliminaries} The goal of this section is to prove Theorem \ref{th: universality}. In the setting of Gaussian disorder in Section \ref{sect: averaged propagation of chaos}, we obtained a tractable representation \eqref{eq: averaged dynamics Markovian} for the averaged law $\bbE[P^N(\bfG)]$ by computing a closed form solution of $\bbE[\bfG_i\mid \mF_t^X]$, see  \eqref{eq: conditional expectation under P}. Unfortunately, no such solution will be possible in the case of general disorder $\bfJ$. Nevertheless, in the next subsection we will quantify how far the process with law $\bbE[P^N(\bfJ)]$ is from satisfying \ref{eq: averaged dynamics Markovian} through an argument in the style of Stein's method. This allows us to use yet another coupling argument to establish Theorem \ref{th: universality}.

We now reproduce the setting of Section \ref{sect: averaged propagation of chaos} with the general disorder $\bfJ$. Recall the probability measure $\bbQ$ under which the vector valued process $X = (X^1, \ldots, X^N)$ solves the SDE
\begin{align*}
    dX_t^i&=-U'(X_t^i)\,dt+dV_t^i, \quad i=1,\dots,N,\\
    X_0& \sim P_0^N.
\end{align*}
where $V = (V^1, \ldots, V^N)$ is a standard Brownian motion under $\bbQ$ that is independent of the disorder $\bfJ$. We define by the probability measure $\bbP$ via Girsanov's theorem:
\[\frac{d\bbP}{d\bbQ}[t]:=\exp\bigg(\frac\beta{\sqrt N}\sum_{i=1}^N \int_0^t\bfJ_i \cdot X_s\,dV_s^i-\frac{\beta^2}{2N}\sum_{i=1}^N \bfJ_i^\top \int_0^t X_sX_s^\top\,ds\bfJ_i\bigg).\]
 Thus, under $\bbP$ $(X, V)$ satisfy the following SDEs:
\begin{align}
\begin{split}
    \label{eq: universality equation under P with V}
    dX_t^i&=-U'(X_t^i)\,dt+dV_t^i,\\
    dV_t^i&=\frac\beta{\sqrt N}\bfJ_i \cdot X_t\,dt+dB_t^i,
    \end{split}
\end{align}
where $B = (B^1, \ldots, B^N)$ is a standard Brownian motion under $\bbP$ which is independent of $\bfJ$. Note that $\bfJ$ has the same marginal law under both $\bbQ$ and $\bbP$. Let $\mF_t^X$ be the filtration generated by $X$. The goal is to follow the proof of Lemma \ref{lem: averaged dynamics 0} to obtain the averaged dynamics:
\begin{equation}
\begin{aligned}
    dX_t^i&=- U'(X^i_t)\,dt+dV_t^i,\\
    dV_t^i&=\frac \beta{\sqrt N}\bbE[\bfJ_i \mid \mF_t^X] \cdot X_t\,dt+dW_t^i,
    \end{aligned}
    \label{eq: universality-mimicked eqn}
\end{equation}
where $W$ is a $\mF^X$-Wiener process. The problem is that since $\bfJ$ is not Gaussian, the calculation \eqref{eq: conditional expectation under P} fails and there is no closed form of $\bbE[\bfJ_i \mid \mF_t^X]$. However, we may always define 
\begin{equation}
    Q_t := \Big(I + \frac{\beta^2}{N}\int_0^tX_sX_s^\top ds\Big)^{-1},
    \label{eq: universality Q}
\end{equation}
and set 
\begin{align}
    \label{eq: definition of u error}
    \bbE[\bfJ_i \mid \mF_t^X]-\frac\beta{\sqrt N}Q_t\int_0^t X_s\,dV_s^i=\fru_t^i,
\end{align}
where we refer to $\fru$ in the sequel as the \emph{universality error}. If the universality error $\fru_t$ is small, then we expect the process to \eqref{eq: universality-mimicked eqn} to nearly solve \eqref{eq: averaged dynamics 1}.

\subsection{Universality error} 
For $t \in [0, T]$, define the function $f_t:\R^n \rightarrow \R$
\begin{align*}
     f_t(a):=\exp\bigg( a \cdot \frac{\beta}{\sqrt N}\int_0^t X_s\,dV_s^i-\frac{1}{2}a^\top \bigg(\frac{\beta^2}{N}\int_0^t X_sX_s^\top\,ds\bigg)    a\bigg).
 \end{align*}
Here we are suppressing the pathwise dependence of $f$ on $X$, and in this sense we may interpret $f_t$ as a random function. Using Bayes' theorem, the independence of the rows of $\bfJ$, and the independence of $\bfJ$ and $V$ under $\bbQ$,  we find
\begin{align*}
    \bbE_\bbP\big[\bfJ_i \mid \mF_t^X\big]=\frac{
    \bbE_\bbQ\Big[\bfJ_i \frac{d\bbP}{d\bbQ}[t]\mid \mF_t^X\Big]}{\bbE_\bbQ\Big[ \frac{d\bbP}{d\bbQ}[t]\mid \mF_t^X\Big]}=\frac{
    \bbE_\bbQ\Big[\bfJ_i f_t(\bfJ_i)\mid \mF_t^X\Big]}{\bbE_\bbQ\Big[ f_t(\bfJ_i)\mid \mF_t^X\Big]}.
\end{align*}

We start from a useful perturbation estimate of the density $f_t$. With $r \in [0, 1]$, we will often abuse notation and write $(rJ_{ij},J_{i,-j})$ for the vector which is equal to $\bfJ_i$ except on the entry $i$, which is replaced by $J_{ij}r$.
\begin{lemma}
\label{lm: density bound}
    We have $\bbP$-almost surely that for $r\in [0,1]$,
    \begin{align*}
        \frac{f_t(rJ_{ij},J_{i,-j})}{f_t(\bfJ_i)}=\exp\bigg(\frac{\beta(r-1)}{\sqrt N}J_{ij}\int_0^tX_s^j dB_s^i-\frac{\beta^2(1-r)^2}{2N}|J_{ij}|^2\int_0^t |X_t^j|^2\,ds\bigg).
    \end{align*}
    In particular there exist constants $C_{\fru},C \in (0, \infty)$ that depend only on $(\beta,T,A,\sigma^2)$ such that if  $\sqrt N > C_{\fru}$, for all $r$ we have
    \begin{align*}
        \bbE_\bbP\bigg[\Big| 
        \frac{f_t(rJ_{ij},J_{i,-j})}{f_t(\bfJ_i)}\Big|^8\bigg]\leq C
    \end{align*}
\end{lemma}
\begin{proof}
Start from the definition of $f_t(\bfJ_i)$, perform the division, and expand the dynamics of $V_t^i$ from \eqref{eq: universality equation under P with V} to find
    \begin{align*}
        \log \frac{f_t(rJ_{ij},J_{i,-j})}{f_t(\bfJ_i)}&=\frac{\beta(r-1)}{\sqrt N}J_{ij}\int_0^tX_s^j dV_s^i-(r^2-1)\frac{\beta^2}{2N}|J_{ij}|^2\int_0^t |X_t^j|^2\,ds\\
        &-(r-1)\frac{\beta^2}{N}J_{ij}\sum_{\ell\neq j} J_{i\ell}\int_0^t X_s^jX_s^\ell\,ds\\
        &=\frac{\beta(r-1)}{\sqrt N}J_{ij}\int_0^tX_s^j dB_s^i-\frac{\beta^2(1-r)^2}{2N}|J_{ij}|^2\int_0^t |X_t^j|^2\,ds,
    \end{align*}
    as required. Now we can use Young's inequality repeatedly:
    \begin{align*}
        \bbE_\bbP\bigg[\Big|\frac{f_t(rJ_{ij},J_{i,-j})}{f_t(\bfJ_i)}\Big|^8\bigg] &\leq \bbE_\bbP\bigg[\exp\Big(\frac{8r\beta}{\sqrt N}J_{ij}\int_0^t X_s^j\,dB_s^i\Big)\bigg] \\ &\leq  \bbE_\bbP\bigg[\exp\Big(\frac{4\beta}{\sqrt N}|J_{ij}|^2+\frac{4\beta}{\sqrt N}\Big|\int_0^t X_s^j\,dB_s^i\Big|^2\Big)\bigg]\\
        &\leq  \bbE_\bbP\bigg[\exp\Big(\frac{8\beta}{\sqrt N}|J_{ij}|^2\Big)\bigg]+\bbE_\bbP\bigg[\exp\Big(\frac{8\beta}{\sqrt N}\Big|\int_0^t X_s^j\,dB_s^i\Big|^2\Big)\bigg].
    \end{align*}
    The first term is smaller than $2$ if $\sqrt N > 32\beta \sigma^2$, where $\sigma^2$ is the T$_2$ constant of $J_{11}$ from Assumption \ref{assumption: main} (notice that we do not need the T$_2$ inequality; we could take $\sigma^2$ to be the sub-Gaussian norm of $J_{11}$). The second term is smaller than $2$ if $\sqrt N>32\beta A^2t$.
\end{proof}
\begin{proposition}
\label{pr: universality error bound}
    If $\sqrt N>C_\fru$ where $C_\fru$ is the constant from Lemma \ref{lm: density bound}, then there exists a constant $C$ that depends only on $(T,\sigma^2,\beta,A)$ such that
    \[\sup_{t\leq T}\bbE_\bbP\Big[|\fru_t^i \cdot X_t|^4\Big]^{1/2}\leq C.\]
\end{proposition}
\begin{proof}
Notice that
\begin{align*}
    \nabla f(a)=f(a) \bigg[\frac{\beta}{\sqrt N}\int_0^t X_s\,dV_s^i- \bigg(\frac{\beta^2}{N}\int_0^t X_sX_s^\top\,ds\bigg)  a\bigg].
\end{align*}
Evaluating the above display at $ a = \bfJ_i$, subtracting $\bfJ_i f(\bfJ_i)$ from both sides, and taking conditional expectations under $\bbQ$ with respect to $\mF_t^X$ yields
\begin{align*}
    &\bbE_\bbQ\Big[\nabla f(\bfJ_i)-\bfJ_i f(\bfJ_i)\mid \mF_t^X\Big] \\= &\bbE_\bbQ\Big[f(\bfJ_i)\mid \mF_t^X\Big]\frac{\beta}{\sqrt N}\int_0^t X_s\,dV_s^i- \Big(I+\frac{\beta^2}N\int_0^t X_sX_s^\top \Big) \bbE_\bbQ\Big[\bfJ_i f(\bfJ_i)\mid \mF_t^X\Big].
\end{align*}
We can take the stochastic integral outside of the expectations because $V_t^i=X_t^i-X^i_0+\int_0^t U'(X_s)\,ds$ is $\mF_t^X$-measurable. Recall the definition of $Q$ from \eqref{eq: universality Q}.  Dividing both sides of the last display by $\bbE_\bbQ[ f(\bfJ_i)\mid \mF_t^X]$ and then multiplying by $Q_t$ from the left yields:
\begin{align}
\label{eq: expression for universality error}
    \fru_t^i&=\frac{\bbE_\bbQ\big[\bfJ_i f(\bfJ_i)\mid \mF_t^X\big]}{\bbE_\bbQ\big[f(\bfJ_i)\mid \mF_t^X\big]}-\frac\beta{\sqrt N} Q_t \int_0^t X_s \,dV_s^i= Q_t \frac{\bbE_\bbQ\big[\bfJ_i f(\bfJ_i)\mid \mF_t^X\big]-\bbE_\bbQ\big[\nabla f(\bfJ_i)\mid \mF_t^X\big]}{\bbE_\bbQ\big[f(\bfJ_i)\mid \mF_t^X\big]} .
\end{align}

To control the universality error $\fru$, we bound $\bbE_\bbQ\big[\bfJ_i f(\bfJ_i)-\nabla f(\bfJ_i)\mid \mF_t^X\big]$. Observe that $a\mapsto f_t(a)$ is smooth. Fix an index $j$ and notice that for a vector $a\in \R^N$ partitioned as $a=(a_j,a_{-j})$ a Taylor expansion around $0$ yields
\begin{align*}
    a_j &f(a_j,a_{-j})-\partial_j f(a_j,a_{-j})\\
    &=a_j f(0,a_{-j})+(|a_j|^2-1) \partial_j f(0,a_{-j})+ \int_0^1 (1-r) \partial_j^2 f(ra_j,a_{-j})a_j^3-\partial_j^2 f(ra_j,a_{-j})a_j\,dr.
\end{align*}
Substitute $a=\bfJ_i$ in the above and integrate with respect to $\bbQ(\cdot \mid \mF_t^X)$. Recall that the $J_{ij}$ are iid with mean $0$ and variance $1$. Then the first two terms disappear since under $\bbQ$, the terms $J_{ij}$ and $\bfJ_{i, -j}$ are independent, and in particular  are conditionally independent given $X_{[t]}$. Then we may conclude from the last display that
\begin{align*}
    \bbE_\bbQ\Big[J_{ij} f(\bfJ_i)-\partial_j f(\bfJ_i)\mid \mF_t^X \Big]=\int_0^1 \bbE_\bbQ\Big[\big((1-r) J^3_{ij}-J_{ij} \big)\partial_j^2 f(rJ_{ij},J_{i,-j})\mid \mF_t^X\Big]\,dr.
\end{align*}
A calculation yields
\begin{align*}
    \partial_j^2 f(a)&=f(a) \Big(\frac{\beta}{\sqrt N}\int_0^t X_s^j\,dV_s^i- \sum_{\ell=1}^N a_\ell\frac{\beta^2}{N}\int_0^t X_s^jX_s^\ell\,ds     \Big)^2-f(a) \frac{\beta^2}{N}\int_0^t |X_s^j|^2\,ds.
\end{align*}
We now exploit a key cancellation: for $r\in[0,1]$, evaluate the above at $(rJ_{ij},J_{i,-j})$ and apply the second equation of \eqref{eq: universality equation under P with V} to obtain
\begin{align*}
    &\partial_j^2 f(rJ_{ij},J_{i,-j}) \\=&f(rJ_{ij},J_{i,-j})\frac{\beta^2}{N}\Big(\int_0^t X_s^j\,dB_s^i-\frac{(r-1)J_{ij}\beta}{\sqrt N}\int_0^t|X_s^j|^2\,ds \Big)^2-\frac{\beta^2}N f(rJ_{ij},J_{i,-j})\int_0^t |X_s^j|^2\,ds.
\end{align*}
This is important because it eliminates the integral against $V^i$ which is difficult to understand under $\bbP$. Plugging the above display into \eqref{eq: expression for universality error} and using Bayes' formula yields
\begin{align*}
     &((Q_t)^{-1}\fru_t^i)^j=\frac{\bbE_\bbQ\Big[J_{ij} f(\bfJ_i)-\partial_j f(\bfJ_i)\mid \mF_t^X \Big]}{\bbE_\bbQ\Big[f(\bfJ_i)\mid \mF_t^X \Big]}\\
      &=\frac{\beta^2}{N}\int_0^1 \bbE_\bbP\bigg[\frac{f(J_{ij}r,J_{i,-j})}{f(\bfJ_i)}\big(J_{ij}^3 (1-r)-J_{ij}\big)\Big(\int_0^t X_s^j\,dB_s^i-\frac{(r-1)J_{ij}\beta}{\sqrt N}\int_0^t|X_s^j|^2\,ds \Big)^2\,\Big\vert\,\mF_t^X\bigg]\,dr \\
      &\quad\quad\quad-\frac{\beta^2}{N}\int_0^1 \bbE_\bbP\bigg[\frac{f(J_{ij}r,J_{i,-j})}{f(\bfJ_i)}\big(J_{ij}^3 (1-r)-J_{ij}\big)\,\Big\vert\,\mF_t^X\bigg]\,dr\int_0^t |X_s^j|^2\,ds\\
     &:=\frac{1}{N} (\fra_t^{ij}+\frb_t^{ij}).
\end{align*}
Now use the last display, Cauchy-Schwarz, the bounds $\|Q_t\|_{\text{op}}\leq 1$ and $|X_t|\leq \sqrt NA$, Jensen's inequality and the exchangeability of $\fra,\frb$ to obtain
\begin{align*}
    \bbE\big[| \fru_t^i \cdot X_t|^4\big]\leq 
    \bbE\big[| Q_t(\fra_t^i+\frb_t^i) N^{-1}|^4  |X_t|^4\big]\leq N^2A^4\bbE\Big[\big|\frac{1}{N^2}\sum_{j=1}^N|\fra_t^{ij}+\frb_t^{ij}|^2\big|^2\Big] \leq C\bbE[|\fra^{11}_t|^4+|\frb^{11}_t|^4].
\end{align*}
In what follows it is important to remember that $B$ is a standard Brownian motion under $\bbP$. Use Young's inequality, the fact that $J_{ij}$ is sub-Gaussian, boundedness of $x$, Jensen's inequality, and Lemma \ref{lm: density bound} to obtain
\begin{align*}
    \bbE\big[|\fra_t^{ij}|^4\big]&\leq C\bbE_\bbP\bigg[\Big|\frac{f(J_{ij}r,J_{i,-j})}{f(\bfJ_i)}\Big|^8\bigg]\\
    &+C\bbE_\bbP\bigg[\Big|\big(J_{ij}^3 (1-r)^2-rJ_{ij}\big)^8\Big(\int_0^t X_s^j\,dB_s^i-\frac{(r-1)J_{ij}\beta}{\sqrt N}\int_0^t|X_s^j|^2\,ds \Big)^{16}|\bigg]\leq C.
\end{align*}
We may repeat the same argument to bound $\frb_t^{ij}$:
\begin{align*}
    \bbE\big[|\frb_t^{ij}|^4\big]\leq C\bbE_\bbP\bigg[\Big|\frac{f(J_{ij}r,J_{i,-j})}{f(\bfJ_i)}\Big|^8\bigg]+C\bbE_\bbP\bigg[\Big|(J_{ij}^3 (1-r)^2-rJ_{ij})|^8\bigg]\leq C.
\end{align*}
The last three displays together then yield the result.
\end{proof}

\subsection{Proof of universality} We are now in position to prove Theorem \ref{th: universality}. The idea is to use the mimicking theorem, couple the resulting equation with the one obtained for Gaussian disorder, and prove a global $\bbW_2$ bound of order $O(1)$. Subadditivity yields the result.

\begin{proof}[Proof of Theorem \ref{th: universality}]
Take a subset $v\subseteq \{1,\dots,N\}$ with $|v|=k$. If $\sqrt N\leq C_\fru$ where $C_\fru$ is the constant from Lemma \ref{lm: density bound}, we give a vacuous proof. Indeed, since $P^i(\bfJ),P^i(\bfG)$ are supported on $\mC_T$, we have
\begin{align*}
    \bbW_{2,L}\big(\bbE[P^v_{[t]}(\bfJ)],\bbE[P^v_{[t]}(\bfG)]\big)\leq A\sqrt {kt}\leq \frac{A\sqrt {tk}C_\fru}{\sqrt N},
\end{align*}
and the same holds for the time marginals. 

We now deal with the case $\sqrt N>C_\fru$. Theorem \ref{t:mimicking} gives
\begin{align*}
    dX_t^i&=- U'(X^i_t)+dV_t^i\\
    dV_t^i&=\frac \beta{\sqrt N}\bbE[\bfJ_i \mid \mF_t^X] \cdot X_t\,dt+dW_t^i.
\end{align*}
For some $\mF_t^X$-Wiener process $W$. Using the definition of the error $\fru_t^i$ in \eqref{eq: definition of u error} we get
\begin{align*}
    dX_t^i&=-U'(X^i_t)+dV_t^i,\\
    dV_t^i&=\frac{\beta^2}{N}X_t^\top Q_t\int_0^tX_s\,dV_s^i\,dt+\frac\beta{\sqrt{N}}\fru_t^i\cdot X_t\,dt+dW_t^i.
\end{align*}
Use It\^{o}'s formula and expand the dynamics of $V_t^i$ to obtain
\begin{align*}
    dQ_t\int_0^t X_sdV_s^i&=-\frac{\beta^2}NQ_t X_tX_t^\top Q_t\int_0^t  X_sdV_s^i\,dt+Q_tX_t\,dV_t^i\\
    &=Q_tX_t\,dW_t^i+\frac{\beta}{\sqrt N}Q_tX_t \fru_t^i \cdot X_t\,dt.
\end{align*}
Define the process
\[Z_t^i:=\frac\beta{\sqrt N}\int_0^t Q_s X_s\,dW_s^i, \quad i = 1, \ldots, N.\]
The last 4 displays together then imply:
\begin{equation}
    dX_t^i=-U'(X^i_t)+\frac{\beta}{\sqrt N} X_t \cdot Z_t^i\,dt+\frac{\beta^3}{N^{3/2}}X_t^\top \Big[\int_0^t Q_s X_s (\fru_s^i\cdot X_s)\,ds\Big]\,dt+\frac\beta{\sqrt N} \fru_t^i \cdot X_t\,dt+dW^i_t.
    \label{eq: universality averaged SDE important}
\end{equation}

To conclude, we compare the process $X$ to the averaged SDE with Gaussian disorder \eqref{eq: averaged dynamics Markovian} via synchronous coupling, assuming that $X$ and $\tlX$ have the same initial conditions. By Lemma \ref{lm: gaussian disorder averaged equation}, there exists a unique strong solution to \eqref{eq: averaged dynamics Markovian} with running Brownian motion $W$, and we denote this solution by $\tlX = (\tlX^1, \ldots, \tlX^N)$.

We define \[\tlQ_t = \bigg(I + \frac{\beta^2}{N}\int_0^t \tlX_s \tlX_s^\top ds\bigg)^{-1}, \quad \tlZ^i_t := \frac{\beta}{\sqrt{N}}\int_0^t \tlQ_s \tlX_s dW_s^i, \quad i = 1, \ldots, N.\] Then, by Lemma \ref{lm: gaussian disorder averaged equation}, we may write the SDE for $\tlX$ as
\begin{equation}
    d\tlX_t^i=-U'(\tlX_t^i)\,dt+\frac{\beta}{\sqrt N}\tlX_t\cdot \tlZ_t^i\,dt+dW_t^i.
    \label{eq: universality tlX}
\end{equation}

Recall from Assumption \ref{assumption: main} that $U = U_c + U_\ell$. By the assumed convexity of $U_c$, the fact that $U'_\ell$ is Lipschitz, the SDEs \eqref{eq: universality averaged SDE important} and \eqref{eq: universality tlX}, and Young's inequality, we have 
\begin{align*}
    \frac12 &\frac d{dt}|\tlX_t-X_t|^2\leq C|\tlX_t-X_t|^2+C\sum_{i=1}^N\frac{\beta^6}{N^{3}}\bigg|X_t^\top \bigg[\int_0^t Q_s X_s (\fru_s^i\cdot X_s)\,ds\bigg]\bigg|^2\\
    &+C\frac{\beta^2}{N} \sum_{i=1}^N|\fru_t^i \cdot X_t|^2+\frac{\beta}{\sqrt N}(\tlX_t-X_t)^\top \big( \tlZ_t\tlX_t-Z_tX_t\big).
    \end{align*}
Use the bound $|X_t|\leq A\sqrt N$ and add and subtract $\tlZ_tX_t$ inside the last difference to find 
\begin{align*}
        \frac d{dt}|\tlX_t-X_t|^2&\leq C\big(1+\beta N^{-1/2}\|\tlZ_t\|_{\text{Fr}})|\tlX_t-X_t|^2+\frac{C}{N}\Big( \sum_{i=1}^N\int_0^t |\fru_s^i \cdot x_s|^2\,ds+|\fru_t^i \cdot X_t|^2\Big)\\
    &+C\|\tlZ_t-Z_t\|_{\text{Fr}}^2.
\end{align*}
Finally, by combining the chain rule, Young's inequality, and Jensen's inequality we obtain
    \begin{align}
    \begin{split}
    \label{eq: universality Gronwall step 1}
        \frac d{dt}|\tlX_t-X_t|^4&\leq C\big(1+\beta N^{-1/2}\|\tlZ_t\|_{\text{Fr}})|\tlX_t-X_t|^4+\frac{C}{N} \sum_{i=1}^N\int_0^t |\fru_s^i \cdot x_s|^4\,ds+|\fru_t^i \cdot X_t|^4\\
    &+C\|\tlZ_t-Z_t\|_{\text{Fr}}^4.
\end{split}
\end{align}
We now focus on the bounding the last term of the previous display. It\^{o}'s formula yields
\begin{align*}
    d|\tlZ_t^i-Z_t^i|^2=\frac{\beta^2}{N}|\tlQ_t \tlX_t-Q_tX_t|^2\,dt+ \frac{2\beta}{\sqrt N}\big(\tlZ_t^i-Z_t^i\big)\cdot(\tlQ_t \tlX_t-Q_tX_t)\,dW_t^i.
\end{align*}
Summing over the coordinates yields:
\begin{align*}
    d\|\tlZ_t-Z_t\|^2_{\text{Fr}}=\beta^2|\tlQ_t \tlX_t-Q_tX_t|^2\,dt+ \frac{2\beta}{\sqrt N}\sum_{i=1}^N\big(\tlZ_t^i-Z_t^i\big)\cdot(\tlQ_t \tlX_t-Q_tX_t)\,dW_t^i.
\end{align*}
By applying It\^{o}'s formula to the last display, we have
\begin{align}
\label{eq: universality gronwall step 2}
\begin{split}
    d\|\tlZ_t-Z_t\|^4_{\text{Fr}}&=2\beta^2|\tlQ_t \tlX_t-Q_tX_t|^2\|\tlZ_t-Z_t\|^2_{\text{Fr}}\,dt\\
    &+ \frac{4\beta^2}{N}\sum_{i=1}^N|\big(\tlZ_t^i-Z_t^i\big)\cdot(\tlQ_t \tlX_t-Q_tX_t)|^2\,dt+dM_t\\
    &\leq C|\tlQ_t\tlX_t-Q_t X_t|^4\,dt+C\|\tlZ_t-Z_t\|^4_{\text{Fr}}\,dt+dM_t,
\end{split}
\end{align}
where $M$ is a martingale with continuous paths.
By the operator bound $\|Q_t\|_{\text{op}}\leq 1$,
\begin{align*}
\int_0^t|Q_s X_s-\tlQ_s\tlX_s|^4\,ds&\leq C \int_0^t|Q_s (X_s-\tlX_s)|^4+|(Q_s-\tlQ_s) \tlX_s|^4ds\\
    &\leq C\int_0^t |X_s-\tlX_s|^4+N^2\|Q_s-\tlQ_s\|_{\text{Fr}}^4\,ds.
\end{align*}
Now, by the definitions of $Q$ and $\tlQ$ above, the operator norm bounds $\|Q\|_\op, \|\tlQ\|_\op \leq 1$, Jensen's inequality and the uniform bounds $|\tlX_t|, |X_t| \leq \sqrt{N}A$, we have
\begin{align*}
N^2\int_0^t\|Q_s-\tlQ_s\|_{\text{Fr}}^4ds &=\frac{\beta^8}{N^2}\int_0^t\bigg\|Q_s\bigg(\int_0^s \tlX_u\tlX_u^\top-X_uX_u^\top\,du\bigg)\tlQ_s\bigg\|^4_{\text{Fr}}\,ds\\
&\leq \frac{C}{N^2}\int_0^t\int_0^s |\tlX_u-X_u|^4(|\tlX_u|^4+|X_u|^4)\big]\,du\,ds \\ &\leq C\int_0^t |\tlX_s-X_s|^4\,ds .
\end{align*}
Combining the last two displays with  \eqref{eq: universality gronwall step 2} we get
\begin{align}
\label{eq: universality gronwall step 3}
    \|\tlZ_t-Z_t\|_{\text{Fr}}^4\leq C\int_0^t|\tlX_s-X_s|^4+
    \|\tlZ_s-Z_s\|_{\text{Fr}}^4\,ds+M_t.
\end{align}
By \eqref{eq: universality Gronwall step 1} and \eqref{eq: universality gronwall step 3}, we have
\begin{align*}
    |\tlX_t-X_t|^4+\|\tlZ_t-Z_t\|_{\text{Fr}}^4&\leq \int_0^t C(1+\beta N^{-1/2}\|\tlZ_s\|_{\text{Fr}})\big(|\tlX_s-X_s|^4+\|\tlZ_s-Z_s\|_{\text{Fr}}^4\big)\,ds\\
    &+\frac{C}{N} \sum_{i=1}^N\int_0^t |\fru_s^i \cdot X_s|^4\,ds+M_t.
\end{align*}
Now, Theorem \ref{th: stochastic gronwall} yields
\begin{align*}
    \bbE\Big[\sup_{t\leq T}|\tlX_t-X_t|^2\Big]\leq C\bbE\bigg[\exp\bigg(\frac{5C}{2} \int_0^T (1+\beta\|\tlZ_sN^{-1/2}\|_{\text{Fr}})\,ds\bigg)\bigg]^{1/5}\bbE\bigg[\bigg|\frac{C}{N} \sum_{i=1}^N\int_0^T |\fru_s^i \cdot X_s|^4\,dt\bigg|^{5/8}\bigg]^{4/5}.
\end{align*}
We can prove that the first expectation is bounded uniformly in $N$ by following the proof of \eqref{eq: sub-Gaussian estimate on Z}. Then by Jensen's inequality and exchangeability of $\fru^i$, we have obtained
\begin{align*}
    \bbE\Big[\sup_{t\leq T}|\tlX_t-X_t|^2\Big]\leq C\bbE\bigg[\int_0^T |\fru_s^i \cdot X_s|^4\,dt\bigg]^{1/2}.
\end{align*}
Now, Proposition \ref{pr: universality error bound} yields for some dimension free constant $C$, the following bound:
\begin{align}
\label{eq: universality coupling final error}
    \sup_{t\leq T}\bbE\Big[|\tlX_t-X_t|^2\Big]\leq C\bbE\Big[\sup_{t\leq T}|\tlX_t-X_t|^2\Big]\leq C.
\end{align}
We can then use exchangeability and the definition of $\bbW_{2,L}^2$ to check
\begin{align*}
    \bbW_{2,L}^2\Big(\bbE[P_{[t]}^v (\bfJ)],\bbE[P_{[t]}^v (\bfG)]\Big)\leq \sum_{i\in v} \int_0^t \bbE\big[|\tlX_s^i-X_s^i|^2\big]\,ds\leq \frac{C|v|}{N} \sup_{t\leq T}\bbE\Big[|\tlX_t-X_t|^2\Big]\leq C \frac{|v|}{N},
\end{align*}
and a similar calculation in the time marginal case yields the result.
\end{proof}

\section{Optimality without confinement}
\label{sec: optimality}
\subsection{Linear dynamics} In this section, we provide some evidence that the rate obtained in Theorem \ref{th: main theorem} is optimal for $k = 1$ by studying \eqref{eq: intro quenched dynamics} when $U = 0$. For convenience, we set $\beta = 1$, $X_0 = 0$, and $\bfJ$ to be the standard Gaussian disorder. In this case, for each realization of $\bfJ$ the process $X = (X^1, \ldots, X^N)$ is an Ornstein-Uhlenbeck process satisfying:
\begin{equation}
 \label{eq: U=0 quenched SDE}
    dX_t=\frac{1}{\sqrt N}\bfJ X_tdt+dW_t, \quad X_0=0.
\end{equation}
In particular, we have (in distribution)
\[X_t=\int_0^t e^{sN^{-1/2}\bfJ }\,dW_s.\]
By It\^{o} isometry, the law of $X_t$ conditioned on $\bfJ$ may be written explicitly as
\begin{equation*}
  P_t(\bfJ)=\mathcal N\Big(0, \int_0^te^{sN^{-1/2}\bfJ}e^{sN^{-1/2}\bfJ^\top}\,ds\Big).
\end{equation*}
In particular, the quenched single-spin law satisfies $P^1_t(\bfJ) \sim \mN(0, q^2_t(\bfJ))$, where 
\begin{equation}
    q^2_t(\bfJ) :=\bigg(\int_0^t e^{sN^{-1/2} \bfJ}e^{s N^{-1/2}\bfJ^\top}\bigg)_{11}.
    \label{eq: q_t^2(bfJ)}
\end{equation}

The confining potential $U$ plays a crucial role in the proof of existing results. Neither our Theorem \ref{th: main theorem}, nor the qualitative results of \cite[Theorem 2.10]{Arous1995-zj} and \cite[Theorem 2.10]{Guionnet-1997-PoC} apply without it. Thus, in the case $U = 0$, the existence of a limiting law is not immediate. For completeness, we establish the existence of a limit in the next proposition via explicit manipulations available. 

In the sequel, for $q \in [0, \infty)$, let $\gamma_q$ be a centered Gaussian measure on $\R$ with variance $q^2$. The proof can be found in Appendix \ref{sec: wick} and involves a calculation similar to the combinatorial proof of Wigner's theorem.

\begin{proposition}\label{prop: conv U=0} For all $t \in (0, T]$, there exists $\brq_t \in (0, \infty)$ such that for all $\eps > 0$, we have
\[\lim_{N \rightarrow \infty} \bbP\Big(\bbW_1\big( P^1_t(\bfJ), \gamma_{\brq_t}) > \eps \Big) = 0. \]    
\end{proposition}

We establish the following:
\begin{theorem}
\label{t: optimal}
Fix $t,q\in (0,\infty)$. There exist constants $N_0(t, q) \in \bbN$ and $C(t, q) \in (0, \infty)$ depending on $(t,q)$ such that for all $N\geq N_0(t, q)$, we have
    \begin{align*}
        \bbE\Big[\bbW_1^2(P_t^1(\bfJ),\gamma_q)\Big]\geq \frac{C(t,q)}{N}.
    \end{align*}
\end{theorem}
\begin{proof}
First we recall a simple moment estimate for $q_t(\bfJ)$. By \eqref{eq: q_t^2(bfJ)} and \cite[Theorem 4.4.3]{Vershynin2018HDP}, for all $p\geq 1$ there exists a constant $M_p \in (0, \infty)$ which does not depend on $N$ such that
    \begin{align*}
        \bbE\Big[|q_t(\bfJ)|^{2p}\Big]\leq C\int_0^t \bbE\Big[e^{2ps\|\bfJ\|_{\text{op}}N^{-1/2}}\Big]\,ds\leq M_{2p}.
    \end{align*}
    Let $S\subseteq \R^{N\times N}$ be the set from Proposition \ref{pr: rnd matrix operator bound}, so that $\bbP(S^c)\leq 2e^{-N}$ and for every $\bfJ\in S, \|\bfJ\|_{\textup{op}}\leq C_S\sqrt N$. We emphasize that $C_S$ does not depend on $N$.  Then by the Cauchy--Schwarz inequality, we have
    \begin{align}
    \label{eq: optimality rough S bound}
        \big|\bbE\big[|q_t(\bfJ)|^{2p}\big]-\bbE\big[|q_t(\bfJ)|^{2p}\mid S\big]\big|
        \leq M_{2p}e^{-N/2}.
    \end{align}
    Since $P_t^1(\bfJ),\gamma_q$ are the laws of Gaussian random variables,  the optimal coupling for the $\bbW_2$ distance is given by $(q_t(\bfJ)Z, qZ)$ where $Z$ is a standard Gaussian random variable, see \cite[Theorem 2.12]{villani2021topics}. Furthermore $P_t^1(\bfJ), \gamma_q$ are 1-dimensional probability distributions, by \cite[Remark 2.19]{villani2021topics} this coupling is also optimal for $\bbW_1$ since we are in dimension $1$. Then for each $\bfJ$, we have
    \begin{align*}
\bbW_1(P_t^1(\bfJ),\gamma_q)=\sfE_\bfJ\big[\big|q_t(\bfJ)-q\big||Z|\big]=\sqrt{\frac{2}{\pi}}|q_t(\bfJ)-q|.
    \end{align*}
    Using this fact, integrating with respect to $\bfJ$, and applying \eqref{eq: optimality rough S bound}, we find 
    \begin{align*}
    \begin{split}
        \bbE\Big[\bbW_1^2(P_t^1(\bfJ),\gamma_q)\Big]&=\frac{2}{\pi} \bbE\Big[\big({q_t(\bfJ)}-{q}\big)^2\Big] \\&\geq \frac
        2   \pi\bbE\Big(\big[\big({q_t(\bfJ)}-{q}\big)^2\mid S\big]-M_2e^{-N/2}\big)\\
        &=\frac{2}{\pi} \bigg(\bbE\bigg[\frac{\big(q^2_t(\bfJ)-q^2\big)^2}{{q_t(\bfJ)}+{q}}\,\bigg|\,S\bigg]-M_2e^{-N/2}\bigg).
    \end{split}
    \end{align*}
    Since for $\bfJ \in S$, we have $q_t(\bfJ)\leq e^{2t C_S}$. This, together with  the variational representation of the variance, implies
\begin{equation*}
\bbE\bigg[\frac{\big(q^2_t(\bfJ)-q^2\big)^2}{{q_t(\bfJ)}+{q}}\,\bigg|\,S\bigg] \geq    \frac{1}{e^{tC_S}+ q}\bbE\Big[\big(q^2_t(\bfJ)-q^2\big)^2\mid S\Big] \geq \frac{1}{e^{tC_S}+ q}\text{Var}\big[q_t^2(\bfJ)\big]-4(M_2+q^2)e^{-N/2}.
\end{equation*}
The last two displays together implies
\begin{equation}
    \begin{aligned}
        \bbE\Big[\bbW_1^2(P_t^1(\bfJ),\gamma_q)\Big]\geq \frac{2}{\pi}\bigg( \frac{1}{e^{tC_S}+ q}\textup{Var}[q^2_t(\bfJ)]-8(M_2+q^2)e^{-N/2}\bigg),
    \end{aligned}    
    \label{eq: w1 opt}
\end{equation}
To conclude, we bound the variance from below. By the Cauchy-Schwarz inequality, we have
    \begin{align*}
    \sqrt{\text{Var}[q_t^2(\bfJ)]}\geq \bbE[q_t^2(\mathbf J) J_{11}]=\int_0^t \bbE\Big[\Big(e^{sN^{-1/2}\bfJ}e^{sN^{-1/2}\bfJ^\top}\Big)_{11} J_{11}\Big]\,ds.
    \end{align*}
On expanding the integrand on the right-hand side, we find
\begin{align*}
\label{eq: big exponential}
\begin{split}
    &\Big(e^{sN^{-1/2}\bfJ}e^{sN^{-1/2}\bfJ^\top}\Big)_{11}
    =1+\frac{2s}{\sqrt{N}} J_{11} + \frT_s(\bfJ), \quad s \in [0, T],
\end{split}
\end{align*}
where
\begin{align*}
    \frT_s(\bfJ) := \frac{s^2}{N}(\mathbf J \mathbf J^\top)_{11}&+\frac{2s}{N^{1/2}}\Big(\mathbf J\sum_{p=2}^\infty\frac{s^p}{N^{p/2}p!}(\mathbf J^\top)^p\Big)_{11} \\ &+ 2 \sum_{p=2}^\infty\frac{s^p}{N^{p/2}p!}(\mathbf J^p)_{11} + \sum_{p=2}^\infty \sum_{q=2}^\infty \frac{1}{N^{(p+q)/2}p!q!}\big(\mathbf J^p(\mathbf J^\top)^q\big)_{11} .
\end{align*}
It is clear from the last display that $J_{11}\frT_s$ is a polynomial in the entries of $\bfJ$ with positive coefficients and minimum degree 2. In particular, terms with odd powers do not contribute and thus we have
\[ \bbE\big[J_{11} \frT_s(\bfJ) \big]   \geq 0.\]
The last 4 displays together imply the existence of a constant $C_0(t, q) \in (0, \infty)$ which does not depend on $N$ such that
\[\frac{1}{e^{tC_S}+ q}\text{Var}[q_t^2(\bfJ)] \geq \frac{1}{N(e^{tC_S}+ q)}\bigg[\int_0^t \bbE\Big[2 sJ^2_{11}\Big]\,ds\bigg]^2 \geq \frac{C_0(t, q)}{N}. \]
By \eqref{eq: w1 opt} and the last inequality we have
\begin{equation*}
    \bbE\Big[\bbW_1^2(P_t^1(\bfJ),\gamma_q)\Big] \geq \frac{C_0(t, q)}{N}- \frac{16(M_2+q^2)}{\pi}e^{-N/2}.
\end{equation*}
Choosing $N_0(t, q) := 8\log (\frac{320M_2}{\pi C_0(t, q)})$ and $C(t, q) = C_0(t,q)/2$ yields the desired result.  
\end{proof}

\appendix

\section{Proof of Proposition \ref{prop: conv U=0}}
\label{sec: wick}
In this section we prove Proposition \ref{prop: conv U=0}. Recall the definition of $q^2_t(\bfJ) from \eqref{eq: q_t^2(bfJ)}$. 
The claim follows on establishing the existence of $(\brq_t)_{t \in [0, T]}$ such that
\begin{equation}
  \lim_{N \rightarrow \infty}\bbE[q_t^2(\bfJ)] = \brq_t^2 ,
  \label{eq: OU-covariance mean}
\end{equation}
and
\begin{equation}
      \label{eq: OU-covariance var}
\lim_{N \rightarrow \infty} \text{Var}\big[q_t^2(\bfJ)\big] = 0.
\end{equation}
The proofs of both \eqref{eq: OU-covariance mean} and \eqref{eq: OU-covariance var} follow closely the combinatorial proof of Wigner's theorem, see for instance \cite[Chapter 2]{speicher2024rmt}. Indeed, we may expand \eqref{eq: q_t^2(bfJ)} to obtain

\begin{equation}
    \bbE[q_t^2(\bfJ)] = \sum_{k = 0}^\infty \sum_{l = 0}^\infty  c_1(k,l) \frac{\bbE\big[\big(J^k(J^\top)^l\big)_{11}\big]}{N^{(k+l)/2}}.
    \label{eq: first order formula}
\end{equation}
and similarly,
\begin{equation}
        \text{Var}[q_t^2(\bfJ)] = \sum_{k = 0}^\infty \sum_{l = 0}^\infty\sum_{m =0}^\infty\sum_{n = 0}^\infty  c_2(k,l,m,n)\frac{\text{Cov}\big[\big(J^k(J^\top)^l\big)_{11},\big(J^n(J^\top)^m\big)_{11}\big]}{N^{(k+l+m+n)/2}},
            \label{eq: second order formula}
\end{equation}
where
\[c_1(k, l) = \frac{t^{k+l+1}}{k!l!(k+l+1)}, \quad c_2(k,l,m,n) = \frac{t^{n+m+k+l+1}}{k!l!n!m!(k+l+1)(n+m+1)}.  \]
Then \eqref{eq: OU-covariance mean}-\eqref{eq: OU-covariance var}, and hence Proposition \ref{prop: conv U=0}, follows from the following lemma:

\begin{lemma}
\label{lem: wick calculation}
    Let $\bfJ$ be a $N \times N$ matrix of i.i.d. standard Gaussian random variables. Then for all $k, l \in \bbN$ we have
    \begin{equation}
        \frac{\bbE\big[\big(J^k(J^\top)^l\big)_{11}\big]}{N^{(k+l)/2}} = \left \{ \begin{aligned}&\delta_{k, l} + \frac{1}{N}\eps_{k, l}  &\quad & k+ l \text{ even,} 
        \\ &0 &\quad & k+l \text{ odd},\end{aligned}\right.
        \label{eq: OU-wick first order}
    \end{equation}
    where $\delta_{k,l}$ denotes the Kronecker delta and $|\eps_{k, l}| \leq (k+l)!!$. Additionally for all $k, l, n, m \in \bbN$,
\begin{equation}\frac{\textup{Cov}\big[\big(J^k(J^\top)^l\big)_{11},\big(J^n(J^\top)^m\big)_{11}\big]}{N^{(k+l+m+n)/2}} =
\left \{ \begin{aligned}& \frac{1}{N}\eps_{k,l,n,m}  &\quad & k+ l+n+m \text{ even,} 
        \\ &0 &\quad & k+l+n+m \text{ odd},\end{aligned}\right.
         \label{eq: OU-wick second order}\end{equation}
         where $|\eps_{k,l,m,n}| \leq (k+l+m+n)!!$.
\end{lemma}

\begin{proof}[Proof of Lemma \ref{prop: conv U=0}] It suffices to establish \eqref{eq: OU-covariance mean} and \eqref{eq: OU-covariance var}. By \eqref{eq: first order formula} and \eqref{eq: OU-wick first order}, we have
\begin{equation}
        \bbE[q_t^2(\bfJ)] = \sum_{k = 0}^\infty \frac{t^{2k+1}}{(2k+1)(k!)^2} + \frac{1}{N}\sum_{k = 0}^\infty\sum_{l = 0}^\infty 1_{\{k + l \text{ even}\}} \frac{t^{k+l+1}\eps_{k,l}}{k!l!(k+l+1)}.
\end{equation}
It suffices to show that the second infinite sum converges. We may rewrite the sum in terms of $m = (k+l)/2$ and use the binomial formula to obtain
\[\begin{aligned}\sum_{k = 0}^\infty\sum_{l = 0}^\infty 1_{\{k + l \text{ even}\}} \frac{t^{k+l+1}\eps_{k, l}}{k!l!(k+l+1)} \leq \sum_{m = 0}^\infty \sum_{k = 0}^{2m} \frac{|t|^{2m+1}(2m)!!}{k!(2m-k)!(2m+1)} = \sum_{m = 0}^\infty \frac{t^{2m+1}2^{2m} (m!)}{(2m+1)!}.\end{aligned} \]
The last series is summable, and thus \eqref{eq: OU-covariance mean} holds. The proof of \eqref{eq: OU-covariance var} from \eqref{eq: wick second-order sufficient} and \eqref{eq: OU-wick second order} follows similarly.
\end{proof}

 We now prove Lemma \ref{lem: wick calculation}. The key tool is the \emph{Wick formula}. We recall a version below given in \cite[Theorem 2.8]{speicher2024rmt}.
\begin{theorem}[Wick formula] Fix $p, n \in \bbN$. Let $Y_1, \ldots, Y_p$ be independent standard Gaussian random variables and consider $X_1, \ldots, X_n \in \{Y_1, \ldots, Y_p\}$. Let $\scrP_2(n)$ denote the set of pairings of $n$ elements. Then we have
\[ \bbE[X_1 \cdots X_n] = \sum_{\pi \in \scrP_2(n)} \prod_{i = 1}^n \bbE[x_ix_{\pi_i}].\]
    \label{thm: wick}
\end{theorem}

\begin{proof}[Proof of Lemma \ref{lem: wick calculation}]
We begin with the first claim. By exchangeability and expanding the matrix powers, we have
    \begin{equation}
    \begin{aligned}
            \frac{1}{N^{\frac{k+l}{2}}}\bbE\big[\big(J^k(J^\top)^l\big)_{11}\big]
            = &  \frac{1}{N^{\frac{k+l}{2}+1}}\bbE\big[ \text{Trace}\big(J^k(J^\top)^l\big) \big]\\= &\frac{1}{N^{\frac{k+l}{2}+1}}\sum_{i = 1}^N \sum_{j = 1}^N \bbE\big[ (J^k)_{ij} (J^l)_{ij}\big] \\
            =& \frac{1}{N^{\frac{k+l}{2}+1}}  \sum_{a_1, \ldots, a_{k+1}} \sum_{b_1, \ldots, b_{l+1}} 1_{\{a_1 = b_1,\, a_{k+1} = b_{l+1}\}} \bbE\bigg[\prod_{i = 1}^k J_{a_i a_{i+1}} \prod_{j = 1}^l J_{b_j b_{j+1}}\bigg].
    \end{aligned}
 \end{equation}
 From this formula, it is clear that one only needs to consider the cases where $k+l$ is even.
 
 For the combinatorial argument which follows, it is useful to reinterpret the last line of the previous display as the sum over certain discrete functions. We introduce the formal vertex set
 \[V_{k, l} := \{v_1, \ldots, v_{k+1}\} \cup\{u_1, \ldots, u_{l+1}\}, \]
 and in the sequel we will also use the formal relations $v_{i+1} = v_i +1$ and $u_{j+1} = u_j +1$. Consider the set of functions
 \[\mS_{k,l} = \big\{ \phi : V_{k, l} \rightarrow \{1, \ldots, N\}\mid \phi(v_1) = \phi(u_1), \, \phi(v_{k+1}) = \phi(u_{l+1})\big\}. \]
 By the identification $\phi(v_i) = a_i$ and $\phi(u_j) = b_j$, we have
\[\frac{1}{N^{\frac{k+l}{2}}}\bbE\big[\big(J^k(J^\top)^l\big)_{11}\big] = \frac{1}{N^{\frac{k+l}{2}+1}} \sum_{ \phi \in \mS_{k,l}} \bbE\bigg[\prod_{i = 1}^k J_{\phi(v_i)\phi(v_{i+1})} \prod_{j = 1}^l  J_{\phi(u_j)\phi(u_{j+1})} \bigg]. \]
 
Let $V^\circ_{k, l} := V_{k, l} \backslash\{v_{k+1}, u_{l+1}\}$. By Theorem \ref{thm: wick},  for each $\phi \in \mS_{k, l}$ we have
 \begin{equation*}
 \begin{aligned}
     \bbE\bigg[\prod_{i = 1}^k J_{\phi(v_i)\phi(v_{i+1})} \prod_{j = 1}^l  J_{\phi(u_j)\phi(u_{j+1})} \bigg] &= \sum_{ \pi \in \scrP_2(V_{k, l}^\circ)} \prod_{w \in V^\circ_{k, l}} \bbE\Big[J_{\phi(w)\phi(w+1)} J_{\phi(\pi_w)\phi(\pi_w+1)} \Big] \\
     &= \sum_{ \pi \in \scrP_2(V_{k, l}^\circ)} \prod_{w \in V^\circ_{k, l}} 1_{\{\phi(w) = \phi(\pi_w)\}} 1_{\{ \phi(w+1) = \phi(\pi_w +1) \}},
 \end{aligned}
 \end{equation*}
 where the last equality follows from the fact that $\bfJ$ consists of i.i.d. Gaussian random variables.

 Combining the last two displays and swapping the order of summation, we have
\begin{equation}
    \frac{1}{N^{\frac{k+l}{2}}}\bbE\big[\big(J^k(J^\top)^l\big)_{11}\big] =\frac{1}{N^{\frac{k+l}{2}+1}} \sum_{ \pi \in \scrP_2(V_{k, l}^\circ)}\sum_{ \phi \in \mS_{k,l}}  \prod_{w \in V^\circ_{k, l}} 1_{\{\phi(w) = \phi(\pi_w)\}} 1_{\{ \phi(w+1) = \phi(\pi_w +1) \}}.
    \label{eq: wick final}
\end{equation}
For each pairing $\pi \in \scrP_2(V_{k, l}^\circ)$, we define a tri-colored graph $G^1_\pi$ with vertex set $V(G^1_\pi) = V_{k, l}$. We label the three colors by $\{\frr, \frb, \fro\}$ and each colored edge is written as $(e; \frf_e)$, where $\frf_e$ is the color of the edge $e \in V_{k,l} \times V_{k,l}$. The edge set of $G^1_\pi$ is
\[ 
\begin{aligned}
E(G^1_\pi) &:= {E_{\textup{pair}}(\pi)} \cup    {E_{\textup{shift}}(\pi)} \cup {E^1_{\textup{bdy}}}, 
\end{aligned}
\]
where 
\begin{equation}
    \begin{aligned}
    &{E_{\textup{pair}}(\pi)} := \{(w, \pi_w; \frb): w \in V^\circ_{k, l}\}, \\
&{E_{\textup{shift}}(\pi)} := \{(w+1, \pi_w+1; \frr): w \in V^\circ_{k, l}\}, \\
&{E^1_{\textup{bdy}}} := \{(v_1, u_1; \fro), (v_{k+1}, u_{l+1}; \fro)\}.
\end{aligned}
\label{eq: first-order edge set}
\end{equation}
In particular, we mark the edges in $E_{\textup{pair}}$ blue ($\frb$), the edges in $E_{\textup{shift}}$ red ($\frr$), and the edges in $E_{\textup{bdy}}$ orange ($\fro$). Below, we refer to $E_{\textup{bdy}}$ as the boundary condition. The edge set $E(G^1_\pi)$ encodes the conditions on the left hand side of \eqref{eq: wick final}. In particular, $E_{\textup{pair}}$ encodes the condition $\{\phi(w) = \phi(\pi_w)\}$, $E_{\textup{shift}}$ encodes $\{\phi(w+1) = \phi(\pi_w+1)\}$, and $E^1_{\textup{bdy}}$ encodes $\phi \in S_{k, l}$. Then for $\pi \in \scrP_2(V^\circ_{k,l})$ and $\phi:V_{k, l} \rightarrow \{1, \ldots, N\}$ we have
\begin{equation*}
    1_{\{\phi \in \mS_{k, l}\}}\prod_{w \in V^\circ_{k, l}} 1_{\{\phi(w) = \phi(\pi_w)\}}1_{\{ \phi(w+1) = \phi(\pi_w +1) \}} = \left \{ \begin{aligned}
        &1 &\quad& \phi \text{ is constant on connected components of }G^1_\pi, \\
        &0 &\quad& \text{otherwise.}
    \end{aligned}\right.
\end{equation*}
If $\delta(G^1_\pi)$ denotes the number of connected components of $G^1_\pi$, then for each $\pi$, the set of functions $\phi$ for which the above product is non-zero has cardinality $N^{\delta(G^1_\pi)}$. Therefore we have
\begin{equation}
     \frac{1}{N^{\frac{k+l}{2}}}\bbE\big[\big(J^k(J^\top)^l\big)_{11}\big] =\frac{1}{N^{\frac{k+l}{2}+1}} \sum_{ \pi \in \scrP_2(V_{k, l}^\circ)} N^{\delta(G^1_\pi)}.
\end{equation}
It is easy to see that $G^1_\pi$ is a 2-regular multi-graph, and $|V(G^1_\pi)| = |E(G^1_\pi)| = k+l+2$. Therefore the maximum number of connected components is $(k+l)/2 + 1$. The remainder term follows from the crude bound
\[ \eps_{k, l} :=\frac{1}{N^{\frac{k+l}{2}+1}} \sum_{ \pi \in \scrP_2(V_{k, l}^\circ)} N^{\delta(G^1_\pi)}1_{\{\delta(G^1_\pi) < (k+l)/2 +1 \}} \leq \frac{|\scrP_2(V_{k, l}^\circ)|}{N}, \]
and the fact that $|\scrP(V_{k, l}^\circ)| = (k+l)!!$, see e.g., \cite[Proposition 2.5]{speicher2024rmt}. Thus to establish \eqref{eq: OU-wick first order}, it suffices to show that a pairing $\pi$ achieves $\delta(G^1_\pi) = \frac{k+l}{2}+1$ if and only if $k = l$, and moreover this  pairing $\pi$ is unique. In fact we will show that the optimal pairing is given by
$ \pi^*(v_i) = u_i, \quad i = 1, \ldots, k, $
which clearly satisfies $\delta(G^1_{\pi^*}) = \frac{k+l}{2}+1$.

The above construction of $G^1_\pi$ implies that for each $\pi \in \scrP_2(V_{k, l}^\circ)$, the graph $G^1_\pi$ satisfies the following:
\begin{enumerate}
    \item Since $\pi$ is a pairing, the incident edges of each $w \in V_{k, l}$ must be of different colors.
    \item If $(w_1, w_2; \frb) \in E(G^1_\pi)$ is defined by the pairing $\pi$, then $(w_1+1, w_2+1; \frr) \in E(G^1_\pi)$. Similarly if $(w_1, w_2; \frr) \in G^1_\pi$, then $(w_1-1, w_2-1; \frb) \in E(G^1_\pi)$.
\end{enumerate}

Moreover, any optimizer $G^1_\pi$ achieving $\delta(G^1_\pi) = \frac{k+l}{2}+1$ must be composed of $\frac{k+l}{2}+1$ disjoint cycles of length $2$. This is easily seen from the fact that $G^1_\pi$ is a 2-regular multi-graph with $|V(G^1_\pi)| = |E(G^1_\pi)|$. 

We now show that an optimal graph is only possible when $k = l$. To do this, we rule out \emph{horizontal pairings} of the form $\pi(v_a) = v_b$ for some $a, b \in \{1, \ldots, k\}$, and similarly $\pi(u_a) = u_b$ for some $a, b \in \{1, \ldots, l\}$. Indeed if $k \neq l$, then by the pigeonhole principle there must be at least one horizontal pairing. 

Suppose towards contradiction that $G^1_\pi$ is an optimizer and there exists a horizontal pairing, say $\pi(v_a) = v_b$ for some $a, b \in \{1, \ldots, k\}$ with $a \neq b$. By symmetry it suffices to argue for the $v$'s, and without loss of generality we take $b > a$. Then $(v_a, v_b; \frb) \in E(G_\pi^1)$. If $a = 1$, then the boundary condition implies that $(u_1, v_1, v_b)$ forms a connected component, which contradicts optimality. Next, suppose $a \neq 1$. Since $G^1_\pi$ is optimal and property (1), there are two edges of different colors connecting these vertices. Then the presence of the edge $(v_a, v_b; \frb) \in E(G_\pi^1)$ implies the existence of $(v_a, v_b; \frr) \in E(G^1_\pi)$, and this second edge is known to be red because it is a horizontal pairing and cannot lie in $E_\textup{bdy}$. Then by property (2), $(v_{a-1}, v_{b-1}; \frb)$ is in $E(G^1_\pi)$ as well. Iterating this as well leads to the existence of the edge between $v_1$ and  $v_{b - a+1}$. Since $b > a$, the boundary edge $(v_1, u_1)$ this implies the existence of a connected component of size $3$ given by $\{u_1, v_1, v_{b - a+1}\}$.

From this, we may conclude that the only optimizer is the graph associated to $\pi^*$. Suppose that $\pi \in \scrP_2(V^\circ_{k, l})$ defines an optimizer $G^1_\pi$ and consider the pair of vertices $(v_1, u_1)$. By optimality and property (1), there are two edges of different colors connecting these vertices. By the boundary condition one of the edges must be orange, that is the edge $(v_1, u_1; \fro)$ is in $E(G_\pi^1)$. The other edge must be blue since by definition $E_{\textup{shift}}$ cannot contain an edge between $v_1$ and $u_1$. Therefore $(v_1, u_1; \frb) \in E(G_\pi^1)$ and by property (2) above, this implies then that $(v_2, u_2; \frr) \in E(G_\pi^1)$. Once again, by the optimality condition this implies that $v_2$ and $u_2$ form a disjoint cycle, and thus $(v_2, u_2; \frb) \in E(G_\pi^1)$. This argument repeats until one reaches the other boundary condition, which shows that $\pi = \pi^*$.

The second order estimate \eqref{eq: OU-covariance var} is established in much the same way. By \eqref{eq: OU-covariance mean} and the formula,
\[ \text{Cov}\big[\big(J^k(J^\top)^l\big)_{11},\big(J^n(J^\top)^m\big)_{11}\big] = \bbE\big[\big(J^k(J^\top)^l\big)_{11}\big(J^n(J^\top)^m\big)_{11}\big] - \bbE[\big(J^k(J^\top)^l\big)_{11}]\bbE[\big(J^n(J^\top)^m\big)_{11}],\]
it suffices to establish that
\begin{equation}
    \frac{1}{N^{\frac{k+l+m+n}{2}}}\bbE\big[\big(J^k(J^\top)^l\big)_{11}\big(J^n(J^\top)^m\big)_{11}\big] = \delta_{k, l} \delta_{n,m} + \frac{1}{N} \eps_{k,l,m,n},
    \label{eq: wick second-order sufficient}
\end{equation}
where $|\eps_{k,l,m,n}|\leq (k+l+m+n)!!$.
Define the analogous vertex set
\[V_{k, l, m,n} = \{v_1, \ldots, v_{k+1}\} \cup \{u_1, \ldots, u_{l+1}\} \cup \{x_1, \ldots, x_{m+1}\} \cup \{y_1, \ldots, y_{n+1}\}, \]
with
\[V_{k,l,m,n}^\circ = V_{k, l, m,n} \backslash \{v_{k+1}, u_{l+1}, x_{m+1}, y_{n+1}\}. \]
Once again, we adopt the formal notation $v_i +1 = v_{i+1}$ for vertices $i = 1, \ldots, k$, and similarly for $u, x, y$. Repeating the above Wick formula argument yields
\begin{equation}
\begin{aligned}
     &\frac{1}{N^{\frac{k+l+m+n}{2}}}\bbE\big[\big(J^k(J^\top)^l\big)_{11}\big(J^n(J^\top)^m\big)_{11}\big]\\ = &\frac{1}{N^{\frac{k+l+m+n}{2}+1}}\sum_{ \tilde \pi \in \scrP_2(V_{k,l,m,n}^\circ)} \sum_{\phi \in S_{k,l,m,n}} \prod_{w \in V_{k,l,m,n}^\circ} 1_{\{\phi(w) = \phi(\pi_w), \, \phi(w+1) = \phi(\pi_w +1) \}},
\end{aligned}
\label{eq: second order wick}
\end{equation}
where
\[ S_{k,l,m,n} = \bigg\{\phi: V_{k,l,m,n} \rightarrow \{1, \ldots, N\}\,\bigg|\,\substack{\phi(v_1) = \phi(u_1) = \phi(x_1) = \phi(y_1),\\[0.1 cm]\phi(v_{k+1}) = \phi(u_{l+1}),\, \phi(x_{m+1}) = \phi(y_{n+1})}\bigg\}\]

For $\tilde \pi \in \scrP_2(V^\circ_{k,l,m,n})$, we construct the graph $G^2_{\tilde \pi}$ in the same way we constructed $G^1_\pi$ above. That is, $V(G^2_{\tilde \pi}) = V_{k, l,m, n}$ and
\[ 
\begin{aligned}
E(G^2_{\tilde \pi}) &:= {E_{\textup{pair}}(\tilde \pi)} \cup    {E_{\textup{shift}}(\tilde \pi)} \cup {E^2_{\textup{bdy}}}, 
\end{aligned}
\]
where $E_\textup{pair}$ and $E_\textup{shift}$ are defined as in \eqref{eq: first-order edge set} and\[{E^2_{\textup{bdy}}} := \{(v_1, u_1; \fro), (x_1, y_1; \fro),  (v_{k+1}, u_{l+1}; \fro), (x_{m+1}, y_{n+1}; \fro)\}.\]
As before, the graph encodes the left hand side of \eqref{eq: second order wick}, except for the condition $\phi(u_1) = \phi(x_1)$ in the definition of $S_{k, l, m, n}$.
We may repeat the argument for the first order term with the assertions of the last paragraph to obtain
\[      \frac{1}{N^{\frac{k+l+m+n}{2}}}\bbE\big[\big(J^k(J^\top)^l\big)_{11}\big(J^n(J^\top)^m\big)_{11}\big] = \frac{1}{N^{\frac{k+l+m+n}{2}+1}}\sum_{\tilde \pi \in \scrP_2(V_{k,l,m,n}^\circ)}  N^{\delta(G^2_{\tilde \pi}) -1},\]
where the exponent is $\delta(G^2_{\tilde \pi}) -1$ due to the additional constraint $\phi(x_1) = \phi(u_1)$ imposed in the definition of $S_{k,l,m,n}$ which is not encoded in the graph $G^2_{\tilde \pi}$. As before, $G^2_{\tilde \pi}$ is a 2-regular multi-graph with $|V(G^2_{\tilde \pi})| = |E(G^2_{\tilde \pi})|$, and thus any optimizer of $\delta$ must consist only of disjoint cycles of length 2. Moreover, the graph $G^2_{\tilde \pi}$ clearly satisfies properties (1) and (2) above. Thus there is a unique $\pi$ satisfying ${\delta(G^2_{\tilde \pi})} = \frac{k+l+m+n}{2} +2 $ if and only if $k = l$ and $n = m$, and such a $\pi$ is unique. Indeed by taking $\eps_{k,l,m,n}$ once again to be the lower order terms in the last display, the argument above may be repeated verbatim to establish this. This together with the fact that $|\scrP(V_{k, l,m,n}^\circ)| = (k+l+m+n)!!$ implies \eqref{eq: wick second-order sufficient}, and hence \eqref{eq: OU-wick second order}. \end{proof}

\section{Volterra kernels}
\label{sec: Volterra}
\begin{lemma}
    Let $\nu \in \mP(\mC_T^1)$ and $H_\nu^t$ be the operator defined in \eqref{eq: H_mu^t}. Then we have
    \begin{align}
\begin{split}
    \label{eq: Fredholm resolvent}
    K_\nu(s, u)- H^t_\nu(s,u)= \beta^2 \int_0^t K_\nu(s,r)H^t_\nu(r,u)\,dr
    =\beta^2 \int_0^t H^t_\nu(s,r) K_\nu(r,u)\,dr,\end{split}
\end{align}
for almost every $u, s, t \in [0, T]$. Moreover, the kernel $H_\nu^t$ extends to a continuous function on $[0, t]^2$, and we have
\begin{equation}
\label{eq: H bound}
        \sup_{t \in [0, T]} \sup_{u, s \in [0, t]} H_\nu^t(u, s)  \leq A^2.
\end{equation}
\end{lemma}

\begin{proof}
    Note that by \eqref{eq: H_mu^t} that $H^t_\mu$ commutes with $K^t_\nu$, and in particular we have \[H_\nu^t + \beta^2 K_\nu^t H_\nu^t = H_\nu^t + \beta^2 H_\nu^t K_\nu^t = K_\nu^t.\] Thus for every smooth function $\varphi:[0, t] \rightarrow \R$ we have
    \[ \int_0^t H_\nu^t(s,r)\varphi(r)\,dr+\beta^2 \int_0^t \int_0^t K_\nu(s,r)H_\nu^t(r,x)\varphi(x)\,dr\,dx=\int_0^t K_\nu(s,r)\varphi(r)\,dr,\]
    and thus \eqref{eq: Fredholm resolvent} holds. Since $K_\nu$ is continuous, by \eqref{eq: Fredholm resolvent} $ H^t_\nu(s,u)$ is equal to a continuous function almost everywhere and therefore admits a continuous extension. The bound \eqref{eq: H bound} is in \cite[Lemma A.5]{Arous1995-zj}.

\end{proof}
\begin{proof}[Proof of Lemma \ref{lem: resolvent}] By \eqref{eq: H bound}, there exists a decomposition of $[0, T]$ into a finite number of intervals $(I_i)_{i = 1}^\ell$ such that we have
\[\sup_{i \in\{ 1, \ldots, \ell\}}\int_{I_i} \int_{I_i} |H_\mu^t(t, s)1_{[0, t]}(s)|^2 ds dt < 1.\]
Thus by \cite[Corollary 3.14]{VolterraEquations} that $H^t(t, s)1_{[0, t]}(s)$ admits a Volterra resolvent kernel $R_\mu \in L^2([0, T]^2)$ which satisfies \eqref{eq: H resolvent}.
Next, we prove \eqref{eq: R contiuity}. Applying Gr\"{o}nwall's inequality to \eqref{eq: H resolvent} and \eqref{eq: H bound}, we have \eqref{eq: R bound}. 

We now prove the stability estimate \ref{eq: R contiuity}. By \eqref{eq: H resolvent}, \eqref{eq: H bound}, and \eqref{eq: R bound}, there exists $C$ depending only on $(A, \beta, T)$ such that
\begin{equation*}
\begin{aligned}
&|R_\mu(t, s) - R_\nu(t, s)|^2 \\
 = &\bigg|H_\mu(t, s) - H_\nu(t, s) +  \beta^2 \int_s^t H_\mu(t, u)R_\mu(u, s) du - \beta^2 \int_s^t H_\nu(t, u)R_\nu(u, s) du    \bigg|^2 \\
 \leq &C \bigg(|H_\mu(t, s) - H_\nu(t, s)|^2 + \int_s^t|H_\mu(t, u) - H_\nu(t, u)|^2 du + \int_s^t |R_\mu(u, s) -R_\nu(u, s)|^2du \bigg).
\end{aligned}    
\end{equation*}
Integrate the last display and apply Fubini's theorem to get
\[\begin{aligned}
\int_0^t|R_\mu(t, s) - R_\nu(t, s)|^2 ds \leq C \bigg(\int_0^t|H_\mu(t, u) - H_\nu(t, u)|^2 du + \int_0^t \bigg[\int_0^u|R_\mu(u, s) - R_\nu(u, s)|^2 ds\bigg] du \bigg)
\end{aligned}\]
By Gr\"{o}nwall's inequality, we have
\begin{align*}
    \int_0^t|R_\mu(t, s) - R_\nu(t, s)|^2 ds &\leq C \int_0^t|H_\mu(t, u) - H_\nu(t, u)|^2 du+C\int_0^t\int_0^s |H_\mu(s, u) - H_\nu(s, u)|^2 du\,ds.\\
\end{align*}
Integrating again in $dt$ yields
\begin{align*}
    \int_0^T\int_0^t|R_\mu(t, s) - R_\nu(t, s)|^2 ds &\leq C \int_0^T\int_0^t|H_\mu(t, s) - H_\nu(t, s)|^2 ds\,dt\\
\end{align*}
To conclude, by \eqref{eq: Fredholm resolvent} we have
    \begin{align*}
        H^t_\mu(t,s)-H^t_\nu(t,s) +&\beta^2 \int_0^t K_\mu(u,s)\cdot(H^t_\mu(t,u)-H^t_\nu(t,u))\,du\\
        &= K_\mu(t,s)-K_\nu(t,s)+\beta^2 \int_0^t H_\nu(t,u)\cdot(K_\mu(u,s)-K_\nu(u,s))\,du.
    \end{align*}
    Let $f_t(s)$ be the right hand side as a function of $s$. By symmetry of $K^t_\mu$, the left hand side is equal to
    \begin{align*}
        (I+\beta^2K_\mu^t) (H^t_\mu(t,\cdot)-H^t_\nu(t,\cdot))(s),
    \end{align*}
    and we have $(H^t_\mu(t,\cdot)-H^t_\nu(t,\cdot))(s)=(I+\beta^2\overline K_\mu^t)^{-1}f_t(s)$. Since $K^t_\mu$ is positive semi-definite, we get
    \begin{align*}
        \|H^t_\mu(t,\cdot)-&H^t_\nu(t,\cdot)\|_{L^2(0,t)}^2\leq \|(I+\beta^2\overline K_\mu^t)^{-1}f_t(\cdot)\|_{L^2(0,t)}^2\leq \|(I+\beta^2\overline K^t_\mu)^{-1}\|^2_{\text{op}}\|f_t(\cdot)\|^2_{L^2(0,t)}\\
        &\leq \int_0^t \Big|K_\mu(t,s)-K_\nu(t,s)+\beta^2 \int_0^t H_\nu(t,u)\cdot(K_\mu(u,s)-K_\nu(u,s))\,du \Big|^2\,ds\\
        &\leq2 \int_0^t |K_\mu(t,s)-K_\nu(t,s)|^2\,ds+2\beta^4tA^4 \int_0^t\int_0^t |K_\mu(u,s)-K_\nu(u,s)|^2\,du\,ds
    \end{align*}
    Integrating once again yields
\[{\int_0^T}\int_0^t|H_\mu(t, u) - H_\nu(t, u)|^2 \,du\,dt \leq C {\int_0^T}\int_0^T|K_\mu(t, u) - K_\nu(t, u)|^2 \,dt \,du. \]
\end{proof}

\begin{proof}[Proof of Lemma \ref{lem: N kernels}]
    We check that $H^t_{\hat\mu^N}$ and $R_{\hat\mu^N}$ satisfy the resolvent equations. For $H_{\hat\mu^N}$ we get
    \begin{align*}
        \beta^2 \int_0^t H_{\hat\mu^N}(r,s) K_{\hat\mu^N}(s,u)\,ds&=\frac{\beta^2}{N^2} \int_0^t X_r^{\top}Q_t X_sX_s^\top X_u\,ds\\
        &= \frac{1}{N} X_r^\top \Big(I+\frac{\beta^2}{N}\int_0^t X_s X_s^\top \Big)^{-1} \frac{\beta^2}{N}\int_0^t X_s^\top X_s\,ds X_u\\
        &= \frac{1}{N} X_r^\top X_u-\frac{1}{N}X_r^\top Q_t X_s= K_{\hat\mu^N}(r,s)-H_{\hat\mu^N}(r,u).
    \end{align*}
    On the other hand, 
    \begin{align*}
        -\beta^2 \int_s^t R_{\hat\mu^N}(t,u) H_{\hat\mu^N}(u,s)\,du&=\frac{\beta^2}{n^2}\int_s^t X_t^{\top} Q_u X_u X_u^{\top} Q_u X_s\,du\\
        &=-\frac{1}{N} X_t^\top \frac{\beta^2}{n}\int_s^t Q_u X_u X_u^\top Q_u\,du X_s\\
        &=\frac{1}{N} X_t^\top \int_s^t \frac{d}{du}Q_u\,du X_s\\
        &=\frac{1}{N}X_t^\top Q_t X_s-\frac{1}{N}X_tQ_sX_s\\
        &= H_{\hat\mu^N}(t,s)-R_{\hat\mu^N}(t,s).
    \end{align*}
\end{proof}

\bibliographystyle{abbrv}
\bibliography{Biblio}

\end{document}